\documentclass{article}
\usepackage[utf8]{inputenc}
\usepackage{url}
\usepackage[colorlinks,allcolors=black]{hyperref}
\hypersetup{
     colorlinks = true,
     citecolor = black
}
\usepackage[utf8]{inputenc}
\usepackage{amssymb}
\usepackage{mathrsfs}
\usepackage{amsthm}
\usepackage{amsmath}
\usepackage{hyperref}
\usepackage{cleveref}
\usepackage[french,USenglish]{babel}
\usepackage{float}
\usepackage{color}
\usepackage{graphicx}
\usepackage[explicit]{titlesec}
\usepackage{sectsty}
\usepackage{pict2e}
\usepackage[T1]{fontenc}
\usepackage{filecontents}
\usepackage{mathtools}
\usepackage{lipsum}
\usepackage{dsfont}
\usepackage{stmaryrd}
\usepackage{setspace}
\usepackage{tikz}
\usepackage{subfig}

\newtheorem{theorem}{Theorem}[section]
\newtheorem{corollary}[theorem]{Corollary}
\newtheorem{proposition}[theorem]{Proposition}
\newtheorem{lemma}[theorem]{Lemma}
\newtheorem{definition}[theorem]{Definition}
\newtheorem{example}[theorem]{Example}
\newtheorem{remark}[theorem]{Remark}

\usepackage{geometry}
\newcommand{\N}{\mathbb{N}}
\newcommand{\R}{\mathbb{R}}
\renewcommand{\P}{\mathbb{P}}

\title{Non-decreasing martingale couplings}

\author{Benjamin Jourdain\thanks{CERMICS, Ecole des Ponts, INRIA, Marne-la-Vallée, France. E-mail: benjamin.jourdain@enpc.fr - This research benefited from the support of the \textquotedblleft Chaire Risques Financiers\textquotedblright , Fondation du Risque.}
\and 
Kexin Shao\thanks{
INRIA Paris,  2 rue Simone Iff, CS 42112, 75589 Paris Cedex 12, France, Universit\'e Paris-Dauphine, Ecole des Ponts ParisTech. E-mail: kexin.shao@inria.fr. This project has received funding from the European Union’s Horizon 2020 research and innovation
programme under the Marie Sk\l{}odowska-Curie grant agreement No 945322.
}}
\date{\today}

\begin{document}

\maketitle
\begin{abstract}
 For many examples of couples $(\mu,\nu)$ of probability measures on the real line in the convex order, we observe numerically that the Hobson and Neuberger martingale coupling, which maximizes for $\rho=1$ the integral of $|y-x|^\rho$ with respect to any martingale coupling between $\mu$ and $\nu$, is still a maximizer for $\rho\in(0,2)$ and a minimizer for $\rho>2$. We investigate the theoretical validity of this numerical observation and give rather restrictive sufficient conditions for the property to hold. We also exhibit couples $(\mu,\nu)$ such that it does not hold. The support of the Hobson and Neuberger coupling is known to satisfy some monotonicity property which we call non-decreasing. We check that the non-decreasing property is preserved for maximizers when $\rho\in(0,1]$. 
 In general, there exist distinct non-decreasing martingale couplings, and we find some decomposition of $\nu$ which is in one-to-one correspondence with martingale couplings non-decreasing in a generalized sense.
\end{abstract}
\section{Introduction}

In this paper,  for $\mu,\nu$ in the set ${\cal P}_1(\R)$ of probability measures on the real line with a finite first order moment, we are interested in martingale couplings between $\mu$ and $\nu$ that attain $\overline{\mathcal M}_\rho(\mu,\nu)$ and $\underline{\mathcal M}_\rho(\mu,\nu)$ defined  for $\rho>0$ by 
 \[
   \overline{\mathcal M}^\rho_\rho(\mu,\nu)=\sup_{\pi\in\Pi_{\mathrm{M}}(\mu,\nu)}\int_{\R\times\R}\vert x-y\vert^\rho\,\pi(dx,dy)\mbox{ and }\underline{\mathcal M}^\rho_\rho(\mu,\nu)=\inf_{\pi\in\Pi_{\mathrm{M}}(\mu,\nu)}\int_{\R\times\R}\vert x-y\vert^\rho\,\pi(dx,dy).\]
 Here $\Pi_M(\mu,\nu)$ denotes the subset of $$\Pi(\mu,\nu)=\left\{\pi\mbox{ probability measure on $\R\times\R$} \mid \int_{y\in\R} \pi(dx,dy) = \mu(dx), \int_{x\in\R} \pi(dx,dy) = \nu(dy) \right\}$$ consisting in martingale couplings :
 $$\Pi_M(\mu,\nu)=\left\{\pi(dx,dy)=\mu(dx)\pi_x(dy)\in\Pi(\mu,\nu)\mid\mu(dx)\text{-a.e.},\ \int_{y\in\R}y\,\pi_x(dy)=x\right\}.$$
By Strassen's theorem \cite{St65},
\[
\Pi_{ M}(\mu,\nu)\neq\emptyset\quad \iff \quad \mu\leq_{cx}\nu,
\]
where $\mu\leq_{cx}\nu$ means that $\mu$ is smaller than $\nu$ in the convex order :
$$\forall \varphi:\R\to\R\mbox{ convex},\;\int_{\R}\varphi(x)\,\mu(dx)\le\int_{\R}\varphi(y)\,\nu(dy).$$
When $\mu,\nu\in{\cal P}_1(\R)$ are such that $\mu\le_{cx}\nu$, then $\underline{\mathcal M}_\rho(\mu,\nu)\le \overline{\mathcal M}_\rho(\mu,\nu)<+\infty$ for $\rho\in (0,1]$. For $\rho>1$, the finiteness remains true when $\mu,\nu$ belong to the set ${\cal P}_\rho(\R)$ of probability measures on the real line with a finite moment of order $\rho$. For $\rho\ge 1$, a martingale Wasserstein inequality is investigated in \cite{JoMa22}: it is proved that there exists a finite constant $K_\rho$ such that for $\mu,\nu\in{\cal P}_\rho(\R)\mbox{ with }\mu\le_{cx}\nu$,
$$\underline{\mathcal M}^\rho_\rho(\mu,\nu)\le K_\rho\left(\inf_{\pi\in\Pi(\mu,\nu)}\int_{\R\times\R}\vert x-y\vert^\rho\,\pi(dx,dy)\right)^{\frac 1\rho}\min_{z\in\R}\left(\int_\R\vert z-y\vert^\rho\nu(dy)\right)^{\frac{\rho-1}\rho}.$$

The couplings attaining $\overline{\mathcal M}_1(\mu,\nu)$ and $\underline{\mathcal M}_1(\mu,\nu)$ were first investigated in the literature. Motivated by the robust pricing and hedging of forward start straddle options, Hobson and Neuberger state in Theorem 8.2 \cite{HobsonNeuberger} that for the cost function $|x-y|$, there exists a maximizing martingale coupling with the form $$\pi^{\rm HN}=\int_0^1\left(\frac{r(u)-q(u)}{r(u)-p(u)}\delta_{(q(u),p(u))}+\frac{q(u)-p(u)}{r(u)-p(u)}\delta_{(q(u),r(u))}\right)du$$ with $p,q,r$ non-decreasing on $(0,1)$ and such that $p\le q\le r$ (by convention the integrand is equal to $\delta_{(q(u),q(u))}$ when $p(u)=q(u)=r(u)$).  This provides an example of a martingale coupling non-decreasing in the sense of Definition \ref{def:NdCoupling} that we give below. The study of such non-decreasing martingale couplings is one of main contributions of the present paper. The necessary optimality criterion given by Beiglb\"ock and Juillet in Lemma 1.11 \cite{BeJu16} ensures that any maximizing martingale coupling is non-decreasing. When $\mu$ does not weight points, we have the equality $$\pi^{\rm HN}=\int_\R\left(\frac{g(x)-x}{g(x)-f(x)}\delta_{(x,f(x))}+\frac{x-f(x)}{g(x)-f(x)}\delta_{(x,g(x))}\right)\mu(dx)$$ for non-decreasing functions $f$ and $g$ such that $\forall x\in\R$, $f(x)\le x\le g(x)$ and, according to Theorem 7.3 \cite{BeJu16}, uniqueness of maximizing martingale couplings holds. Under the dispersion assumption that there exists a finite interval $I$ such that $(\nu-\mu)^+(I)=0$ and $(\mu-\nu)^+(I)=(\mu-\nu)^+(\R)$, Hobson and Klimmek state in Theorem \cite{HoKl} that for the cost function $|x-y|$, there is a minimizing martingale coupling of the form
$$\pi^{\rm HK}=\int_\R\delta_{(z,z)}\mu\wedge\nu(dz)+(1-\mu\wedge\nu(\R))\int_0^1\left(\frac{r(u)-q(u)}{r(u)-p(u)}\delta_{(q(u),p(u))}+\frac{q(u)-p(u)}{r(u)-p(u)}\delta_{(q(u),r(u))}\right)du$$ for a non-decreasing function $q$ and non-increasing functions $p,r$ such that $p\le q\le r$. When $\mu\neq \nu$ the integral over $u$ provides an example of a coupling in $\Pi_M\left(\frac{(\mu-\nu)^+}{1-\mu\wedge\nu(\R)},\frac{(\nu-\mu)^+}{1-\mu\wedge\nu(\R)}\right)$ non-increasing in the sense of Definition \ref{def:NdCoupling} below. Without the dispersion assumption, according to Theorem 7.4 \cite{BeJu16}, there is a unique minimizing coupling when $\mu$ does not weight points and this coupling writes $$\int_\R\delta_{(z,z)}\mu\wedge\nu(dz)+\int_\R\left(\frac{g(x)-x}{g(x)-f(x)}\delta_{(x,f(x))}+\frac{x-f(x)}{g(x)-f(x)}\delta_{(x,g(x))}\right)(\mu-\nu)^+(dx)$$ with $f$ and $g$ such that $\forall x\in\R$, $f(x)\le x\le g(x)$.

For many choices of $\mu,\nu\in{\cal P}_{\rho\vee 1}(\R)$ such that $\mu\le_{cx}\nu$, when solving the linear programming problem for the cost function $|x-y|^\rho$ obtained by approximating $\mu$ and $\nu$ by finitely supported probability measures still in the convex order, it turns out that the coupling $\pi^{\rm HN}$ maximizing when $\rho=1$ still maximizes when $\rho\in(0,2)$ and minimizes when $\rho>2$: see the numerical results in Section \ref{sec:numres}. It is not surprising that $2$ appears as a threshold for the power $\rho$ since the martingale property ensures that when $\mu,\nu\in{\cal P}_2(\R)$,
$$\forall \pi\in\Pi_M(\mu,\nu),\;\int_{\R^2}|x-y|^2\pi(dx,dy)=\int_\R y^2\nu(dy)-\int_\R x^2\mu(dx)=\underline{\mathcal M}^2_2(\mu,\nu)=\overline{\mathcal M}^2_2(\mu,\nu).$$
With the complete lattice structure of $\{\eta\in{\cal P}_1(\R):\int_{\R}z\eta(dz)=\int_\R y^2\nu(dy)-\int_\R x^2\mu(dx)\}$ for the convex order stated in \cite{kertzrosler}, we deduce in Proposition \ref{prop:infsupconv} below that $\left\{{\rm sq}\# \pi:\pi\in\Pi_{M}(\mu,\nu)\right\}$ where ${\rm sq}\# \pi$ denotes the image of the measure $\pi$ by the square cost function ${\rm sq}(x,y)=(y-x)^2$ admits an infimum and a supremum for this order. 
When there exists $\underline{\pi} \in\Pi_{M}(\mu,\nu)$ such that ${\rm sq}\# \underline \pi$ is equal to the infimum, then $\int_{\R^2}|x-y|^\rho\underline \pi(dx,dy)$ is equal to $\overline{\mathcal M}^\rho_\rho(\mu,\nu)$ when $\rho\in (0,2)$ and to $\underline{\mathcal M}^\rho_\rho(\mu,\nu)$ when $\rho>2$. In particular, when $\nu(dy)=\int_{x\in\R}\frac{1}{2}\left(\delta_{x-a}(dy)+\delta_{x+a}(dy)\right)\mu(dx)$ for some $a\in\R$, the infimum is equal to $\delta_{a^2}$ and is attained with $\underline \pi(dx,dy)=\frac{1}{2}\left(\delta_{x-a}(dy)+\delta_{x+a}(dy)\right)\mu(dx)$. When there exists $\overline{\pi} \in\Pi_{M}(\mu,\nu)$ such that ${\rm sq}\# \pi$ is equal to the supremum, then $\int_{\R^2}|x-y|^\rho\overline \pi(dx,dy)$ is equal to $\underline{\mathcal M}^\rho_\rho(\mu,\nu)$ when $\rho\in (0,2)$ and to $
\overline{\mathcal M}^\rho_\rho(\mu,\nu)$ when $\rho>2$.  When $\mu$ does not weight points, the uniqueness of the maximizing and minimizing couplings for $\rho=1$ stated in Theorems 7.3 and 7.4 \cite{BeJu16} implies that $\pi^{\rm HN}$ (resp. $\pi^{\rm HK}$) is the only possible value for $\underline\pi$ (resp. $\overline\pi$). Unfortunately, even if we are able to exhibit in Proposition \ref{prop:M^HK leq_cx} a simple situation where the infimum and supremum are attained, we also provide examples where they are not attained (see Propositions \ref{prop:diffcouplcr} and \ref{prop:M*HN} where some of the above consequences of attainment do not hold).
That is why we directly study couplings that attain $\overline{\mathcal M}^\rho_\rho(\mu,\nu)$ and $\underline{\mathcal M}^\rho_\rho(\mu,\nu)$. It turns out that, the non-decreasing property of maximizing couplings as well as their uniqueness when $\mu$ does not weight points can be extended from the specific value $\rho=1$ to the case $\rho\in (0,1]$. But the optimizers may be distinct for distinct values of $\rho$. This motivates our investigation of non-decreasing martingale couplings and we find the decomposition of $\nu$ into $(\int_{x\in\R}\mathds{1}_{\{y<x\}}\pi(dx,dy),\int_{x\in\R}\mathds{1}_{\{y=x\}}\pi(dx,dy),\int_{x\in\R}\mathds{1}_{\{y>x\}}\pi(dx,dy))$ for $\pi\in\Pi_M(\mu,\nu)$ is in one to one correspondence with martingale couplings in $\Pi_M(\mu,\nu)$ non-decreasing in a generalized sense. We also show that the existence of a non-increasing martingale coupling is equivalent to a restrictive nested supports condition between $\mu$ and $\nu$, under which there is a unique non-increasing and a unique non-decreasing martingale couplings. Even under the nested supports condition, the image ${\rm sq}\# \pi^{\rm HN}$ of the unique non-decreasing martingale coupling $\pi^{\rm HN}$ is not necessarily the infimum of $\left\{{\rm sq}\# \pi:\pi\in\Pi_{\mathrm{M}}(\mu,\nu)\right\}$ for the convex order. But under a more and more restrictive supports condition as $\rho$ grows, this coupling $\pi^{\rm HN}$ attains $\overline{\mathcal M}^\rho_\rho(\mu,\nu)$ when $\rho\in (1,2)$ and $\underline{\mathcal M}^\rho_\rho(\mu,\nu)$ when $\rho>2$.

Our results are stated in Section \ref{sec:monotonic} while their proofs are given in Section \ref{sec:proofs}.

The proofs rely on several properties of the quantile function, which we list below. We denote the cumulative distribution function of $\eta\in\mathcal P(\R)$ by $F_\eta: \R\ni x\mapsto \eta(-\infty,x]\in[0,1]$, and its quantile function by $F_\eta^{-1}: (0,1)\ni u\mapsto \inf\left\{x\in\R: u \leq F_\eta(x)\right\}\in \R$. The following properties hold : 
\begin{itemize}
    \item[(a)] $F_\eta$ is right-continuous with left-hand limits, $ F_\eta^{-1}$ is left-continuous with right-hand limits;
    \item[(b)] For all $u\in(0,1)$ and $x\in\R$,
        \begin{equation}\label{eq:qt flip}
          F_\eta^{-1}(u) \leq x  \quad \iff \quad u \leq F_\eta(x) ,
        \end{equation}
        denoting $F_\eta(y-)$ the left-hand limit of $F_\eta$ at $y\in\R$,
        \begin{equation}\label{eq:qt ineq flip}
              F_\eta(x-)< u \leq F_\eta(x) \implies  x = F_\eta^{-1}(u) \quad\mbox{and}\quad F_\eta(F_\eta^{-1}(u)-)\leq u \leq F_\eta(F_\eta^{-1}(u));
        \end{equation}
    \item[(c)] For $\eta(dx)$-almost every $x\in\R$,
        \begin{equation}\label{eq:F-1F=x}
            F_\eta(x)>0,\,  F_\eta(x-)<1, \quad\mbox{and}\quad F_\eta^{-1}(F_\eta(x)) = x;
        \end{equation}
    \item[(d)] By the inverse transform sampling, for $f:\R\to\R$ measurable and bounded,
        \begin{equation}\label{eq:ITS}
            \int_\R f(x) \,\eta(dx) = \int_0^1 f(F_\eta^{-1}(v)) \,dv
        \end{equation}
      \end{itemize}
      We also denote by $u_\eta:\R\ni x\mapsto\int_\R|y-x|\eta(dy)\in\R_+$ the potential function of $\eta\in{\cal P}_1(\R)$. For $\mu,\nu\in{\cal P}_1(\R)$, we have
      $$\mu\le_{cx}\nu\Longleftrightarrow\forall x\in\R,\;u_\mu(x)\le u_\nu(x).$$
We say that $\mu$ is smaller than $\nu$ in the stochastic order and denote $\mu\le_{st}\nu$ if there exists $\pi\in\Pi(\mu,\nu)$ such that $\pi(\{(x,y)\in\R^2:x\le y\})=1$. We have $\mu\le_{st}\nu\Longleftrightarrow F_\mu\ge F_\nu\Longleftrightarrow F_\mu^{-1}\le F_\nu^{-1}$.

      \section{Main results}\label{sec:monotonic}
 \subsection{Numerical experiments}\label{sec:numres}
For $\mu,\nu$ such that $\mu\le_{cx}\nu$, we consider $(X_i)_{1\le i\le I}$ (resp. $(Y_j)_{1\le j\le J}$) independent and identically distributed according to $\mu$ (resp. $\nu$). The empirical measures \[
\hat \mu_I = \frac1I\sum_{i=1}^I \delta_{X_i - \bar X_I}\mbox{ with }\bar X_I=\frac 1I\sum_{i=1}^IX_i \quad\mbox{and}\quad 
\hat \nu_J = \frac1J\sum_{j=1}^J  \delta_{Y_j - \bar Y_J}\mbox{ with }\bar Y_J=\frac 1J\sum_{j=1}^JY_j,
\]
are both centred and respectively approximate the respective images $\hat\mu$ and $\hat\nu$ of $\mu$ and $\nu$ by $x\mapsto x-\int_\R y\nu(dy)$.
Applying \cite[Algorithm 1]{AlCoJo17a}, we compute $\hat \mu_I \vee_{cx} \hat \nu_J$. The computation of $\overline{\mathcal{M}}_\rho(\hat \mu_I,\hat \mu_I \vee_{cx} \hat \nu_J)$ (resp. $\underline{\mathcal{M}}_\rho(\hat \mu_I,\hat \mu_I \vee_{cx} \hat \nu_J)$) is a linear programming problem that we solve using the CVXPY package \cite{diamond2016cvxpy} in python for $I=J=100$. We call $\pi^{\rm HN}$ the obtained maximizer for the cost $|y-x|$. When $\rho\in (0,1)\cup(1,2)$ (resp. $\rho>2$), we compare the optimal value $\overline{\mathcal{M}}_\rho(\hat \mu_I,\hat \mu_I \vee_{cx} \hat \nu_J)$ (resp. $\underline{\mathcal{M}}_\rho(\hat \mu_I,\hat \mu_I \vee_{cx} \hat \nu_J)$) and corresponding maximizer $\overline \pi$ (resp. minimizer $\underline \pi$) computed by the solver with $\mathcal{I}_\rho^{\,\pi^{\rm HN}} :=\int \vert x-y\vert^\rho \pi^{\rm HN}(dx,dy)$ and $\pi^{\rm HN}$ respectively. It turns out that the values always coincide while the couplings $\overline\pi$ and $\underline\pi$ are distinct from $\pi^{\rm HN}$ when $\mu$ and $\nu$ are continuous distributions. 


\begin{center}
\begin{tabular}{||c | @{\hspace{0.4cm}} c @{\hspace{0.4cm}}| @{\hspace{0.4cm}} c @{\hspace{0.4cm}} | c || c | @{\hspace{0.4cm}} c @{\hspace{0.4cm}} | @{\hspace{0.4cm}} c @{\hspace{0.4cm}} | c ||}
    \hline
     \multicolumn{8}{|c|}{Normal distributions : $\mu= \mathcal {N}(0, 0.24)$ and  $\nu=\mathcal {N}(0, 0.28)$} \\
     \hline\rule{0pt}{3ex}
     $\rho$ &    $\overline{\mathcal{M}}_\rho$ &  $\mathcal{I}_\rho^{\,\pi^{\rm HN}}$ & $\lVert \overline \pi - \pi^{HN}\rVert$
     &$\rho$ &    $\underline{\mathcal{M}}_\rho$  & $\mathcal{I}_\rho^{\,\pi^{\rm HN}}$ & $\lVert \underline \pi - \pi^{HN}\rVert$ \\ [0.5ex]
     \hline\rule{0pt}{2.5ex}
        0.3 & 0.586396 &      0.586391 &     0.022987 &    2.1 & 0.025338 &      0.025338 &     0.027715 \\
        0.7 & 0.289132 &      0.289131 &     0.012609 &    2.3 & 0.017985 &      0.017987 &     0.034845 \\
        1.0 & 0.170704 &      0.170704 &     0.0      &    2.5 & 0.012783 &      0.012785 &     0.037657 \\
        1.4 & 0.084928 &      0.084927 &     0.021816 &    3.0 & 0.005473 &      0.005477 &     0.044496 \\
        1.9 & 0.035740 &      0.035740 &     0.026054 &    5.0 & 0.000199 &      0.000202 &     0.055856 \\
     \hline
\end{tabular}
\end{center}
\begin{center}
\begin{tabular}{||c | @{\hspace{0.4cm}} c @{\hspace{0.4cm}}| @{\hspace{0.4cm}} c @{\hspace{0.4cm}} | c || c | @{\hspace{0.4cm}} c @{\hspace{0.4cm}} | @{\hspace{0.4cm}} c @{\hspace{0.4cm}} | c ||}
    \hline
     \multicolumn{8}{|c|}{Log-normal distributions : $\mu=\exp\#\mathcal{N}(0, 0.24)$ and $\nu=\exp\#\mathcal{N}(-0.0104, 0.28)$} \\
     \hline\rule{0pt}{3ex}
     $\rho$ &    $\overline{\mathcal{M}}_\rho$ &  $\mathcal{I}_\rho^{\,\pi^{\rm HN}}$ & $\lVert \overline \pi - \pi^{HN}\rVert$
     &$\rho$ &    $\underline{\mathcal{M}}_\rho$  & $\mathcal{I}_\rho^{\,\pi^{\rm HN}}$ & $\lVert \underline \pi - \pi^{HN}\rVert$ \\ [0.5ex]
     \hline\rule{0pt}{2.5ex}
        0.3 & 0.607721 &      0.607720 &     0.017452 &    2.1 & 0.036152 &      0.036153 &     0.030068 \\
        0.7 & 0.316957 &      0.316957 &     0.008263 &    2.3 & 0.026842 &      0.026843 &     0.033392 \\
        1.0 & 0.196361 &      0.196361 &     0.0      &    2.5 & 0.019985 &      0.019988 &     0.034966 \\
        1.4 & 0.104947 &      0.104946 &     0.020343 &    3.0 & 0.009673 &      0.009678 &     0.041519 \\
        1.9 & 0.048834 &      0.048833 &     0.027026 &    5.0 & 0.000614 &      0.000619 &     0.062155 \\
     \hline
\end{tabular}
\end{center}
\begin{center}
\begin{tabular}{||c | @{\hspace{0.4cm}} c @{\hspace{0.4cm}}| @{\hspace{0.4cm}} c @{\hspace{0.4cm}} | c || c | @{\hspace{0.4cm}} c @{\hspace{0.4cm}} | @{\hspace{0.4cm}} c @{\hspace{0.4cm}} | c ||}
    \hline
     \multicolumn{8}{|c|}{{Exponential distributions : $\mu={\cal L}(Y-1)$ $\nu={\cal L}(X-2)$ with $X\sim{\cal E}(1)$ and $Y\sim{\cal E}(0.5)$}} \\
     \hline
      $\rho$ &    $\overline{\mathcal{M}}_\rho$ &   $\mathcal{I}_\rho^{\,\pi^{\rm HN}}$ & $\lVert \overline \pi - \pi^{HN}\rVert$
     &$\rho$ &    $\underline{\mathcal{M}}_\rho$  & $\mathcal{I}_\rho^{\,\pi^{\rm HN}}$ & $\lVert \underline \pi - \pi^{HN}\rVert$ \\ [0.5ex]
     \hline\rule{0pt}{2.5ex}
        0.3 & 0.768373 &      0.768293 &     0.039204 &    2.1 & 0.196164 &      0.196179 &     0.050821 \\
        0.7 & 0.550068 &      0.550059 &     0.023933 &    2.3 & 0.171925 &      0.171980 &     0.060709 \\
        1.0 & 0.433394 &      0.433394 &     0.0      &    2.5 & 0.151178 &      0.151284 &     0.062783 \\
        1.4 & 0.320305 &      0.320281 &     0.032152 &    3.0 & 0.111067 &      0.111337 &     0.067680 \\
        1.9 & 0.224602 &      0.224591 &     0.043663 &    5.0 & 0.037085 &      0.038208 &     0.078178 \\
     \hline
\end{tabular}
\end{center}
\begin{center}
\begin{tabular}{||c | @{\hspace{0.4cm}} c @{\hspace{0.4cm}} | @{\hspace{0.4cm}} c @{\hspace{0.4cm}} | c || c | @{\hspace{0.25cm}} c @{\hspace{0.25cm}} | @{\hspace{0.25cm}} c @{\hspace{0.25cm}} | c ||}
    \hline
     \multicolumn{8}{|c|}{Binomial distributions : $\mu={\cal L}(X-5)$, $\nu={\cal L}(Y-20)$ with  $X\sim {\cal B}(10, 0.5)$, $Y\sim {\cal B}(40, 0.5)$} \\
     \hline\rule{0pt}{3ex}
     $\rho$ &    $\overline{\mathcal{M}}_\rho$ &   $\mathcal{I}_\rho^{\,\pi^{\rm HN}}$ & $\lVert \overline \pi - \pi^{HN}\rVert$
     &$\rho$ &    $\underline{\mathcal{M}}_\rho$  & $\mathcal{I}_\rho^{\,\pi^{\rm HN}}$ & $\lVert \underline \pi - \pi^{HN}\rVert$ \\ [0.5ex]
     \hline\rule{0pt}{2.5ex}
        0.3 & 1.366952 &      1.366793 & 5.6e-03 &    2.1 &  10.025562 &     10.025562 & 3.9e-09 \\
        0.7 & 2.090477 &      2.090467 & 2.9e-03 &    2.3 &  12.680593 &     12.680593 & 4.0e-09 \\
        1.0 & 2.893664 &      2.893664 & 0.0     &    2.5 &  16.085421 &     16.085421 & 3.3e-09 \\
        1.4 & 4.504064 &      4.504064 & 4.2e-09 &    3.0 &  29.533935 &     29.533935 & 4.5e-09 \\
        1.9 & 7.949088 &      7.949088 & 3.2e-09 &    5.0 & 407.035561 &    407.347830 & 2.7e-03 \\
     \hline
\end{tabular}
\end{center}
\begin{center}
\begin{tabular}{||c | @{\hspace{0.4cm}} c @{\hspace{0.4cm}}| @{\hspace{0.4cm}} c @{\hspace{0.4cm}} | c || c | @{\hspace{0.3cm}} c @{\hspace{0.3cm}} | @{\hspace{0.3cm}} c @{\hspace{0.3cm}} | c ||}
    \hline
     \multicolumn{8}{|c|}{Poisson distributions : $\mu={\cal L}(X-1)$, $\nu={\cal L}(Y-4)$  with $X\sim{\cal P}(1)$, $Y\sim{\cal P}(4)$} \\
     \hline\rule{0pt}{3ex}
     $\rho$ &    $\overline{\mathcal{M}}_\rho$ &   $\mathcal{I}_\rho^{\,\pi^{\rm HN}}$ & $\lVert \overline \pi - \pi^{HN}\rVert$
     &$\rho$ &    $\underline{\mathcal{M}}_\rho$  & $\mathcal{I}_\rho^{\,\pi^{\rm HN}}$ & $\lVert \underline \pi - \pi^{HN}\rVert$ \\ [0.5ex]
     \hline\rule{0pt}{2.5ex}
        0.3 & 1.116329 &      1.116329 & 1.2e-09 &    2.1 &  2.776364 &      2.776364 & 7.0e-09 \\
        0.7 & 1.317531 &      1.317531 & 2.0e-09 &    2.3 &  3.148135 &      3.148135 & 1.4e-09 \\
        1.0 & 1.513291 &      1.513291 & 0.0     &    2.5 &  3.584910 &      3.584910 & 4.8e-09 \\
        1.4 & 1.854353 &      1.854353 & 1.1e-09 &    3.0 &  5.047828 &      5.047828 & 1.1e-09 \\
        1.9 & 2.459404 &      2.459404 & 1.2e-08 &    5.0 & 24.477011 &     24.477011 & 1.2e-09 \\
     \hline
\end{tabular}
\end{center} 
\subsection{Infimum and supremum of $\left\{{\rm sq}\# \pi:\pi\in\Pi_{M}(\mu,\nu)\right\}$for the convex order}
\begin{proposition}\label{prop:infsupconv}
   Let $\mu,\nu\in{\cal P}_2(\R)$ be such that $\mu\le_{cx}\nu$. Then the set $\left\{{\rm sq}\# \pi:\pi\in\Pi_{M}(\mu,\nu)\right\}$ admits an infimum and a supremum for the convex order. 
\end{proposition}
The conclusion follows from the next lemma and the complete lattice structure of the set $\{\eta\in{\cal P}_1(\R):\int_{\R}z\eta(dz)=\int_\R y^2\nu(dy)-\int_\R x^2\mu(dx)\}$ with minimal element $\delta_{\int_\R y^2\nu(dy)-\int_\R x^2\mu(dx)}$ stated in \cite{kertzrosler}.
\begin{lemma}\label{lem:bornesups}
   Let $\mu,\nu\in{\cal P}_2(\R)$ be such that $\mu\le_{cx}\nu$. Then the set $\left\{{\rm sq}\# \pi:\pi\in\Pi_{M}(\mu,\nu)\right\}$ is bounded from above in the convex order.
 \end{lemma}
\subsection{Definition and optimality of non-decreasing and non-increasing martingale couplings}\label{sec:defoptimonot}
\begin{definition}\label{def:NdCoupling}
Let $\mu,\nu\in\mathcal P_1(\R)$ be such that $\mu\leq_{cx}\nu$. A martingale coupling $\pi(dx,dy)\in\Pi_M(\mu, \nu)$ is called \textbf{non-decreasing} if there exists a Borel set $\Gamma \subseteq \R\times\R$ such that $\pi(\Gamma)=1$ and
\begin{itemize}
    \item[(a)] if $(x_-, y_-), (x_+, y_+)\in \Gamma$ with $y_-\leq x_-$, $y_+\leq x_+$, and $x_-< x_+$, then $y_-\leq y_+$
    \item[(b)] if $(x_-, z_-), (x_+, z_+) \in \Gamma$ with $x_-\leq z_-$, $x_+\leq z_+$, and $x_-<x_+$, then $z_-\leq z_+$.
\end{itemize}
Respectively, $\pi(dx,dy)$ is called \textbf{non-increasing} if there exists a Borel set $\Gamma \subseteq \R\times\R$ such that $\pi(\Gamma)=1$ and
\begin{itemize}
    \item[(c)] if $(x_-, y_+), (x_+, y_-)\in \Gamma$ with $y_+\leq x_-$, $y_-\leq x_+$, and $x_-< x_+$, then $y_-\leq y_+$
    \item[(d)] if $(x_-, z_+), (x_+, z_-) \in \Gamma$ with $x_-\leq z_+$, $x_+\leq z_-$, and $x_-<x_+$, then $z_-\leq z_+$.
\end{itemize}
\begin{figure}[H]
    \subfloat[\centering non-decreasing (left)] {{
        \begin{tikzpicture}
        \node at (0,0.2) (y-) {$y_-$};
        \node at (1,0.2) (y+) {$y_+$};
        \node at (1.5,-2.2) (x-) {$x_-$};
        \node at (2.5,-2.2) (x+) {$x_+$};
        \draw (1,0) -- (2.5,-2);
        \draw (0,0) -- (1.5,-2);
        \end{tikzpicture}}}
    \quad
    \subfloat[\centering non-decreasing (right)]{{
        \begin{tikzpicture}
        \node at (0,-2.2) (x-) {$x_-$};
        \node at (1,-2.2) (x+) {$x_+$};
        \node at (1.5,0.2) (z-) {$z_-$};
        \node at (2.5,0.2) (z+) {$z_+$};
        \draw (1,-2) -- (2.5,0);
        \draw (0,-2) -- (1.5,0);
        \end{tikzpicture}
    }}
    \quad
    \subfloat[\centering non-increasing (left)] {{
        \begin{tikzpicture}
        \node at (0,0.2) (y-) {$y_-$};
        \node at (1,0.2) (y+) {$y_+$};
        \node at (1.5,-2.2) (x-) {$x_-$};
        \node at (2.5,-2.2) (x+) {$x_+$};
        \draw (1,0) -- (1.5,-2);
        \draw (0,0) -- (2.5,-2);
        \end{tikzpicture}}}
    \quad
    \subfloat[\centering non-increasing (right)]{{
        \begin{tikzpicture}
        \node at (0,-2.2) (x-) {$x_-$};
        \node at (1,-2.2) (x+) {$x_+$};
        \node at (1.5,0.2) (z-) {$z_-$};
        \node at (2.5,0.2) (z+) {$z_+$};
        \draw (1,-2) -- (1.5,0);
        \draw (0,-2) -- (2.5,0);
        \end{tikzpicture}
    }}
    \caption{Non-decreasing / non-increasing coupling }%
\end{figure}
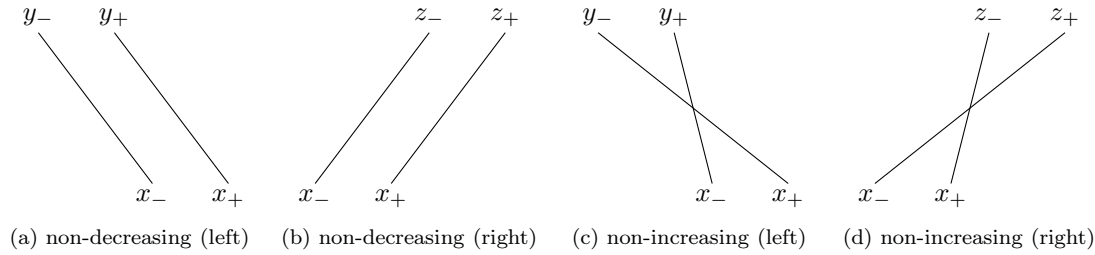
\end{definition}
A different notion of monotone martingale couplings is introduced by Beiglb\"ock and Juillet in \cite{BeJu16} where they prove that there exists a unique martingale coupling $\pi$ left-monotone (resp. right-monotone) in the sense that $\pi(\Gamma)=1$ for some Borel subset $\Gamma$ of $\R^2$ such that $(x_-,y_-)$, $(x_-,z_-)$ and $(x_+,w)\in\Gamma$ with $x_-<x_+$ and $y_-<z_-$ implies that $w\notin(y_-,z_-)$ (resp. $(x_+,y_+)$, $(x_+,z_+)$ and $(x_-,w)\in\Gamma$ with $x_-<x_+$ and $y_+<z_+$ implies that $w\notin(y_+,z_+)$). According to \cite{HeTo13}, the left-monotone (resp. right-monotone) martingale coupling minimizes (resp. maximizes) over $\Pi_M(\mu,\nu)$ the integral of smooth cost functions $c(x,y)$ satisfying the martingale Spence-Mirrlees condition $\partial^3_{xyy}c\le 0$.

In the beginning of the proof of Theorem 7.3 \cite{BeJu16}, combining the necessary optimality condition that they give in Lemma 1.11 together with some specific properties of the cost function $\R^2\ni(x,y)\mapsto |x-y|$ stated in Lemma 7.5, Beiglb\"ock and Juillet show that any coupling maximizing $\int_{\R^2}|x-y|\pi(dx,dy)$ over $\pi\in\Pi_M(\mu,\nu)$ is non-decreasing. Replacing Lemma 7.5 \cite{BeJu16} by Lemma \ref{lemma: 5 point cost} below, we generalize this result to costs $\varphi(|x-y|)$ where $\R_+\ni z\mapsto \varphi(z)\in\R$ is increasing and concave. To check that, on the other hand, any coupling minimizing $\int_{\R^2}\varphi(|x-y|)\pi(dx,dy)$
for such a function is non-increasing, we need $\mu$ and $\nu$ to satisfy the following nested supports condition. \begin{definition}
   We say that $\mu,\nu\in{\cal P}_1(\R)$ satisfy the nested supports condition if there exist $-\infty<a\leq b<+\infty$ such that $\mu\left([a,b]\right)=1$ and $\nu\left((a,b)\right) = 0$.
 \end{definition}
 The nested supports condition is slightly weaker than the Reinforced Dispersion Assumption 7.1 considered by Hobson and Klimmek \cite{HoKl} and which amounts to the existence of some interval $I$ such that $\mu(I)=1$ and $\nu(I)=0$. Indeed, under the nested supports condition, it may happen that $\mu(\{a\})\nu(\{a\})>0$ or $\mu(\{b\})\nu(\{b\})>0$.
 \begin{proposition}\label{prop:variational lemma to nd coupling}
Let $\mu,\nu\in{\cal P}_1(\R)$ be such that $\mu\le_{cx}\nu$ and  $\varphi: \R_+\to\R$ be some increasing concave function. \begin{description}
\item[$(i)$]If $\tilde \pi\in\Pi_M(\mu,\nu)$  maximizes $\int \varphi (\vert x-y\vert) \pi(dx,dy)$ over $\pi\in\Pi_M(\mu,\nu)$, then $\tilde \pi$ is non-decreasing and, when $\varphi$ is continuous, there exists such a maximizing coupling.
  \item[$(ii)$]There exists $\tilde \pi\in\Pi_M(\mu,\nu)$ which minimizes $\int \varphi (\vert x-y\vert) \pi(dx,dy)$ over $\pi\in\Pi_M(\mu,\nu)$ and when $\mu$ and $\nu$ satisfy the nested supports condition, then any such coupling $\tilde M$ is non-increasing.
\end{description}
\end{proposition}
\begin{remark}\label{rk:integvarphi}
  Let $\lambda\in[\varphi'_r(1),\varphi'_l(1)]$ where $\varphi'_r(1)$ and $\varphi'_l(1)$ are the respective right-hand and left-hand derivatives of the increasing concave function $\varphi$ at point $1$. We have $\forall z\in\R_+,\;\varphi(0)\le\varphi(z)\le \varphi(1)+\lambda(z-1)$ so that for $\pi\in\Pi_M(\mu,\nu)$,
  \begin{align*}
   \varphi(0)&\le\int_{\R^2}\varphi (\vert x-y\vert) \pi(dx,dy)\le \varphi(1)+\lambda\int_{\R^2}\vert x-y\vert \pi(dx,dy)\\&\le \varphi(1)+\lambda\left(\int_\R|x|\mu(dx)+\int_\R|y|\nu(dy)\right),
  \end{align*}
    where the right-hand side is finite for $\mu,\nu\in{\cal P}_1(\R)$. The concavity of the function $\varphi:\R_+\to\R$ implies that this function is continuous on $(0,+\infty)$. It may have a positive jump $\lim_{x\to 0+}\varphi(x)-\varphi(0)$ at the origin and then $\R^2\ni(x,y)\mapsto \varphi (\vert x-y\vert)$ is lower semi-continuous but not upper-semicontinuous.
\end{remark}

\begin{remark}
For $\rho\in(0,1]$, since $\R_+\ni z\mapsto z^\rho$ is concave and increasing, a martingale coupling that maximizes $\int \vert x-y\vert^\rho\,\pi(dx,dy)$ is non-decreasing.
\end{remark}
It turns out that the nested supports condition is also necessary for the existence of a non-increasing coupling in $\Pi_M(\mu,\nu)$ and that, under this condition there is exists a unique non-increasing coupling and also a unique non-decreasing coupling in $\Pi_M(\mu,\nu)$.
\begin{proposition}\label{prop:deccoupl}
  Let $\mu,\nu\in{\cal P}_1(\R)$ be such that $\mu\leq_{cx}\nu$. The following assertions are equivalent \begin{description}
  \item[$(i)$] there exists a non-increasing martingale coupling in $\Pi_M(\mu,\nu)$,
  \item[$(ii)$]$\mu$ and $\nu$ satisfy the nested supports condition : $\exists-\infty<a\leq b<+\infty$ s.t. $\mu\left([a,b]\right)=1$ and $\nu\left((a,b)\right) = 0$,
   \item[$(iii)$] there exists a unique non-increasing martingale coupling $\pi^\downarrow$ in $\Pi_M(\mu,\nu)$ and $\pi^\downarrow(\{(a,a)\})=\mu(\{a\})\wedge \nu(\{a\})$, $\pi^\downarrow(\{(b,b)\})=\mu(\{b\})\wedge \nu(\{b\})$. 
   \end{description}
   Moreover, under the nested supports condition, there exists a unique non-decreasing martingale coupling $\pi^\uparrow\in \Pi_M(\mu,\nu)$ and, when $u_\nu(a)>u_\mu(a)$ (resp. $u_\nu(b)>u_\mu(b)$), then $\pi^\uparrow(\{(a,a)\})=(\mu(\{a\})-p_-(a)-p_+(a))^+$ (resp. $\pi^\uparrow(\{(b,b)\})=(\mu(\{b\})-p_-(b)-p_+(b))^+$) where $p_-$ and $p_+$ are defined in Proposition \ref{prop:nu0} below.
\end{proposition}
The next proposition generalizes to continuous, increasing and concave functions $\varphi:\R_+\to\R$ the statement in Theorem 7.3 \cite{BeJu16} which deals with the case when $\varphi$ is the identity function . 
\begin{proposition}\label{prop:uniccrois}
Let $\mu,\nu\in\mathcal P(\R)$ be such that $\mu\leq_{cx}\nu$ and that $\mu$ does not weight points. Let $\varphi: \R_+\to\R$ be some continuous increasing concave function. There exists a unique non-decreasing martingale coupling $\tilde \pi$ that maximizes $\int \varphi (\vert x-y\vert) \pi(dx,dy)$. Moreover, there exist two non-decreasing functions $T_1,T_2:\R\to\R$ such that $T_1(x)\le x\le T_2(x)$ and $\tilde \pi$ is concentrated on the graphs of these functions.\end{proposition}
Notice that the example in Section 7.3.1 \cite{BeJu16} shows that for general continuous, increasing and concave functions $\varphi$, the uniqueness of maximizers may fail without the continuity assumption on $\mu$. The next proposition illustrates that unique maximizers may be distinct for distinct power functions $\varphi$. Since those maximizers are non-decreasing couplings, this motivates the investigation of the set of non-decreasing couplings in the next subsection.
\begin{proposition}\label{prop:diffcouplcr}
Let $0<\rho'<\rho\leq1$, $y<\frac{y+z}{2}<m<z$, $$\beta\in \left(0,2\rho\left(\frac{m-y}{z-y}(z-m)^{\rho-1} -\frac{z-m}{z-y}(m-y)^{\rho-1}\right) \wedge \frac{2m-y-z}{z-m}\right),\mbox{ and}$$
\[\mu_\varepsilon = \frac{1}{1+\left(1+\beta \varepsilon^{1-\rho}\right)^\frac{1}{\rho}}\delta_{m-\left(1+\beta \varepsilon^{1-\rho}\right)^\frac{1}{\rho}\varepsilon} +\frac{\left(1+\beta \varepsilon^{1-\rho}\right)^\frac{1}{\rho}}{1+\left(1+\beta \varepsilon^{1-\rho}\right)^\frac{1}{\rho}} \delta_{m+\varepsilon},\]
\[\nu = \frac13\delta_{m} + \frac{2(z-m)}{3(z-y)}\delta_{y}+\frac{2(m-y)}{3(z-y)}\delta_{z}.\]
For $\varepsilon< (z-m) \wedge \frac{
m-y}{{(1+\beta)}^\frac{1}{\rho}}\wedge\frac{2(z-m)(m-y)}{3(z-y)}\wedge 1$, $\Pi_M(\mu_\varepsilon,\nu)$  contains distinct martingale couplings and any element of $\Pi_M(\mu_\varepsilon,\nu)$ is non-decreasing. For $\varepsilon$ small enough, the non-decreasing martingale coupling that attains $\overline{\mathcal M}_{\rho'}(\mu_\varepsilon,\nu)$ is different from the one that attains $\overline{\mathcal M}_\rho(\mu_\varepsilon,\nu)$.
\end{proposition}
\subsection{Directionally decomposed non-decreasing martingale couplings\label{sec:dirdecmartcoupl}}
The uniqueness of non-decreasing couplings under the nested supports condition which determines the parts of the measure $\nu$ attained in the left direction and in the right direction in any martingale coupling motivates the introduction of the following decomposition.
For $\pi\in\Pi_M(\mu,\nu)$, we denote by 
\begin{align*}
    \nu^\pi_{l}(dy)&=\int_{x\in\R}\mathds{1}_{\{y<x\}}\pi(dx,dy)=\int_{x\in\R}\mathds{1}_{\{y<x\}}\pi_x(dy)\mu(dx),\\\nu^\pi_{0}(dy)&=\int_{x\in\R}\mathds{1}_{\{y=x\}}\pi(dx,dy)=\int_{x\in\R}\mathds{1}_{\{y=x\}}\pi_x(dy)\mu(dx),\\\nu^\pi_{r}(dy)&=\int_{x\in\R}\mathds{1}_{\{y>x\}}\pi(dx,dy)=\int_{x\in\R}\mathds{1}_{\{y>x\}}\pi_x(dy)\mu(dx).
\end{align*}
Note that $\nu^\pi_l+\nu^\pi_0+\nu^\pi_r=\nu$. 
For $(\nu_l,\nu_r)$ a couple of non-negative measures such that $\nu_l+\nu_r\le\nu$, we denote $$\Pi_M(\mu,\nu,\nu_l,\nu_r)=\{\pi\in\Pi_M(\mu,\nu):(\nu^\pi_l,\nu^\pi_0,\nu^\pi_r)=(\nu_l,\nu-\nu_l-\nu_r,\nu_r)\}.$$
Our main result, stated in the next theorem and in assertion $(i)$ of the corollary that follows is that $\Pi_M(\mu,\nu,\nu_l,\nu_r)$ is not empty if and only if there exists a coupling in this set which is non-decreasing away from the support of $\nu-\nu_l-\nu_r$ where it goes straight. 
\begin{theorem}\label{thm:inccoup}
Let $\mu,\nu\in\mathcal P_1(\R)$ be such that $\mu\leq_{cx}\nu$ and $\mu\neq \nu$. Let $(\nu_l,\nu_r)$ a couple of non-negative measures such that $\nu_l+\nu_r=\nu$. We have
$$\Pi_M(\mu,\nu,\nu_l,\nu_r)\neq \emptyset\Longleftrightarrow \exists!\pi^\uparrow\in\Pi_M(\mu,\nu,\nu_l,\nu_r)\mbox{ non-increasing},$$
and then, for each continuous increasing concave function $\varphi:\R_+\to\R$, \begin{equation}
   \forall \pi\in\Pi_M(\mu,\nu,\nu_l,\nu_r)\setminus\{\pi^\uparrow\},\;\int_{\R^2} \varphi(|y-x|)\pi(dx,dy)<\int_{\R^2}\varphi(|y-x|)\pi^{\uparrow}(dx,dy)\label{eq:maxintabs}.
\end{equation}\end{theorem}
\begin{corollary}\label{cor:inccoup}
 Let $\mu,\nu\in\mathcal P_1(\R)$ be such that $\mu\leq_{cx}\nu$ and $\mu\neq \nu$. For $(\nu_l,\nu_r)$ a couple of non-negative measures such that $\nu_l+\nu_r\le\nu$, we have that \begin{align}
      &(i)\;\Pi_M(\mu,\nu,\nu_l,\nu_r)\neq \emptyset\Longleftrightarrow \exists!\pi^\uparrow\in\Pi_M(\mu,\nu,\nu_l,\nu_r) \mbox{ s.t. }\frac{\pi^\uparrow(dx,dy)+(\nu_l+\nu_r-\nu)(dx)\delta_x(dy)}{\nu_l(\R)+\nu_r(\R)}\notag\\&\phantom{(i)\;\Pi_M(\mu,\nu,\nu_l,\nu_r)\neq \emptyset\Longleftrightarrow \exists!\pi^\uparrow\in\Pi_M(\mu,\nu,\nu_l,\nu_r) \mbox{ s.t. }}\mbox{ is non-decreasing}\notag\\&\mbox{ and then }\forall \pi\in\Pi_M(\mu,\nu,\nu_l,\nu_r)\setminus\{\pi^\uparrow\},\;\int_{\R^2} \varphi(|y-x|)\pi(dx,dy)<\int_{\R^2}\varphi(|y-x|)\pi^{\uparrow}(dx,dy)\label{eq:cormaxintabs}\\&\mbox{ for each continuous increasing concave function $\varphi:\R_+\to\R$},\notag\\
                                                                                                                                                                                        &(ii)\;\exists\pi\in\Pi_M(\mu,\nu,\nu_l,\nu_r)\mbox{ non-decreasing }\notag\\&\phantom{(ii)}\Longrightarrow \nu-\nu_l-\nu_r=\mathds{1}_{\{u_\mu=u_\nu\}}\mu+\sum_{\stackrel{x\in\R:}{\mu\wedge\nu(\{x\})\times(u_\nu(x)-u_\mu(x))>0}}p(x)\delta_x\mbox{ for some  }p(x)\in[0,\mu\wedge\nu(\{x\})],\notag\\&(iii)\mbox{ There is at most one non-decreasing coupling $\pi\in\Pi_M(\mu,\nu,\nu_l,\nu_r)$.}\notag
   \end{align}
 \end{corollary}
 According to the next proposition, some restrictive structure on the supports of $\mu$ and $\nu$ is needed to ensure the existence of a martingale coupling in $\Pi_M(\mu,\nu,\nu_l,\nu_r)$ non-increasing  away from the support of $\nu-\nu_l-\nu_r$ where it goes straight.
\begin{proposition}\label{prop:decompdecrois}
   Let $\mu,\nu\in\mathcal P_1(\R)$ be such that $\mu\leq_{cx}\nu$ and $\mu\neq \nu$. For $(\nu_l,\nu_r)$ a couple of non-negative measures such that $\nu_l+\nu_r\le\nu$, the existence of $\pi^\downarrow\in\Pi_M(\mu,\nu,\nu_l,\nu_r)$ such that $\frac{\pi^\downarrow(dx,dy)+(\nu_l+\nu_r-\nu)(dx)\delta_x(dy)}{\nu_l(\R)+\nu_r(\R)}$ is non-increasing is equivalent to the Dispersion Assumption 2.1 in \cite{HoKl} i.e. the existence of a non-empty finite interval $I$ with ends $a\le b$ such that $(\mu-\nu)^+(I)=(\mu-\nu)^+(\R)$ and $(\nu-\mu)^+(I)=0$ combined with $\nu_l(dy)=\mathds{1}_{\{y\le a\}}(\nu-\mu)^+(dy)$ and $\nu_r(dy)=\mathds{1}_{\{y\ge b\}}(\nu-\mu)^+(dy)$.
 \end{proposition}
 In view of statement $(ii)$ in Corollary \ref{cor:inccoup}, to better characterize the couples $(\nu_l,\nu_r)$ such that $\nu_l+\nu_r\le \nu$ and there exists a truly non-decreasing coupling in $\Pi_M(\mu,\nu,\nu_l,\nu_r)$, it is useful to study the points $x\in\R$ such that $\inf_{\pi\in\Pi_M(\mu,\nu)}\pi(\{(x,x)\})>0$.
 \begin{proposition}\label{prop:nu0}
      Let $\mu\le_{cx}\nu$ and $x\in\R$ be such that $u_\nu(x)>u_\mu(x)$. Then there exists a unique couple $(p_-(x),p_+(x))\in (0,F_\nu(x-)]\times(0,1-F_\nu(x)]$ such that
     \begin{align}
      &\int_\R(z-x)^+\nu(dz)=\int_\R(y-x)^+\mu(dy)+\int_{F_\nu(x)}^{F_\nu(x)+p_+(x)}(F_\nu^{-1}(v)-x)dv.\label{eq:dd}\\
                             \mbox{ and }&\int_\R(x-z)^+\nu(dz)=\int_\R(x-y)^+\mu(dy)+\int_{F_\nu(x-)-p_-(x)}^{F_\nu(x-)}(x-F_\nu^{-1}(v))dv\label{eq:gg}
      \end{align} 
      Moreover, $\inf_{\pi\in\Pi_M(\mu,\nu)}\pi(\{(x,x)\})=(\mu(\{x\})-p_-(x)-p_+(x))^+$ where the infimum is attained. When $\mu(\{x\})>p_-(x)+p_+(x)$, then any $\pi\in\Pi_M(\mu,\nu)$ satisfying $\pi(\{(x,x)\})=(\mu(\{x\})-p_-(x)-p_+(x))^+$ is 
        such that $\pi(\{(-\infty,x)\times(x,+\infty)\}\cup\{(x,+\infty)\times(-\infty,x)\})=0$ and $\pi_x$ is equal to
      \begin{equation}
    \eta_x=\left(1-\frac{p_-(x)+p_+(x)}{\mu(\{x\})}\right)\delta_x+\frac{1}{\mu(\{x\})}\int_{F_\nu(x-)-p_-(x)}^{F_\nu(x-)}\delta_{F_\nu^{-1}(v)}dv+\frac{1}{\mu(\{x\})}\int_{F_\nu(x)}^{F_\nu(x)+p_+(x)}\delta_{F_\nu^{-1}(v)}dv.
    \label{eq:pix}     
  \end{equation}
 When $\mu(\{x\})\in(0,p_-(x)+p_+(x)]$, there exists a unique $q(x)\in\left[\left(\mu(\{x\})-p_+(x)\right)^+,p_-(x)\wedge{\mu(\{x\})}\right]$ such that 
$\int_{F_\nu(x-)-q(x)}^{F_\nu(x-)}F_\nu^{-1}(v)dv+\int_{F_\nu(x)}^{F_\nu(x)+\mu(\{x\})-q(x)}F_\nu^{-1}(v)dv=\mu(\{x\})x$ and there exists
  $\pi\in\Pi_M(\mu,\nu)$ such that $\pi_x$ is equal to
  \begin{equation}
    \eta_x=\frac{1}{\mu(\{x\})}\int_{F_\nu(x-)-q(x)}^{F_\nu(x-)}\delta_{F_\nu^{-1}(v)}dv+\frac{1}{\mu(\{x\})}\int_{F_\nu(x)}^{F_\nu(x)+\mu(\{x\})-q(x)}\delta_{F_\nu^{-1}(v)}dv\label{eq:pixx}.\end{equation}
   \end{proposition}
\begin{remark}
      The case when $\mu=\delta_x$ and $\nu\in{\cal P}_1(\R)$ is such that $\int_{\R}z\nu(dz)=x$ and $\nu(\{x\})\in (0,1)$ provides an example where $p_-(x)=F_\nu(x-)$, $p_+(x)=1-F_\nu(x)$ so that $p_-(x)+p_+(x)=1-\nu(\{x\})<\mu(\{x\})$. The only element $\pi(dy,dz)=\delta_x(dy)\nu(dz)$ of $\Pi(\mu,\nu)$ and $\Pi_M(\mu,\nu)$ satisfies $\pi(\{(x,x)\})=\nu(\{x\})=1-F_\nu(x-)-(1-F_\nu(x))=\mu(\{x\})-p_-(x)-p_+(x)$.
    \end{remark}
    
    \begin{corollary}\label{cor:ttdroit}
      Let $\mu\le_{cx}\nu$ and ${\cal X}_0=\{x\in\R:u_\nu(x)>u_\mu(x)\mbox{ and }\mu(\{x\})>p_-(x)+p_+(x)\}$. Then for each $\pi\in\Pi_M(\mu,\nu)$, one has
      $$\mathds{1}_{\{u_\mu(y)=u_\nu(y)\}}\mu(dy)+\sum_{x\in{\cal X}_0}\left(\mu(\{x\})-p_-(x)-p_+(x)\right)\delta_x(dy)\le \nu^\pi_0(dy)\le \mu\wedge\nu(dy),$$
      where the bound from above is attained but the bound from below may not be attained.
      Moreover, when $\mu\neq \nu$ and $(\nu_l,\nu_r)$ is a couple of non-negative measures such that $\nu-\nu_l-\nu_r=\mathds{1}_{\{u_\mu(y)=u_\nu(y)\}}\mu+\sum_{x\in{\cal X}_0}\left(\mu(\{x\})-p_-(x)-p_+(x)\right)\delta_x$, then $$\Pi_M(\mu,\nu,\nu_l,\nu_r)\neq \emptyset\Longleftrightarrow \exists!\pi^\uparrow\in\Pi_M(\mu,\nu,\nu_l,\nu_r)\mbox{ non-decreasing}.$$  \end{corollary}
    In the next example, the bound from below is not attained. 
   
    \begin{example}\label{ex:minopidroit}
      Let $\mu=\frac 12\left(\delta_{-1}+\delta_1\right)$ and $$\nu(dy)=\frac 14\left(\delta_{-1}(dy)+\delta_1(dy)\right)+\frac 16\left(\mathds{1}_{[-2,-1]}(y)+\mathds{1}_{[1,2]}(y)\right)dy+\frac 1{12}\mathds{1}_{[-1,1]}(y)dy.$$ Then $u_\mu(-1)=1=u_\mu(1)$, $u_\nu(-1)=\frac 76=u_\nu(1)$, $p_-(-1)=\frac 16=p_+(1)$ and $p_+(-1)=\frac{\sqrt{2}}{12}=p_-(1)$.
By Proposition \ref{prop:nu0}, there exists $\pi\in\Pi_M(\mu,\nu)$ such that $\pi_{-1}(\{-1\})=1-\frac{p_-(-1)+p_+(-1)}{\mu(\{-1\})}=\frac{4-\sqrt{2}}{6}$ and then $\pi_{-1}=\eta_{-1}$ with $\eta_{-1}$ given by \eqref{eq:pix} and $\pi_1=2\nu-\eta_{-1}$.     Therefore the unique coupling $\pi\in\Pi_M(\mu,\nu)$ such that $\pi_{-1}(\{-1\})=1-\frac{p_-(-1)+p_+(-1)}{\mu(\{-1\})}=\frac{4-\sqrt{2}}{6}$ is
      \begin{align*}
         \pi(dx,dy)&=\frac 12\delta_{-1}(dx)\left(\frac 13\mathds{1}_{[-2,-1]}(y)dy+\frac{4-\sqrt{2}}{6}\delta_{-1}(dy)+\frac 16\mathds{1}_{[-1,-1+\sqrt{2}]}(y)dy\right)\\&+\frac 12\delta_{1}(dx)\left(\frac{\sqrt{2}-1}{6}\delta_{-1}(dy)+\frac 16\mathds{1}_{[-1+\sqrt{2},1]}(y)dy+\frac 12\delta_1(dy)+\frac 13\mathds{1}_{[1,2]}(y)dy\right)
      \end{align*}
and satisfies $\pi_1(\{1\})=\frac 12>\frac{4-\sqrt{2}}{6}=1-\frac{p_-(1)+p_+(1)}{\mu(\{1\})}$. In a symmetric way, the only coupling $\pi\in\Pi_M(\mu,\nu)$ such that $\pi_{1}(\{1\})=1-\frac{p_-(1)+p_+(1)}{\mu(\{1\})}$ satisfies $\pi_{-1}(\{-1\})=\frac 12>\frac{4-\sqrt{2}}{6}=1-\frac{p_-(-1)+p_+(-1)}{\mu(\{-1\})}$.
    \end{example}
The next example shows that one cannot restrict the summation over $\{x\in\R:\mu\wedge\nu(\{x\})\times(u_\nu(x)-u_\mu(x))>0\}$ to a summation over $\{x\in\R:\inf_{\pi\in\Pi_M(\mu,\nu)}\pi(\{x,x\})>0\}$ in Assertion $(ii)$ of Corollary \ref{cor:inccoup}.
\begin{example}
  Let for $\alpha\in[0,1]$, $\nu=\frac 1 6\left((2-\alpha)\delta_{-4}+\alpha\delta_{-1}+2\delta_0+\alpha\delta_{1}+(2-\alpha)\delta_{4}\right)$ and $\mu=\frac 13\left(\delta_{-2}+\delta_0+\delta_2\right)$. Then $u_\nu(0)=\frac{8-3\alpha}{3}=u_\mu(0)+\frac{4-3\alpha}{3}$, $p_-(0)=\frac 16=p_+(0)$, $q(0)=\frac 16$  and $$\pi:=\frac{1}{6}\left(\delta_{(-2,-4)}+\delta_{(-2,0)}+(1-\alpha)\delta_{(0,-4)}+\alpha\delta_{(0,-1)}+\alpha\delta_{(0,1)}+(1-\alpha)\delta_{(0,4)}+\delta_{(2,0)}+\delta_{(2,4)}\right)$$ is a non-decreasing martingale coupling such that $\pi_0=\eta_0$ given by \eqref{eq:pixx} and which does not weight $(0,0)$. On the other hand, $$\frac{4-\alpha}{24}\left(\delta_{(-2,-4)}+\delta_{(2,4)}\right)+\frac \alpha6\left(\delta_{(-2,-1)}+\delta_{(2,1)}\right)+\frac {4-3\alpha}{24}\left(\delta_{(-2,0)}+\delta_{(0,-4)}+\delta_{(0,4)}+\delta_{(2,0)}\right)+\frac \alpha4\delta_{(0,0)}$$
  (and any strict convex combination with the former coupling) is a non-decreasing martingale coupling which weights $(0,0)$ when $\alpha>0$.
\end{example}\subsection{Do {$\pi^\uparrow$} and {$\pi^\downarrow$} optimize the cost function $\vert x-y\vert^\rho$ for $\rho>1$ {under the nested supports condition?}}\label{sec:MHN_MHK}

In this section, we suppose that the probability measures $\mu$ and $\nu$ in the convex order satisfy the nested supports condition  so that, by Proposition \ref{prop:deccoupl}, there exist a unique non-decreasing coupling $\pi^\uparrow$ and a unique non-increasing coupling $\pi^\downarrow$ in $\Pi_M(\mu,\nu)$. Then, by Proposition \ref{prop:variational lemma to nd coupling}, $\pi^\uparrow$ maximises (resp. $\pi^\downarrow$ minimizes) $\int_{\R^2}c(x,y)\pi(dx,dy)$ over $\pi\in \Pi_M(\mu,\nu)$ for cost functions $c(x,y)=\varphi(|x-y|)$ with $\varphi:\R_+\to\R$ continuous, increasing and concave and in particular for $c(x,y)=|x-y|^\rho$ with $\rho\in(0,1]$. We investigate whether $\pi^\uparrow$ (resp. $\pi^\downarrow$) still maximises (resp. minimises) for $c(x,y)=|x-y|^\rho$ with $\rho\in (1,2)$ and minimises (resp. maximises) for $c(x,y)=|x-y|^\rho$ with $\rho>2$. It turns out that this is the case under a reinforced support condition with $\mu$ and $\nu$ not weighting some non empty intervals at the left of $a$ and at the right of $b$. The larger $\rho$, the larger these gaps should be chosen. When $\mu$ only weights two points and $\nu$ weights two points to the left of the support of $\mu$ and two points to the right, we exhibit conditions ensuring that ${\rm sq}\# \pi^{\uparrow}$ is equal to the infimum $\inf_{cx}\{{\rm sq}\#\pi:\pi\in \Pi_M(\mu,\nu)\}$ and ${\rm sq}\# \pi^{\downarrow}$ is equal to the supremum $\sup_{cx}\{{\rm sq}\#\pi :\pi\in \Pi_M(\mu,\nu)\}$ for the convex order. Last, we check that, still with such finite supports, optimality may fail when the gaps are too small. 

\subsubsection{Preservation of the optimality}
\begin{proposition}\label{prop:M^HN with alpha_rho}
Let $\mu\leq_{cx}\nu$ be such that there exist  $\underline y < \overline y < \underline x < \overline x <\underline z < \overline z$ with $\mu\left([\underline x, \overline x]\right) =1$ and $\nu\left([\underline y, \overline y] \cup [\underline z, \overline z]\right) =1$. If
\[\underline x - \overline y \geq \alpha_\rho(\overline z-\overline y),\] 
\[\underline z - \overline x \geq \alpha_\rho(\underline z-\underline y),\]
where $\alpha_\rho\in(0,\frac12)$ is the unique solution of $\psi_\rho(\alpha)=0$ with $\psi_\rho: (0,1] \to\R$ defined by 
\begin{equation*}
\psi_\rho(\alpha)= \alpha +  \alpha^{2-\rho}(1-\alpha)^{\rho-1} +1-\rho.
\end{equation*} 
Then, the unique non-decreasing coupling $\pi^{\uparrow}$ (resp. non-increasing coupling $\pi^\downarrow$) in $\Pi_M(\mu,\nu)$ is the unique optimal coupling in  $\Pi_M(\mu,\nu)$ that attains $\overline{\mathcal M}_\rho (\mu,\nu)$ (resp. $\underline{\mathcal M}_\rho (\mu,\nu)$) when $\rho\in(1,2)$ and $\underline{\mathcal M}_\rho (\mu,\nu)$ (resp. $\overline{\mathcal M}_\rho (\mu,\nu)$) when $\rho > 2$.
\end{proposition}
\begin{remark}
Let $\rho\in (1,2)\cup(2,+\infty)$ and $\overline{y}<\underline{z}$. Since, by the fifth and seventh assertions in Lemma \ref{lemma psi_rho} below, $\alpha_\rho<\frac 12$, it is possible to find $\underline{x},\overline{x}\in (\overline{y}+\alpha_\rho(\underline{z}-\overline{y}),\underline{z}-\alpha_\rho(\underline{z}-\overline{y}))$ such that $\overline{x}-\underline{x}>0$. Then for any $\underline{y}\in\left[\frac{1}{\alpha_\rho}(\overline{x}-(1-\alpha_\rho)\underline{z}),\overline{y}\right)$ and any $\overline{z}\in\left(\underline{z},\frac{1}{\alpha_\rho}(\underline{x}-(1-\alpha_\rho)\overline{y})\right]$, the conditions on the points $\underline y,\overline y,\underline x,\overline x,\underline z,\overline z$ in Proposition \ref{prop:M^HN with alpha_rho} are satisfied.
\end{remark}\begin{proposition}\label{prop:M^HK leq_cx}
Let $p\in(0,1)$ and $y_- < y_+ < x_-<x_+ < z_- < z_+$ be such that 
\begin{align}
    x_+ -y_- &\geq z_--x_-, \;\;\;\;\;z_+-x_- \geq x_+-y_+ \mbox{ and } \nonumber\\
    (x_--y_-)&\wedge(x_+-y_+)\wedge(z_--x_-) \wedge(z_+-x_+)\geq(x_--y_+)\vee (z_--x_+).\label{ineq: 4vee}
\end{align}
We set $\mu = p\delta_{x_-} + (1-p)\delta_{x_+}$, 
\begin{align*}
  \underline \nu &= p\frac{z_--x_-}{z_--y_-}\delta_{y_-} + p\frac{x_--y_-}{z_--y_-}\delta_{z_-}+ (1-p)\frac{z_+-x_+}{z_+-y_+}\delta_{y_+} +(1-p)\frac{x_+-y_+}{z_+-y_+}\delta_{z_+},\\
\pi^{\uparrow} &= p\left(\frac{z_--x_-}{z_--y_-}\delta_{(x_-,y_-)} + \frac{x_--y_-}{z_--y_-}\delta_{(x_-,z_-)}\right)+ (1-p)\left(\frac{z_+-x_+}{z_+-y_+}\delta_{(x_+,y_+)} +\frac{x_+-y_+}{z_+-y_+}\delta_{(x_+,z_+)}\right),\\\overline\nu &= p\frac{z_+-x_-}{z_+-y_+}\delta_{y_+} +p\frac{x_--y_+}{z_+-y_+}\delta_{z_+}+(1-p)\frac{z_--x_+}{z_--y_-}\delta_{y_-} + (1-p)\frac{x_+-y_-}{z_--y_-}\delta_{z_-} ,\\
\pi^{\downarrow} &= p\left(\frac{z_+-x_-}{z_+-y_+}\delta_{(x_-,y_+)} + \frac{x_--y_+}{z_+-y_+}\delta_{(x_-,z_+)}\right)+ (1-p)\left(\frac{z_--x_+}{z_--y_-}\delta_{(x_+,y_-)} +\frac{x_+-y_-}{z_--y_-}\delta_{(x_+,z_-)}\right).\end{align*}  
One has $\pi^\uparrow\in\Pi_M(\mu,\underline\nu)$, ${\rm sq}\# \pi^{\uparrow}(dx,dy)=\inf_{cx}\{{\rm sq}\# \pi:\pi\in \Pi_M(\mu,\underline\nu)\}$ and $\pi^\downarrow\in\Pi_M(\mu,\overline\nu)$, ${\rm sq}\# \pi^{\downarrow}(dx,dy)=\sup_{cx}\{{\rm sq}\# \pi(dx,dy):\pi\in \Pi_M(\mu,\overline\nu)\}$.
\end{proposition}
\begin{remark}
\begin{itemize}
\item If $x_--y_+ = z_--x_+$, then inequality \eqref{ineq: 4vee} is satisfied. So for $y_+<x_-<x_+<z_-$ such that $x_--y_+ = z_--x_+$, we may choose $z_+$ large enough and $y_-$ small enough so that the conditions in Proposition \ref{prop:M^HK leq_cx} are met.
\item We deduce that for any convex function $\varphi:\R_+\to\R$
\begin{align*}
   &\int_{\R^2}\varphi(\vert y-x\vert^2)\pi^\uparrow(dx,dy)=\inf_{\pi\in \Pi_M(\mu,\underline\nu)}\int_{\R^2}\varphi(\vert y-x\vert^2)\pi(dx,dy)\mbox{ and }\\&\int_{\R^2}\varphi(\vert y-x\vert^2)\pi^\downarrow(dx,dy)=\sup_{\pi\in \Pi_M(\mu,\overline\nu)}\int_{\R^2}\varphi(\vert y-x\vert^2)\pi(dx,dy).
\end{align*}
Indeed the function $\varphi(z)=\mathds{1}_{\{z=0\}}\lim_{w\to 0+}\varphi(w)+\mathds{1}_{\{z>0\}}\varphi(z)$ is the uniform limit of the restriction to $\R_+$ of the non-decreasing sequence $(\varphi_n)_{n\ge 1}$ of convex functions on $\R$ defined by $\varphi_n(z)=\mathds{1}_{\{z<\frac{1}{n}\}}(\varphi(\frac{1}{n})+\varphi'_l(\frac{1}{n})(z-\frac{1}{n}))+\mathds{1}_{\{z\ge\frac{1}{n}\}}\varphi(z)$ where $\varphi'_l$ denotes the left-hand derivative of $\varphi$. Moreover, $\varphi(0)-\lim_{w\to 0+}\varphi(w)\ge 0$ and for probability measures $\theta,\vartheta\in{\cal P}_1(\R)$ such that $\theta(\R_+)=1=\vartheta(\R_+)$, $\theta\le_{cx}\vartheta\Rightarrow \theta(\{0\})\le\vartheta(\{0\}) $.

In particular since $\R_+\ni z\mapsto z^{\rho/2}$ is concave for $\rho\in(0,2)$ and convex for $\rho>2$, 
\begin{align*}
   \forall \rho>0,\;(\rho-2)\int_{\R^2}\vert y-x\vert^\rho \pi^\uparrow(dx,dy)&=\inf_{\pi\in \Pi_M(\mu,\underline\nu)}(\rho-2)\int_{\R^2}\vert y-x\vert^\rho \pi(dx,dy)\mbox{ and }\\(\rho-2)\int_{\R^2}\vert y-x\vert^\rho \pi^\downarrow(dx,dy)&=\sup_{\pi\in \Pi_M(\mu,\overline\nu)}(\rho-2)\int_{\R^2}\vert y-x\vert^\rho \pi(dx,dy).
\end{align*}
\end{itemize}
\end{remark}\subsubsection{Non-preservation of the optimality}
For $p\in(0,1)$ and $y_- < y_+ < x_-<x_+ < z_- < z_+$, we set 
\begin{align*}
    \mu &= p\delta_{x_-} + (1-p)\delta_{x_+} , \\
    \overline \nu &= p\frac{z_+-x_-}{z_+-y_-}\delta_{y_-} + p\frac{x_--y_-}{z_+-y_-}\delta_{z_+}+ (1-p)\frac{z_--x_+}{z_--y_+}\delta_{y_+} +(1-p)\frac{x_+-y_+}{z_--y_+}\delta_{z_-}, \\
    \underline \nu &= p\frac{z_--x_-}{z_--y_+}\delta_{y_+} + p\frac{x_--y_+}{z_--y_+}\delta_{z_-}+ (1-p)\frac{z_+-x_+}{z_+-y_-}\delta_{y_-} +(1-p)\frac{x_+-y_-}{z_+-y_-}\delta_{z_+}.
\end{align*} 
The couplings 
\begin{align*}
\overline \pi^\star &= p\left( \frac{z_+-x_-}{z_+-y_-}\delta_{(x_-,y_-)} + \frac{x_--y_-}{z_+-y_-}\delta_{(x_-,z_+)}\right) + (1-p)\left(\frac{z_--x_+}{z_--y_+}\delta_{(x_+,y_+)} +\frac{x_+-y_+}{z_--y_+}\delta_{(x_+,z_-)}\right) , \\
\underline \pi^\star &= p\left( \frac{z_--x_-}{z_--y_+}\delta_{(x_-,y_+)} + \frac{x_--y_+}{z_--y_-}\delta_{(x_-,z_-)}\right) + (1-p)\left(\frac{z_+-x_+}{z_+-y_-}\delta_{(x_+,y_-)} +\frac{x_+-y_-}{z_+-y_-}\delta_{(x_+,z_+)}\right),
\end{align*}
respectively belong to $\Pi_M(\mu,\overline \nu)$ and $\Pi_M(\mu,\underline \nu)$. Since $(x_-,z_+)$ and $(x_+,z_-)$ (resp. $(x_-,y_+)$ and $(x_+,y_-)$) belong to any Borel subset $\Gamma$ of $\R^2$ such that $\overline \pi^\star(\Gamma)=1$ (resp. $\underline \pi^\star(\Gamma)=1$), $\overline \pi^\star$ and $\underline \pi^\star$ are not non-decreasing. Since $(x_-,y_-)$ and $(x_+,y_+)$ (resp. $(x_-,z_-)$ and $(x_+,z_+)$) belong to any Borel subset $\Gamma$ of $\R^2$ such that $\overline \pi^\star(\Gamma)=1$ (resp. $\underline \pi^\star(\Gamma)=1$), $\overline \pi^\star$ and $\underline \pi^\star$ are not non-increasing.

\begin{proposition}\label{prop:M*HN}
Let $\rho\in (1,2)\cup(2,+\infty)$ and $y_- < y_+ < z_- < z_+$ be such that $z_- - y_- > z_+ - z_-$ and $z_--y_+\ge (\rho-1)^{\frac{1}{2-\rho}}(z_+-z_-)$. For $\rho \in (1,2)$ (resp. $\rho >2$), there exists $x_\rho \in (y_+, z_-)$ such that for $x_\rho < x_-<x_+<z_-$, $\overline \pi^\star$ is the unique optimal coupling that attains $\overline{\mathcal M}_\rho (\mu,\overline\nu)$ (resp. $\underline{\mathcal M}_\rho (\mu,\overline\nu)$) and $\underline \pi^\star$ is the unique optimal coupling that attains $\underline{\mathcal M}_\rho (\mu,\underline\nu)$ (resp. $\overline{\mathcal M}_\rho (\mu,\underline\nu)$).
\end{proposition}
\begin{remark}
   For $\rho\in (1,2)\cup(2,+\infty)$, $(\rho-1)^{\frac{1}{2-\rho}}<1$ so that when $y_- < y_+ < z_- < z_+$, it is enough that $z_--y_+\ge z_+-z_-$ to ensure $z_- - y_- > z_+ - z_-$ and $z_--y_+\ge (\rho-1)^{\frac{1}{2-\rho}}(z_+-z_-)$.\end{remark}\begin{remark}
Let $y_- < y_+ < x_-<x_+ < z_- < z_+$ be such that $z_--y_+\ge z_+-z_-$ and $\beta = \frac{z_--y_+}{z_++z_--y_+-y_-}$. Since $\beta\in(0,\frac12)$, by Lemma \ref{lemma psi_rho}, the equation $\alpha_\rho = \beta$ admits a unique root $\rho_\beta>1$. Since $(1,+\infty)\ni\rho\mapsto \alpha_\rho$ is increasing by the fifth item in Lemma \ref{lemma psi_rho}, for $\rho\in (1,\rho_\beta)$, $\alpha_\rho<\beta$ and setting $x_+ = (1-\alpha_\rho)z_-+\alpha_\rho y_-$, we have
\begin{align*}
x_+-y_+ &= z_--y_++\alpha_\rho(y_--z_-)= z_--y_+-\alpha_\rho(z_++z_--y_+-y_-)+\alpha_\rho(z_+-y_+)  \\
&> z_--y_+ -\beta(z_++z_--y_+-y_-)+\alpha_\rho(z_+-y_+)=\alpha_\rho(z_+-y_+).
\end{align*}
We may choose $x_-<x_+$ close to $x_+$ such that $x_--y_+\geq \alpha_\rho(z_+-y_+)$ and then, by Proposition \ref{prop:M^HN with alpha_rho}, the only martingale coupling in $\Pi_M(\mu,\overline\nu)$ that maximizes (resp. minimizes) $\int \vert x-y\vert^\rho \pi(dx,dy)$ when $\rho\in(1,2\wedge\rho_\beta)$ (resp. $\rho\in(2,\rho_\beta)$) is non-decreasing. 
Hence for $\rho\in(1,2)\cup(2,+\infty)$ such that $\rho<\rho_\beta$, one has $x_\rho \geq (1-\alpha_\rho)z_-+\alpha_\rho y_-$. In particular, since $\lim_{\rho\to 1+}\alpha_\rho=0$ by the sixth item in Lemma \ref{lemma psi_rho}, we have $\lim_{\rho\to 1+}x_\rho=z_-$. 

\end{remark}
\begin{remark}
 By Propositions \ref{prop:variational lemma to nd coupling} and \ref{prop:deccoupl}, the unique non-decreasing coupling in $\Pi_M(\mu,\overline\nu)$ (resp. non-increasing coupling in $\Pi_M(\mu,\underline\nu)$) is the unique optimal coupling that attains $\overline{\mathcal M}_\rho (\mu,\overline\nu)$ (resp. $\underline{\mathcal M}_\rho (\mu,\underline\nu)$) for $\rho\in (0,1]$. When $\overline \pi^\star$ (resp. $\underline \pi^\star$) is the unique optimal coupling that attains $\overline{\mathcal M}_\rho (\mu,\overline\nu)$ (resp. $\underline{\mathcal M}_\rho (\mu,\underline\nu)$) for $\rho\in(1,2)$ and $\underline{\mathcal M}_\rho (\mu,\overline\nu)$ (resp. $\overline{\mathcal M}_\rho (\mu,\underline\nu)$) for $\rho>2$, then  $\inf_{cx}\{{\rm sq}\# \pi:\pi\in \Pi_M(\mu,\overline\nu)\}$ (resp. $\sup_{cx}\{{\rm sq}\# \pi:\pi\in \Pi_M(\mu,\underline\nu)\}$) does not belong to $\{{\rm sq}\# \pi:\pi\in \Pi_M(\mu,\overline\nu)\}$ (resp. $\{{\rm sq}\# \pi:\pi\in \Pi_M(\mu,\underline\nu)\}$).
\end{remark}

 \section{Proofs}
\label{sec:proofs}

To recall the necessary optimality condition stated by Beiglb\"ock and Juillet in Lemma 1.11 \cite{BeJu16}, we first recall their definition of competitors :
\begin{definition}[Competitor, taken from Definition 1.10 \cite{BeJu16}]\label{def:competitor}
Let $\alpha$ be a measure on $\R\times\R$ with finite first moment in the second variable. We say that $\alpha'$, a measure on the same space, is a competitor of $\alpha$ if $\alpha'$ has the same marginals as $\alpha$ and 
\[
\int_{y\in\R}\alpha(dx,dy)\mbox{ a.e.},\;\int y \,\alpha_x(dy) = \int y \,\alpha'_x(dy).
\]

\end{definition}
 Now, we state a version of Lemma 1.11 \cite{BeJu16} slightly reinforced by also taking into account Lemma \ref{lemma: x,z in gamma} just below.\begin{lemma}\label{lem:varbj}
  Let $\mu,\nu\in{\cal P}_1(\R)$ be such that $\mu\le_{cx}\nu$ and $c:\R^2\to\R$ be a measurable cost function such that $\forall (x,y)\in\R^2$, $c(x,y)\ge a(x)+b(y)$ for measurable functions $a$ and $b$ integrable with respect to $\mu$ and $\nu$ respectively. When $\pi\in\Pi_M(\mu,\nu)$ is a minimizing coupling leading to finite cost, there exists a Borel set $\Gamma\subset\R^2$ such that $\pi(\Gamma)=1$,
\begin{equation}
   \forall x\in\R,\;\exists y<x\mbox{ s.t. }(x,y)\in\Gamma\Longleftrightarrow\exists z>x\mbox{ s.t. }(x,z)\in\Gamma,\label{eq:l2c}
\end{equation}
  and if $\alpha$ is a measure finitely supported on $\Gamma$, then $\int c(x,y)\alpha(dx,dy)\le \int c(x,y)\alpha'(dx,dy)$ for each competitor $\alpha'$ of $\alpha$.
 \end{lemma}

\begin{lemma}\label{lemma: x,z in gamma}
  Let $\mu,\nu\in\mathcal P_1(\R)$ be such that $\mu\leq_{cx}\nu$ and $\pi\in\Pi_M(\mu,\nu)$. For any Borel subset $\Gamma$ of $\R^2$ such that $\pi(\Gamma)=1$, we may find a Borel subset $\widetilde \Gamma$ of $\R^2$ such that $\widetilde \Gamma\subset\Gamma$, $\pi(\widetilde \Gamma)=1$ and 
$$\forall x\in\R,\;\exists y<x\mbox{ s.t. }(x,y)\in\widetilde\Gamma\Longleftrightarrow\exists z>x\mbox{ s.t. }(x,z)\in\widetilde\Gamma.$$ 
\end{lemma}
    
\begin{proof}[Proof of Lemma \ref{lemma: x,z in gamma}]
For $x\in\R$, let $\Gamma_x = \left\{y\in\R: (x,y)\in\Gamma\right\}$. We set 
\[
A =  \left\{x\in\R: \pi_x\left(\Gamma_x\cap(-\infty,x)\right)\times \pi_x\left(\Gamma_x\cap(x,+\infty)\right) = 0\right\},
\]
and 
\[
\widetilde \Gamma = \Gamma \setminus \bigcup_{x\in A} \{x\}\times \left\{\Gamma_x\cap\{(-\infty,x)\cup(x,+\infty)\} \right\}.
\]
Since $\mu(dx)$ almost everywhere, $\pi_x(\Gamma_x) = 1$ and $\int_{\Gamma_x} y\,\pi_x(dy) = \int_\R y\,\pi_x(dy) = x$, 
\[
\mu(dx)\mbox{ a.e.},\;x\in A\Rightarrow \pi_x\left(\Gamma_x\cap\{(-\infty,x)\cup(x,+\infty)\} \right) = 0.
\]
Therefore\[
\pi \left( \bigcup_{x\in A} \{x\}\times \left\{\Gamma_x\cap\{(-\infty,x)\cup(x,+\infty)\} \right\} \right) = \int_{x\in A} \pi_x\left(\Gamma_x\cap\{(-\infty,x)\cup(x,+\infty)\} \right) \mu(dx) =0,
\]
which implies $\pi(\widetilde \Gamma) = \pi(\Gamma) = 1$. Let $\widetilde \Gamma_x = \left\{y\in\R: (x,y)\in\widetilde \Gamma\right\}$ for $x\in\R$. For $x\in A$, $\widetilde \Gamma_x \subset \{x\}$. For $x\notin A$, $\widetilde \Gamma_x\cap(-\infty,x)=\Gamma_x\cap(-\infty,x)\neq \emptyset$ and $\widetilde \Gamma_x\cap(x,+\infty)=\Gamma_x\cap(x,+\infty) \neq \emptyset$, which concludes the proof.

\end{proof}
\subsection{Proof of Lemma \ref{lem:bornesups}}
When $\nu=\mu$, the conclusion holds since the set under consideration is $\{\delta_0\}$. We now suppose that $\nu\neq \mu$, which implies that $\int_\R y^2\nu(dy)>\int_\R x^2\mu(dx)$.
   Since $2(x^2+y^2)\ge (y-x)^2$, we have that $(y-x)^2> z$ implies $4x^2> z$ or $4y^2> z$ so that, for $\pi\in\Pi(\mu,\nu)$ and $z\ge 0$,
 \begin{align}
  \pi\left(\left\{(x,y):(y-x)^2> z\right\}\right)\le f(z)\mbox{ where }f(z)=
   \left(\mu(\{x:4x^2> z\})+\nu(\{y:4y^2> z\})\right)\wedge 1.\label{eq:minoqueuepi}
 \end{align}
 The non-increasing function $f$ is right-continuous on $[0,+\infty)$ and so is $[0,+\infty)\ni z\mapsto zf(z)+\int_z^{+\infty}f(w)dw$.  We have
 $$zf(z)+\int_z^{+\infty}f(w)dw\le 4\left(\int\mathds{1}_{\{x>\sqrt{z}/2\}}x^2\mu(dx)+\int\mathds{1}_{\{y>\sqrt{z}/2\}}y^2\nu(dy)\right)\underset{z\to+\infty}{\longrightarrow}0.$$
 On the other hand $$\int_0^{+\infty}f(w)dw\ge \int_0^{+\infty}\nu(\{y:4y^2\ge w\})dw
=4\int_\R y^2\nu(dy)>\int_\R y^2\nu(dy)-\int_\R x^2\mu(dx).$$ 
As a consequence $\overline z:=\sup\{z>0:zf(z)+\int_z^{+\infty}f(w)dw\ge \int_{\R}y^2\nu(dy)-\int_{\R}x^2\mu(dx)\}$ belongs to $(0,+\infty)$ and is such that $\overline z\lim_{z\to\overline z-}f(z)+\int_{\overline z}^{+\infty}f(w)dw\ge \int_{\R}y^2\nu(dy)-\int_{\R}x^2\mu(dx)\ge \overline zf(\overline z)+\int_{\overline z}^{+\infty}f(w)dw$. Therefore we may find $p\in[f(\overline z),\lim_{z\to\overline z-}f(z)]\subset[0,1]$ such that $\overline zp+\int_{\overline z}^{+\infty}f(w)dw=\int_{\R}y^2\nu(dy)-\int_{\R}x^2\mu(dx)$. The probability measure $\overline\eta$ with cumulative distribution function  $F_{\overline\eta}(z)=\mathds{1}_{\{0\le z\le\overline{z}\}}(z)(1-p)+\mathds{1}_{\{z>\overline z\}}(1-f(z))$ is such that $\int_{\R}z\overline\eta(dz)=\int_0^{+\infty}(1-F_{\overline\eta}(z))dz=\int_{\R}y^2\nu(dy)-\int_{\R}x^2\mu(dx)$. Let $\pi\in\Pi_M(\mu,\nu)$, $\eta={\rm sq}\# \pi$ and $G(z)=\int_z^{+\infty}(1-F_{\overline \eta}(w))dw-\int_z^{+\infty}(1-F_\eta(w))dw$ for $z\in\R$. By \eqref{eq:minoqueuepi}
, we have $\forall z>\overline z$, $1-F_\eta(z)\le f(z)=1-F_{\overline \eta}(z)$ so that $G(z)\ge 0$ for all $z\in [\overline z,+\infty)$. Since $\int_0^{+\infty}(1-F_\eta(w))dw=\int_{\R}z\eta(dz)=\int_{\R}y^2\nu(dy)-\int_{\R}x^2\mu(dx)$, $G(0)=0$ and since $F_{\overline \eta}(z)=F_\eta(z)=0$ for $z\in (-\infty,0)$, $G(z)=0$ for $z\in (-\infty,0]$. The function $G$ being concave on $[0,\overline z]$ by constancy of $F_{\overline\eta}$ on this interval, we conclude that $G(z)\ge 0$ for all $z\in\R$. By Theorem 3.A.1 (a) \cite{ShakedShanthikumar}, this together with the equality of means implies that $\eta\le_{cx}\overline\eta$. Therefore $\left\{{\rm sq}\# \pi:\pi\in\Pi_{M}(\mu,\nu)\right\}$ is bounded by $\overline\eta$ for the convex order.

\subsection{Proofs of the optimality results in Section \ref{sec:defoptimonot} but Proposition \ref{prop:deccoupl}}
The proof of Proposition \ref{prop:deccoupl}, which relies on the one of Theorem \ref{thm:inccoup}, is postponed to the next section.

\begin{lemma}\label{lemma: 5 point cost}
Let $\varphi:\R_+\to\R$ be concave and increasing. Let $y<m<z$ and $f: \R\to\R$ be the function defined by
\[
f(x) = \frac{z-m}{z-y}\varphi(\vert x-y\vert ) + \frac{m-y}{z-y}\varphi(\vert z-x\vert) -\varphi(\vert x-m\vert).
\]
Then, $f$ is increasing on $[y,m]$ and decreasing on $[m,z]$. 
\end{lemma}

\begin{proof}[Proof of Lemma \ref{lemma: 5 point cost}] 
To check the monotonicity on $[y,m]$, we choose $x_-$ and $x_+$ such that $y\leq x_-<x_+\leq m$, 
\begin{align*}
    f(x_+) - f(x_-) &= \frac{z-m}{z-y}\left(\varphi(x_+-y) -\varphi(x_--y) - \varphi(m-x_+)+\varphi(m-x_-)\right)\\
    &\quad +\frac{m-y}{z-y}\left(\varphi(z-x_+) -\varphi(z-x_-) - \varphi(m-x_+)+\varphi(m-x_-)\right).
\end{align*}
Since $x_+-y>x_--y$ and $m-x_->m-x_+$, we have $\varphi(x_+-y) >\varphi(x_--y)$ and $\varphi(m-x_-)>\varphi(m-x_+)$ so that the first term in the right-hand side is positive. On the other hand, we know that
\[m-x_+\leq (z-x_+)\wedge(m-x_-) \leq(z-x_+)\vee(m-x_-)\leq z-x_-.\]
By the concavity of $\varphi$, we have
\begin{align*}
\varphi(z-x_+) \geq \frac{z-m}{z-m+x_+-x_-}\varphi(z-x_-)+\frac{x_+-x_-}{z-m+x_+-x_-}\varphi(m-x_+) ,\\
\varphi(m-x_-) \geq \frac{z-m}{z-m+x_+-x_-}\varphi(m-x_+)+\frac{x_+-x_-}{z-m+x_+-x_-}\varphi(z-x_-),
\end{align*}
which implies $\varphi(z-x_+) +\varphi(m-x_-)\ge \varphi(z-x_-) + \varphi(m-x_+)$. Therefore, $f(x_+)-f(x_-)>0$ and $f$ is increasing on $[y,m]$.

$ $\newline
Similarly, to check the monotonicity on $[m,z]$, we choose $x_-$ and $x_+$ such that $m\leq x_-< x_+\leq z$, 
\begin{align*}
    f(x_+) - f(x_-) &= \frac{z-m}{z-y}\left(\varphi(x_+-y) -\varphi(x_--y) - \varphi(x_+-m)+\varphi(x_--m)\right)\\
    &\quad +\frac{m-y}{z-y}\left(\varphi(z-x_+) -\varphi(z-x_-) - \varphi(x_+-m)+\varphi(x_--m)\right).
\end{align*}
Since $z-x_+<z-x_-$ and $x_--m<x_+-m$, we have $\varphi(z-x_+)<\varphi(z-x_-)$ and $\varphi(x_--m)<\varphi(x_+-m)$ so that the second term in the right-hand side is negative. On the other hand, we know that
\[x_--m\leq (x_+-m)\wedge(x_--y) \leq(x_+-m)\vee(x_--y)\leq x_+-y.\]
By the concavity of $\varphi$, we have
\begin{align*}
\varphi(x_+-m) \geq \frac{m-y}{m-y+x_+-x_-}\varphi(x_--m)+\frac{x_+-x_-}{m-y+x_+-x_-}\varphi(x_+-y) ,\\
\varphi(x_--y) \geq \frac{m-y}{m-y+x_+-x_-}\varphi(x_+-y)+\frac{x_+-x_-}{m-y+x_+-x_-}\varphi(x_--m),
\end{align*}
which implies $\varphi(x_+-y)+\varphi(x_--m)\le \varphi(x_--y) + \varphi(x_+-m)$. Therefore, $f(x_+)-f(x_-)<0$ and $f$ is decreasing on $[m,z]$.


\end{proof}

\begin{proof}[Proof of Proposition \ref{prop:variational lemma to nd coupling}] The cost function $\R^2\ni(x,y)\mapsto\varphi(\vert x-y\vert)$ is lower-semicontinuous and even continuous when $\varphi$ is continuous. Since it satisfies the bounds in Remark \ref{rk:integvarphi}, the existence of optimizers follows from Theorem 1 \cite{BeHePe12}. Let us establish the monotonicity properties. By Lemma \ref{lem:varbj}, if $\tilde \pi\in \Pi_M(\mu,\nu)$ is a maximizing (resp. minimizing) martingale coupling which leads to finite costs, then there exists a Borel set $\Gamma \subseteq \R\times\R$ (resp. $\Gamma \subseteq [a,b]\times(-\infty,a]\cup[b,+\infty)$) with $\tilde \pi(\Gamma) = 1$ such that \begin{equation}
   \forall x\in\R,\;\exists y<x\mbox{ s.t. }(x,y)\in\Gamma\Longleftrightarrow\exists z>x\mbox{ s.t. }(x,z)\in\Gamma,\label{eq:2c}
\end{equation}
  and for a measure $\alpha$ which is finitely supported on $\Gamma$, we have 
\begin{equation}\label{ineq: Variational lemma alpha max}
\int \varphi(\vert x-y\vert) \,\alpha'(dx,dy) \stackrel{(\mbox{resp. }\geq)}{\leq} \int \varphi(\vert x-y\vert) \,\alpha(dx,dy),
\end{equation}
for every competitor $\alpha'$ of $\alpha$. Assume that $(x_-, y_-), (x_+, y_+)\in \Gamma$ (resp. $(x_-, y_+), (x_+, y_-)\in \Gamma$ 
) with $y_-\leq x_-$, $y_+\leq x_+$ (resp. $y_+\leq x_-$, $y_-\leq x_+$) and $x_-<x_+$. Let us assume $y_+<y_-$ and obtain a contradiction against \eqref{ineq: Variational lemma alpha max}. Then $y_+<y_-\le x_-<x_+$ (resp. $y_+<y_-\leq a\leq x_-<x_+$). By \eqref{eq:2c}, there exists $z>x_+$ such that $(x_+,z)\in \Gamma$ (resp. $z>x_-$ such that $(x_-,z)\in \Gamma$ and then $y_+<y_-\leq a\leq x_-<x_+\leq b \leq z$). By Lemma \ref{lemma: 5 point cost} applied with $(y,m,z)=(y_+,y_-,z)$, the function $f$ defined by \[f(x) = \frac{z-y_-}{z-y_+}\varphi(\vert x-y_+\vert ) + \frac{y_--y_+}{z-y_+}\varphi(\vert z-x\vert) -\varphi(\vert x-y_-\vert)\] is decreasing on $[y_-,z]$. We deduce that for 
\begin{align*}
    \beta(dx,dy) &= \delta_{(x_-,y_-)} + \frac{z-y_-}{z-y_+}\delta_{(x_+,y_+)} + \frac{y_--y_+}{z-y_+}\delta_{(x_+,z)} , \\
    \gamma(dx,dy) &= \delta_{(x_+,y_-)} + \frac{z-y_-}{z-y_+}\delta_{(x_-,y_+)} + \frac{y_--y_+}{z-y_+}\delta_{(x_-,z)},
\end{align*}
\[
\int \varphi(\vert x-y\vert) \,\beta(dx,dy) - \int \varphi(\vert x-y\vert) \,\gamma(dx,dy)= f(x_+)-f(x_-)<0.
\]
This contradicts inequality \eqref{ineq: Variational lemma alpha max} since $\gamma$ (resp. $\beta$) is a competitor of $\beta$ (resp. $\gamma$), which is finitely supported on $\Gamma$. Therefore, we deduce that $y_-\leq y_+$.  

On the other hand, assume that $(x_-,z_-),(x_+,z_+)\in\Gamma$ (resp. $(x_-,z_+),(x_+,z_-)\in\Gamma$) with $x_-\leq z_-$, $x_+\leq z_+$ (resp. $x_-\le z_+$, $x_+\le z_-$) and $x_-<x_+$. Let us assume that $z_+<z_-$ and obtain a contradiction against \eqref{ineq: Variational lemma alpha max}. Then $x_-<x_+\le z_+<z_-$ (resp. $x_-<x_+\le b\le z_+<z_-$). By \eqref{eq:2c}, there exists $y<x_-$ such that $(x_-,y)\in \Gamma$ (resp. $y<x_+$ such that $(x_+,y)\in \Gamma$ and then $y\le a\le x_-<x_+\le b\le z_+<z_-$). By Lemma \ref{lemma: 5 point cost} applied with $(y,m,z)=(y,z_+,z_-)$, the function $g$ defined by\[
g(x) = \frac{z_--z_+}{z_--y}\varphi(\vert x-y\vert ) + \frac{z_+-y}{z_--y}\varphi(\vert z_--x\vert) -\varphi(\vert x-z_+\vert)\]
is increasing on $[y,z_+]$. We deduce that for
\begin{align*}
    \beta(dx,dy) &= \delta_{(x_+,z_+)} + \frac{z_--z_+}{z_--y}\delta_{(x_-,y)} + \frac{z_+-y}{z_--y}\delta_{(x_-,z_-)} ,\\
    \gamma(dx,dy) &= \delta_{(x_-,z_+)} + \frac{z_--z_+}{z_--y}\delta_{(x_+,y)} + \frac{z_+-y}{z_--y}\delta_{(x_+,z_-)},
\end{align*}\[
\int \vert x-y\vert \,\beta(dx,dy) - \int \vert x-y\vert \,\gamma(dx,dy)= g(x_-)-g(x_+)<0.
\]
This contradicts inequality \eqref{ineq: Variational lemma alpha max} since $\gamma$ (resp. $\beta$) is a competitor of $\beta$ (resp. $\gamma$), which is finitely supported on $\Gamma$. Therefore, we deduce that $z_-\leq z_+$.
\end{proof}
\begin{remark}
   The support assumption $\mu([a,b])=1=1-\nu((a,b))$ made in Proposition \ref{prop:variational lemma to nd coupling} to ensure that any minimizing martingale coupling for the cost $\varphi(|y-x|)$ is non-increasing is crucial to avoid in the proof the situations $(x_-,y_+),(x_+,y_-)\in\Gamma$ with $y_+\le x_-<y_-\le x_+$ and $(x_-,z_+),(x_+,z_-)\in\Gamma$ with $x_-\le z_+<x_+\le z_-$  where we could not derive a contradiction.
\end{remark}

\begin{proof}[Proof of Proposition \ref{prop:uniccrois}]
The existence follows from Proposition \ref{prop:variational lemma to nd coupling}.   The proof for uniqueness is the same as the one of Theorem 7.3 \cite{BeJu16}. Since, from Lemma \ref{lemma: 5 point cost}, we only know that $f$ is increasing on $[y,m]$ and decreasing on $[m,z]$, the forbidden cases in Equation (23) \cite{BeJu16} should be replaced by
  \begin{equation}\label{eq:23rest}
   y'\le x'<x\le y^+\mbox{ or }y^-\le x<x'\le y'.
  \end{equation}
  Actually, in the case $\varphi(u)=u$ considered in Theorem 7.3 \cite{BeJu16}, $f$ is constant on $(-\infty,y]$ and on $[z,+\infty)$ and the forbidden cases should be $y'\le x'<x\wedge y^+$ or $y^-\vee x<x'\le y'$ instead of the more general condition in Equation (23) \cite{BeJu16}. Nevertheless, in view of Lemma \ref{lemma: x,z in gamma}
 ,  in the second paragraph of the proof of Theorem 7.3 \cite{BeJu16}, we can assume that $b^-<a<b^+$ so that we end up with a situation in the restricted set of forbidden cases \eqref{eq:23rest}.
\end{proof}

\begin{proof}[Proof of Proposition \ref{prop:diffcouplcr}]
We set $x_-=m-\left(1+\beta \varepsilon^{1-\rho}\right)^\frac{1}{\rho}\varepsilon$ and $x_+=m+\varepsilon$. Let us suppose that $\varepsilon< (z-m) \wedge \frac{
m-y}{{(1+\beta)}^\frac{1}{\rho}}\wedge 1$. This ensures that $y<x_-$ and $x_+<z$.  For $x\in (-\infty,x_-] \cup [x_+,+\infty)$, $u_{\mu_\varepsilon}(x) = \vert x-m\vert \leq u_\nu(x)$.
The function $u_{\mu_\varepsilon}$ is affine on $[x_-,x_+]$. 
Since $y<x_-<m<x_+<z$, the function $u_\nu$ is affine on $[x_-,m]$ and on $[m,x_+]$. The inequality $\varepsilon < \frac{2(z-m)(m-y)}{3(z-y)}$ in our assumptions implies that $u_{\mu_\varepsilon}(m)=\frac{2\left(1+\beta \varepsilon^{1-\rho}\right)^\frac{1}{\rho}}{1+\left(1+\beta \varepsilon^{1-\rho}\right)^\frac{1}{\rho}}\varepsilon < \frac{4(z-m)(m-y)}{3(z-y)}=u_\nu(m)$. Then,
\[
\forall x\in\R,\quad u_{\mu_\varepsilon}(x) \leq u_\nu(x) \qquad\mbox{ and }\qquad \mu_\varepsilon\leq_{cx}\nu.
\]Since $y<x_-<m<x_+<z$, the set $\Gamma=\{(x_-,y),(x_-,m),(x_-,z),(x_+,y),(x_+,m),(x_+,z)\}$ is such that $\Gamma\cap \{(x_1,x_2)\in\R^2:x_2\le x_1\}=\{(x_-,y),(x_+,y),(x_+,m)\}$ and $\Gamma\cap \{(x_1,x_2)\in\R^2:x_1\le x_2 \}=\{(x_-,m),(x_-,z),(x_+,z)\}$ so that $\Gamma$ satisfies conditions $(a)(b)$ in Definition \ref{def:NdCoupling}. Since any martingale coupling gives full weight to $\Gamma$, it is non-decreasing.

Then, let $\pi(dx,dy)=\mu_\varepsilon(dx)\,\pi_x(dy)\in\Pi_M(\mu_\varepsilon,\nu)$. Since $u_{\mu_\varepsilon}(m) < u_\nu(m)$ and
\begin{align*}
u_{\mu_\varepsilon}(m) &= \mu_\varepsilon(\{x_-\})(m-x_-)+\mu_\varepsilon(\{x_+\})(x_+-m) \\  u_\nu(m)&=\mu_\varepsilon(\{x_-\})\int \vert m-w\vert \,\pi_{x_-}(dw)+\mu_\varepsilon(\{x_+\})\int \vert m-w\vert \,\pi_{x_+}(dw),
\end{align*}
we have either $m-x_- < \int \vert m-w\vert \,\pi_{x_-}(dw)$ so that $\pi_{x_-}(\{y\})\pi_{x_-}(\{z\})>0$ or $x_+ - m < \int \vert m-w\vert \,\pi_{x_+}(dw)$ so that $\pi_{x_+}(\{y\})\pi_{x_+}(\{z\})>0$. We deduce that when $\pi_{x-}(\{m\})\pi_{x+}(\{m\})>0$, then either $\pi_{x_-}(\{m\})\,\pi_{x_+}(\{y\})\,\pi_{x_+}(\{z\}) >0$ or $\pi_{x_+}(\{m\})\,\pi_{x_-}(\{y\})\,\pi_{x_-}(\{z\}) >0$. If $\pi_{x_+}(\{m\}) = 0$ (resp. $\pi_{x_-}(\{m\}) = 0$), then  $\pi_{x_+}(\{y\})\pi_{x_+}(\{z\}) > 0$ (resp. $\pi_{x_-}(\{y\})\pi_{x_-}(\{z\}) > 0$) by the martingale property, and since $\nu(\{m\})>0$, we have $\pi_{x_-}(\{m\}) > 0$ (resp. $\pi_{x_+}(\{m\}) > 0$). Therefore, we always have 
\[
\pi_{x_-}(\{m\})\,\pi_{x_+}(\{y\})\,\pi_{x_+}(\{z\}) >0  \quad\mbox{ or }\quad \pi_{x_-}(\{y\})\,\pi_{x_-}(\{z\})\,\pi_{x_+}(\{m\}) >0 ,
\]  
and there exists more than one martingale coupling between $\mu_\varepsilon$ and $\nu$ since for $\alpha$ positive and small enough in the first case and $-\alpha$ positive and small enough in the second case, adding $$\alpha \left(\delta_{(x_+,m)}+ \frac{z-m}{z-y}\delta_{(x_-,y)}+ \frac{m-y}{z-y}\delta_{(x_-,z)}-\delta_{(x_-,m)}- \frac{z-m}{z-y}\delta_{(x_+,y)}+ \frac{m-y}{z-y}\delta_{(x_+,z)}\right)$$
to $\pi$ leads to a distinct coupling in $\Pi_M(\mu_\varepsilon,\nu)$. 
To further parametrize the coupling $\pi$, let us apply the martingale property again. For $x\in \{x_-,x_+\}$, we have 
\[
\begin{dcases}
\pi_x(\{y\}) + \pi_x(\{m\}) + \pi_x(\{z\}) = 1  \\
y\,\pi_x(\{y\}) + m\,\pi_x(\{m\}) + z\,\pi_x(\{z\}) = x 
\end{dcases}
\quad\implies\quad
\begin{dcases}
\pi_x(\{y\}) &= \frac{z-x-(z-m)\pi_{x}(\{m\})}{z-y} ,\\
\pi_x(\{z\}) &= \frac{x-y-(m-y)\pi_{x}(\{m\})}{z-y}.
\end{dcases}
\]
Since $\mu_\varepsilon(\{x_+\})\pi_{x_+}(\{m\}) = \nu(\{m\}) - \mu_\varepsilon(\{x_-\})\pi_{x_-}(\{m\})$, the coupling $\pi$ can be parameterized with respect to the variable $\pi_{x_-}(\{m\})$ only as follows:
\allowdisplaybreaks
\begin{align*}
    \pi
    &= 
      \gamma+
    \mu_\varepsilon(\{x_-\})\pi_{x_-}(\{m\})\left(\frac{z-m}{z-y}\delta_{(x_+,y)}+ \frac{m-y}{z-y}\delta_{(x_+,z)}-\delta_{(x_+,m)}\right. \\
    &\phantom{=\gamma+
    \mu_\varepsilon(\{x_-\})\pi_{x_-}(\{m\})\bigg(}\left.
    -\frac{z-m}{z-y}\delta_{(x_-,y)}
    -\frac{m-y}{z-y}\delta_{(x_-,z)}
    +\delta_{(x_-,m)}\right) \mbox{ with }\\
   \gamma&=\mu_\varepsilon(\{x_-\})\left(\frac{z-x_-}{z-y}\delta_{(x_-,y)} + \frac{x_--y}{z-y}\delta_{(x_-,z)}\right)+\mu_\varepsilon(\{x_+\})\left(\frac{z-x_+}{z-y}\delta_{(x_+,y)}+
   \frac{x_+-y}{z-y} \delta_{(x_+,z)}\right)\\
    &\quad +\nu(\{m\})\left(\delta_{(x_+,m)} -\frac{z-m}{z-y}\delta_{(x_+,y)}- \frac{m-y}{z-y}\delta_{(x_+,z)}\right).
\end{align*}
For $\tilde \rho\in (0,1]$, defining the cost function $f_{\tilde \rho}:[y,z]\to\R$ by
\[f_{\tilde \rho}(x) = \frac{z-m}{z-y}(x-y)^{\tilde \rho}+\frac{m-y}{z-y}(z-x)^{\tilde \rho}-\vert x-m\vert^{\tilde \rho},\]
we have 
\[
\begin{split}
\int\vert x-y \vert^{\tilde \rho} \pi(dx,dy) = 
 \int\vert x-y \vert^{\tilde \rho} \gamma(dx,dy)+ \mu_\varepsilon(\{x_-\})\pi_{x_-}(\{m\})&(f_\rho(x_+)-f_\rho(x_-)).
\end{split}
\]
To maximize $\int\vert x-y \vert^{\tilde \rho} \pi(dx,dy)$, we need to maximize $\pi_{x_-}(\{m\})$ when $f_{\tilde \rho}(x_+)>f_{\tilde \rho}(x_-)$ and minimize $\pi_{x_-}(\{m\})$ when $f_{\tilde \rho}(x_+)<f_{\tilde \rho}(x_-)$. Applying the first order Taylor expansion at $x=m$ to the first two terms of the function $f_{\tilde \rho}$ with $\tilde{\rho}\in(0,1]$, we get
\begin{align}
f_{\tilde{\rho}}(x_+) - f_{\tilde{\rho}}(x_-) &= \tilde{\rho}\left(\frac{z-m}{z-y}(m-y)^{\tilde{\rho}-1} -\frac{m-y}{z-y}(z-m)^{\tilde{\rho}-1}\right)(x_+-x_-) + o(x_+-x_-)\notag\\
&\quad+ (m-x_-)^{\tilde{\rho}} -(x_+-m)^{\tilde{\rho}}.\label{eq:tayltilde}
\end{align}
Let us suppose that $\rho\in(0,1)$. We have that $(m-x_-)^\rho-(x_+-m)^\rho=\beta\varepsilon$. When $\varepsilon\to 0+$, $\left(1+\beta \varepsilon^{1-\rho}\right)^\frac{1}{\rho} \to 1$ so that $x_+-x_-\sim 2\varepsilon$ and
\begin{align*}
f_{\rho}(x_+) - f_{\rho}(x_-)
&=\left(2\rho\left(\frac{z-m}{z-y}(m-y)^{\rho-1} -\frac{m-y}{z-y}(z-m)^{\rho-1}\right)+\beta\right)\varepsilon + o(\varepsilon).
\end{align*}
Since $\beta\in \left(0,2\rho\left(\frac{m-y}{z-y}(z-m)^{\rho-1} -\frac{z-m}{z-y}(m-y)^{\rho-1}\right) \right)$ by assumption, for $\varepsilon$ positive small enough, $f_\rho(x_+)-f_\rho(x_-) < 0$.
On the other hand, for $\rho'\in(0,\rho)$, we have 
$\left(1+\beta \varepsilon^{1-\rho}\right)^\frac{\rho'}{\rho} -1 \;\sim\; \frac{\rho'}{\rho} \beta\, \varepsilon^{1-\rho}$ so that
$(m-x_-)^{\rho'} -(x_+-m)^{\rho'} \sim \frac{\rho'}{\rho} \beta\, \varepsilon^{1-\rho+\rho'}$.
Hence, by \eqref{eq:tayltilde} for $\tilde\rho=\rho'$, 
$f_{\rho'}(x_+) - f_{\rho'}(x_-)\sim \frac{\rho'}{\rho}\beta\, \varepsilon^{1-\rho+\rho'} >0$.

 Finally, let us suppose that $\rho = 1$. We have
\[
f_1(x_+)-f_1(x_-) = \frac{2\varepsilon(\beta(z-m)-(2m-y-z))}{z-y}.
\]
Since $y<\frac{y+z}{2}<m<z$ and $\beta < \frac{2m-y-z}{z-m}$ by assumption, we have $f_1(x_+)-f_1(x_-)< 0$. For $\rho'\in (0,1)$, $(m-x_-)^{\rho'} -(x_+-m)^{\rho'} \sim \left((1+\beta)^{\rho'}-1\right)\varepsilon^{\rho'}$ and, since $x_+-x_-=(2+\beta)\varepsilon$, by \eqref{eq:tayltilde} applied with $\tilde\rho=\rho'$, $f_{\rho'}(x_+) - f_{\rho'}(x_-)\sim \left((1+\beta)^{\rho'}-1\right)\varepsilon^{\rho'}>0$.

Therefore, for $0<\rho'<\rho\leq1$ and $\varepsilon$ positive small enough, to maximize $\int\vert x-y \vert^{\rho'} \pi(dx,dy)$, we need to maximize $\pi_{x_-}(\{m\})$ and to maximize $\int\vert x-y \vert^\rho \,\pi(dx,dy)$, we need to minimize $\pi_{x_-}(\{m\})$.\end{proof}
\subsection{Proof of the results in Section \ref{sec:dirdecmartcoupl} and Proposition \ref{prop:deccoupl}}
 \begin{proof}[Proof of Corollary \ref{cor:inccoup}]
$(i)$ As $\mu\neq \nu$, the non vacuity of $\pi\in\Pi_M(\mu,\nu,\nu_l,\nu_r)$ implies that $\nu_l+\nu_r$ is not the zero measure. Since \begin{align*}
   \Pi_M(\mu,\nu,\nu_l,\nu_r)\ni \pi\mapsto &\frac{\pi(dx,dy)+(\nu_l+\nu_r-\nu)(dx)\delta_x(dy)}{\nu_l(\R)+\nu_r(\R)}\\&\in\Pi_M\left(\frac{\mu+\nu_l+\nu_r-\nu}{\nu_l(\R)+\nu_r(\R)},\frac{\nu_l+\nu_r}{\nu_l(\R)+\nu_r(\R)},\frac{\nu_l}{\nu_l(\R)+\nu_r(\R)},\frac{\nu_r}{\nu_l(\R)+\nu_r(\R)}\right)
                                                                                                                                                                                                                                                                                                                                                              \end{align*} is a bijection, the statement follows from Theorem \ref{thm:inccoup} (concerning \eqref{eq:cormaxintabs}, this relies on the fact that the corresponding mapping between the integrals is affine with a positive factor).
                                                                                                                                                                                                                                                                                                                                                              
          $(ii)$ Let $\pi\in\Pi_M(\mu,\nu,\nu_l,\nu_r)$ be non-decreasing and let $\Gamma$ denote a Borel subset of $\R\times\R$ satisfying the properties in Definition \ref{def:NdCoupling}. Since $\pi(\{(x,x)\})\le\mu(\{x\})\wedge\nu(\{x\})$, we remove $\{(x,x):\mu(\{x\})>0\mbox{ and }\nu(\{x\})=0\}\cap\Gamma$ from $\Gamma$, while preserving these properties. Since $\pi(\Gamma)=1$, we have \begin{align*}\nu_0^\pi(\R)=\pi(\{(x,x):x\in\R\})=\pi(\{(x,x):x\in\R\}\cap\Gamma)=\nu_0^\pi(\{x\in\R:(x,x)\in\Gamma\}).\end{align*}Let $x\in\R$ be such that $(x,x)\in\Gamma$. Then, $\Gamma\subset \{(-\infty,x]\times (-\infty,x]\}\cup \{[x,+\infty)\times [x,+\infty)\}$ and either $\mu(\{x\})>0$ 
or
\begin{align*}
  u_\nu(x)&=\int_{\R^2}|z-x|\pi(dy,dz)=\int_{(y,z)\in(-\infty,x)\times\R}(x-z)\pi_y(dz)\mu(dy)+\int_{(y,z)\in(x,+\infty)\times\R}(z-x)\pi_y(dz)\mu(dy)\\
  &=\int_{(-\infty,x)}(x-y)\mu(dy)+\int_{(x,+\infty)}(y-x)\mu(dy)=u_\mu(x).
\end{align*}
   Hence
$$\nu_0^\pi(\R)=\nu^\pi_0\left(\left\{x\in\R:u_\nu(x)=u_\mu(x)\mbox{ or }(u_\nu(x)-u_\mu(x))\times\left(\mu(\{x\})\wedge\nu(\{x\})\right)>0\right\}\right).$$
  The conclusion follows since $\nu^\pi_0\le\mu\wedge\nu$ by definition and  $\nu^\pi_0(dx)\ge\mathds{1}_{\{u_\mu(x)=u_\nu(x)\}}\mu(dx)$ from the decomposition in irreducible components stated in \cite[Theorem A.4]{BeJu16}.   


  $(iii)$ Since for any non-decreasing coupling $\pi\in\Pi_M(\mu,\nu,\nu_l,\nu_r)$, $\frac{\pi(dx,dy)+(\nu_l+\nu_r-\nu)(dx)\delta_x(dy)}{\nu_l(\R)+\nu_r(\R)}$ is non-decreasing, the uniqueness is ensured by $(i)$. 

    \end{proof}
The proof of Theorem \ref{thm:inccoup} relies on the following lemmas.
\begin{lemma}\label{lem:compmom}Let $\eta\in{\cal P}(\R)$ and $\hat \eta$ be a non-negative measure on the real line such that $\hat \eta\le \eta$. For $f:\R\to\R$ non-decreasing such that $\int_\R|f(x)|\eta(dx)<+\infty$, one has $\int_\R |f(x)|\hat \eta(dx)<+\infty$ and 
  $$\int_0^{\hat\eta(\R)}f(F_\eta^{-1}(u))du\le \int_\R f(x)\hat \eta(dx)\le \int^1_{1-\hat\eta(\R)}f(F_\eta^{-1}(u))du,$$
  $$\mbox{ and if $\hat\eta(\R)>0$, }\frac{1}{\hat\eta(\R)}\int_0^{\hat\eta(\R)}\delta_{F_\eta^{-1}(u)}du\le_{st}\frac{\hat\eta}{\hat\eta(\R)}\le_{st}\frac{1}{\hat\eta(\R)}\int^1_{1-\hat\eta(\R)}\delta_{F_\eta^{-1}(u)}du.$$
\end{lemma}
\begin{lemma}\label{lem:stricstoch}For $\mu,\nu\in{\mathcal P}(\R)$,$$\exists\pi\in\Pi(\mu,\nu)\mbox{ s.t. }\pi\left(\{(x,y)\in\R^2:x>y\}\right)=1\Longleftrightarrow du\mbox{ a.e. on }(0,1),\;F_\mu^{-1}(u)>F_\nu^{-1}(u).$$
\end{lemma}
\begin{proof}[Proof of Lemma \ref{lem:compmom}] Since the inequalities are obvious when $\hat\eta(\R)=0$, we suppose that $\hat\eta(\R)>0$ and denote $\tilde\eta=\frac{\hat\eta}{\hat\eta(\R)}$. 
  Since $\hat \eta\le \eta$
  , $\int_\R |f(x)|\hat \eta(dx)\le \int_\R |f(x)|\eta(dx)<+\infty$
  . Moreover, by the inverse transform sampling, $$\int_\R f(x)\hat \eta(dx)=\hat\eta(\R)\int_\R f(x)\tilde\eta(dx)=\hat\eta(\R)\int_0^1f(F_{\tilde\eta}^{-1}(u))du.$$
  For $u\in (0,1)$, using \eqref{eq:qt ineq flip} for the first inequality then $\hat\eta(\R)\tilde\eta=\hat\eta\le\eta$ for the second, we have
  \begin{align*}
   \hat\eta(\R)u\le \hat\eta(\R)F_{\tilde\eta}(F_{\tilde\eta}^{-1}(u))\le F_\eta(F_{\tilde\eta}^{-1}(u))
  \end{align*}
  so that $F_{\tilde\eta}^{-1}(u)\ge F_\eta^{-1}(\hat\eta(\R)u)$. With the monotonicity of $f$, we deduce that
  $$\int_\R f(x)\hat \eta(dx)\ge \hat\eta(\R)\int_0^1f(F_\eta^{-1}(\hat\eta(\R)u))du=\int_0^{\hat\eta(\R)}f(F_\eta^{-1}(v))dv.$$
  On the other hand, still for $u\in(0,1)$, using $\tilde\eta\le\frac\eta{\hat\eta(\R)}$ for the first inequality and \eqref{eq:qt ineq flip} for the second,
  \begin{align*}
   1-F_{\tilde\eta}(F_\eta^{-1}(1-\hat\eta(\R)u))&=\tilde\eta((F_\eta^{-1}(1-\hat\eta(\R)u),+\infty))\le \frac{\eta((F_\eta^{-1}(1-\hat\eta(\R)u),+\infty))}{\hat\eta(\R)}\\&=\frac{1-F_\eta(F_\eta^{-1}(1-\hat\eta(\R)u))}{\hat\eta(\R)}\le u
  \end{align*}
  so that $F_{\tilde\eta}^{-1}(1-u)\le F_\eta^{-1}(1-\hat\eta(\R)u)$. We conclude that
  $$\int_\R f(x)\hat \eta(dx)=\hat\eta(\R)\int_0^1f(F_{\tilde\eta}^{-1}(1-u))du\le \hat\eta(\R)\int_0^1f(F_\eta^{-1}(1-\hat\eta(\R)u))du=\int_{1-\hat\eta(\R)}^1f(F_\eta^{-1}(v))dv.$$
\end{proof}
\begin{proof}[Proof of Lemma \ref{lem:stricstoch}]
  Since the image of the Lebesgue measure on $(0,1)$ by $(F_\mu^{-1},F_\nu^{-1})$ belongs to $\Pi(\mu,\nu)$, the sufficient condition is obvious. To prove the necessary condition, we let $(X,Y)$ be distributed according to $\pi\in \Pi(\mu,\nu)$ such that $\pi\left(\{(x,y)\in\R^2:x>y\}\right)=1$. We have $\P(X>Y)=1$ and, since $\mathbb Q$ is dense in $\R$, $\{X>Y\}=\bigcup_{q\in {\mathbb Q}}\{X\ge q>Y\}$. Since ${\mathbb Q}$ is countable, there exists a bijection $\N\ni n\mapsto q_n\in{\mathbb Q}$. For $n\in\N$, let $A_n=\{X\ge q_n>Y\}\cap\left\{\bigcup_{k=0}^{n-1}\{X\ge q_k>Y\}\right\}^c$. The events $(A_n)_{n\in\N}$ are disjoint and such that $\sum_{n\in\N}\P(A_n)=\P(X>Y)=1$. The discrete random variable $Z=\sum_{n\in\N}q_n\mathds{1}_{A_n}$ is such that $\P(X\ge Z>Y)=1$ so that $\nu\le_{st}\eta\le_{st}\mu$, where $\eta$ denotes the distribution of $Z$. For $n\in\N$, we have $$F_{\eta}(q_n)=\P(Z\le q_n)=\P(Y<Z\le q_n)\le \P(Y<q_n)=F_{\nu}(q_n-).$$ We deduce that for all $u\in (F_{\eta}(q_n-),F_{\eta}(q_n))$, $F_\nu^{-1}(u)<q_n=F_{\eta}^{-1}(u)$. Since $$\int_0^1\mathds{1}_{\bigcup_{n\in\N}(F_{\eta}(q_n-),F_{\eta}(q_n))}(u)du=\sum_{n\in\N}\P(Z=q_n)=1,$$ and $F_{\eta}^{-1}\le F_\mu^{-1}$ as $\eta\le_{st}\mu$, the statement follows.
\end{proof}\begin{proof}[Proof of Theorem \ref{thm:inccoup}]
  Since $\mu\neq \nu$, the non vacuity of $\pi\in\Pi_M(\mu,\nu,\nu_l,\nu_r)$ implies that $\nu_l+\nu_r$ is not the zero measure. Since the equality of the means of $\mu$ and $\nu$ implies that $0=\int_{\R^2}(y-x)\mathds{1}_{\{y<x\}}\pi_x(dy)\mu(dx)+\int_{\R^2}(y-x)\mathds{1}_{\{y>x\}}\pi_x(dy)\mu(dx)$, then $\nu_l\neq 0$ and $\nu_r\neq 0$.

      Let $\pi\in\Pi_M(\mu,\nu,\nu_l,\nu_r)$. Since for $x<x'$,
\begin{align*}
   0\le \int_{\R^2}\mathds{1}_{\{z<y\le x'\}}\pi(dy,dz)-\int_{\R^2}\mathds{1}_{\{z<y\le x\}}\pi(dy,dz)\le \int_{\R^2}\mathds{1}_{\{x<y\le x'\}}\pi(dy,dz)=F_\mu(x')-F_\mu(x),
\end{align*}
there exists a function $\phi_\pi$ such that \begin{equation}
   \phi_\pi(F_\mu(x))=\int_{\R^2}\mathds{1}_{\{z<y\le x\}}\pi(dy,dz)\in[0,\nu_l(\R)]. \label{eq:defphipi}
\end{equation} One has, using $\nu^\pi_0(\R)=0$, $$\nu_r(\R)\ge \int_{\R^2}\mathds{1}_{\{y\le x,y<z\}}\pi(dy,dz)=\int_{\R^2}\mathds{1}_{\{y\le x\}}\pi(dy,dz)-\int_{\R^2}\mathds{1}_{\{z<y\le x\}}\pi(dy,dz)=F_\mu(x)-\phi_\pi(F_\mu(x)).$$
Let $\tilde\nu_l=\frac{\nu_l}{\nu_l(\R)}$, $\tilde\nu_r=\frac{\nu_r}{\nu_r(\R)}$ and for $u\in[0,1]$ and $v\in[(u-\nu_r(\R))^+,u\wedge\nu_l(\R)]$, $$G(u,v)=\int_0^{v}F_{\tilde \nu_l}^{-1}\left(\frac{w}{\nu_l(\R)}\right)dw+\int_0^{u-v}F_{\tilde \nu_r}^{-1}\left(\frac{w}{\nu_r(\R)}\right)dw.$$
We are going to prove the existence of a non-decreasing function $\phi_\uparrow:[0,1]\to[0,\nu_l(\R)]$ such that $u\mapsto u-\phi_\uparrow(u)$ is also non-decreasing and 
  \begin{equation}
    \label{eq:eqm}\forall u\in[0,1],\;\phi_\uparrow(u)\in [(u-\nu_r(\R))^+,u\wedge\nu_l(\R)]\mbox{ and }\int_0^uF_\mu^{-1}(w)dw=G(u,\phi_\uparrow(u)).\end{equation}
  The existence of $\phi_\uparrow(F_\mu(x))\in [(F_\mu(x)-\nu_r(\R))^+,\phi_\pi(F_\mu(x))]$ for $x\in\R$ is addressed in {\bf Step 1} below. {\bf Step 2} deals with the monotonicity of $x\mapsto \phi_\uparrow(F_\mu(x))$ and $x\mapsto F_\mu(x)-\phi_\uparrow(F_\mu(x))$. The existence of $\phi_\uparrow(u)$ for $u\in\bigcup_{x\in\R}[F_\mu(x-),F_\mu(x))$, and the monotonicity of $u\mapsto \phi_\uparrow(u)$ and $u\mapsto u-\phi_\uparrow (u)$ on $(0,1)$  are addressed in {\bf Step 3}, while \eqref{eq:maxintabs} is proved in {\bf Step 4}. When $F_\mu(x)>0$ for each $x\in\R$ (resp. $F_\mu(x)<1$ for each $x\in\R$), the function $\phi_\uparrow$ is continuously and monotonically extended by $\phi_\uparrow(0)=0$ since $\lim_{u\to 0+}(u-\nu_r(\R))^+=0=\lim_{u\to 0+}u\wedge\nu_l(\R)$ (resp. $\phi_\uparrow(1)=\nu_l(\R)$ since $\lim_{u\to 1-}(u-\nu_r(\R))^+=\nu_l(\R)=\lim_{u\to 1-}u\wedge\nu_l(\R)$) so that \eqref{eq:eqm} holds for $u=0$ (resp. $u=1$). We are now going to prove the existence of a non-decreasing coupling $\pi^\uparrow\in\Pi_M(\mu,\nu,\nu_l,\nu_r)$ such that $\phi_{\pi^\uparrow}=\phi_\uparrow$. 
  Because of the monotonicity of $[0,1]\ni u\mapsto \phi_\uparrow(u)$ and $[0,1]\ni u\mapsto u-\phi_\uparrow(u)$, these two functions are $1$-Lipschitz and therefore absolutely continuous. Up to modifying the derivative $\phi_\uparrow'$ of $\phi_\uparrow$ in the sense of distributions on a Lebesgue negligible subset of $(0,1)$, we may suppose that it is $[0,1]$-valued. The functions $[0,1]\ni u\mapsto\int_0^{\phi_\uparrow(u)}F_{\tilde \nu_l}^{-1}\left(\frac{w}{\nu_l(\R)}\right)dw$ and $[0,1]\ni u\mapsto\int_0^{u-\phi_\uparrow(u)}F_{\tilde \nu_r}^{-1}\left(\frac{w}{\nu_r(\R)}\right)dw$ are convex and continuous and therefore absolutely continuous. Using \cite[Lemma 1.2]{marmi} to differentiate \eqref{eq:eqm}, we obtain
  \begin{equation}
   F_\mu^{-1}(u)=\phi_\uparrow'(u)F_{\tilde \nu_l}^{-1}\left(\frac{\phi_\uparrow(u)}{\nu_l(\R)}\right)+(1-\phi_\uparrow'(u))F_{\tilde \nu_r}^{-1}\left(\frac{u-\phi_\uparrow(u)}{\nu_r(\R)}\right),\;du\mbox{ a.e. on }(0,1),\label{eq:martprop}
 \end{equation}
 with the convention that the first (resp. second) product in the right-hand side is $0$ when $\phi_\uparrow'(u)=0$ (resp. $\phi_\uparrow'(u)=1$). We set
\begin{align*}
  \pi^\uparrow(dx,dy)=\int_0^1\bigg(&\phi_\uparrow'(u)\delta_{\left(F_\mu^{-1}(u),F_{\tilde \nu_l}^{-1}\left(\frac{\phi_\uparrow(u)}{\nu_l(\R)}\right)\right)}(dx,dy)\\&+(1-\phi_\uparrow'(u))\delta_{\left(F_\mu^{-1}(u),F_{\tilde \nu_r}^{-1}\left(\frac{u-\phi_\uparrow(u)}{\nu_r(\R)}\right)\right)}(dx,dy)\bigg)du.
\end{align*}
We have $\int_{y\in\R}\pi^\uparrow(dx,dy)=\int_0^1\delta_{F_\mu^{-1}(u)}(dx)du$, so that, by the inverse transform sampling, the first marginal of $\pi^\uparrow$ is $\mu$. Using Lemma 2.6 in \cite{JoMa18} like in the proof of Proposition 2.3 in \cite{JoMa18}, one checks that $\pi^\uparrow$ is a martingale coupling. We are next going to check that\begin{align}
\phi_\uparrow'(u)du\mbox{ a.e. on }(0,1),\;&F_{\tilde\nu_l}^{-1}\left(\frac{\phi_\uparrow(u)}{\nu_l(\R)}\right)<F_{\mu}^{-1}\left(u\right)\label{eq:compg}\\
\mbox{and }                                                                     (1-\phi_\uparrow'(u))du\mbox{ a.e. on }(0,1),\;&F_{\mu}^{-1}\left(u\right)<F_{\tilde \nu_r}^{-1}\left(\frac{u-\phi_\uparrow(u)}{\nu_r(\R)}\right).\label{eq:compd}                       \end{align}
With the definition of $\pi^\uparrow$, we deduce that $\nu^{\pi^\uparrow}_l=\int_0^1\phi_\uparrow'(u)\delta_{F_{\tilde \nu_l}^{-1}\left(\frac{\phi_\uparrow(u)}{\nu_l(\R)}\right)}du$. For $g:[0,1]\to\R$ measurable and bounded, by Proposition 4.9 \cite{Revuz-Yor} and the remark just above applied with $f(u)=g\left(\frac{\phi_\uparrow(u)}{\nu_l(\R)}\right)$ and $A_u=\frac{\phi_\uparrow(u)}{\nu_l(\R)}$, \begin{equation}
   \int_0^1g\left(\frac{\phi_\uparrow(u)}{\nu_l(\R)}\right)\frac{\phi'_\uparrow(u)}{\nu_l(\R)}du=\int_0^1g(v)dv.\label{eq:chgtvar}
\end{equation}
Hence $\nu^{\pi^\uparrow}_l=\nu_l(\R)\int_0^{1}\delta_{F_{\tilde \nu_l}^{-1}\left(w\right)}dw
  =\nu_l(\R)\tilde\nu_l=\nu_l$. In the same way, $\nu^{\pi^\uparrow}_r=\int_0^1(1-\phi_\uparrow'(u))\delta_{F_{\tilde \nu_r}^{-1}\left(\frac{u-\phi_\uparrow(u)}{\nu_r(\R)}\right)}du=\nu_r$, 
so that, since $\nu=\nu_l+\nu_r$, $\pi^\uparrow\in\Pi_M(\mu,\nu,\nu_l,\nu_r)$. Let \begin{align*}
   \Gamma=&\left\{\left\{\left(F_\mu^{-1}(u),F_{\tilde\nu_l}^{-1}\left(\frac{\phi_\uparrow(u)}{\nu_l(\R)}\right)\right):u\in\R\right\}\cap\{(x,y)\in\R^2:y<x\}\right\}\\&\cup\left\{\left\{\left(F_\mu^{-1}(u),F_{\tilde\nu_r}^{-1}\left(\frac{u-\phi_\uparrow(u)}{\nu_r(\R)}\right)\right):u\in\R\right\}\cap\{(x,y)\in\R^2:x<y\}\right\}.
                                                                                   \end{align*}
The definition of $\pi^\uparrow$ combined with \eqref{eq:compg} and \eqref{eq:compd} ensures that $\pi^\uparrow(\Gamma)=1$. By monotonicity of $u\mapsto F_\mu^{-1}(u)$, $u\mapsto F_{\tilde\nu_l}^{-1}\left(\frac{\phi_\uparrow(u)}{\nu_l(\R)}\right)$ and $u\mapsto F_{\tilde\nu_r}^{-1}\left(\frac{u-\phi_\uparrow(u)}{\nu_r(\R)}\right)$, $\Gamma$ satisfies conditions $(a)$ and $(b)$ in Definition \ref{def:NdCoupling}. Hence $\pi^\uparrow$ is non-decreasing.                                                                             


To check \eqref{eq:compg}, we introduce $\tilde\mu_l(dy)=\frac{1}{\nu_l(\R)}\int_{z\in\R}\mathds{1}_{\{z<y\}}\pi(dy,dz)
$ 
. As $\frac{\mathds{1}_{\{z<y\}}}{\nu_l(\R)}\pi(dy,dz)\in\Pi(\tilde\mu_l,\tilde\nu_l)$, Lemma \ref{lem:stricstoch} implies that $du$ a.e., $F_{\tilde\nu_l}^{-1}(u)<F_{\tilde\mu_l}^{-1}(u)$.  With \eqref{eq:chgtvar}, we deduce that \begin{equation}
   \phi_\uparrow'(u)du\mbox{ a.e. on }(0,1),\;F_{\tilde\nu_l}^{-1}\left(\frac{\phi_\uparrow(u)}{\nu_l(\R)}\right)<F_{\tilde\mu_l}^{-1}\left(\frac{\phi_\uparrow(u)}{\nu_l(\R)}\right).\label{eq:fnumull}
\end{equation} For $x\in\R$, by \eqref{eq:defphipi}, $F_{\tilde\mu_l}(x)=\frac{\phi_\pi(F_\mu(x))}{\nu_l(\R)}$ and since $F_\mu(F_\mu^{-1}(F_\mu(x)))=F_\mu(x)$, $\phi_\pi(F_\mu(F_\mu^{-1}(F_\mu(x))))=\phi_\pi(F_\mu(x))$ and $F_{\tilde\mu_l}^{-1}\left(\frac{\phi_\pi(F_\mu(x))}{\nu_l(\R)}\right)\le F_\mu^{-1}(F_\mu(x))$. Since, by {\bf Step 1}, $\phi_\uparrow(F_\mu(x))\le \phi_\pi(F_\mu(x))$ and $F_{\tilde\mu_l}^{-1}$ is non-increasing, we deduce that $F_{\tilde\mu_l}^{-1}\left(\frac{\phi_\uparrow(F_\mu(x))}{\nu_l(\R)}\right)\le F_\mu^{-1}(F_\mu(x))$.   As $(0,1)\cap\left\{\bigcup_{x\in\R:\mu(\{x\})>0}[F_\mu(x-),F_\mu(x))\right\}^c$ is included in the range of $F_\mu$, with \eqref{eq:fnumull} we conclude that \begin{equation*}
   \phi_\uparrow'(u)du\mbox{ a.e. on }(0,1)\cap\left\{\bigcup_{x\in\R:\mu(\{x\})>0}[F_\mu(x-),F_\mu(x))\right\}^c,\;F_{\tilde\nu_l}^{-1}\left(\frac{\phi_\uparrow(u)}{\nu_l(\R)}\right)<F_{\mu}^{-1}\left(u\right).
\end{equation*}
If $\mu(\{x\})>0$, then by {\bf Step 3} below, $\phi_\uparrow'(u)du$ a.e. on $(F_\mu(x-),F_\mu(x))$, $F_{\tilde \nu_l}^{-1}\left(\frac{\phi_\uparrow(u)}{\nu_l(\R)}\right)<F_\mu^{-1}(u)$ and therefore \eqref{eq:compg} holds.

In a symmetric way, since $\frac{\mathds{1}_{\{y<z\}}}{\nu_r(\R)}\pi(dy,dz)\in \Pi(\tilde \mu_r,\tilde\nu_r)$ where $\tilde\mu_r(dy)=\frac 1{\nu_r(\R)}\int_{z\in\R}\mathds{1}_{\{y<z\}}\pi(dy,dz)$, Lemma \ref{lem:stricstoch} implies that $du$ a.e., $F_{\tilde\nu_r}^{-1}(u)>F_{\tilde\mu_r}^{-1}(u)$. For $x\in\R$, $F_{\tilde \mu_r}(x)=\frac{F_\mu(x)-\phi_\pi(F_\mu(x))}{\nu_r(\R)}$. Hence for $w>\frac{F_\mu(x-)-\lim_{y\to x-}\phi_\pi(F_\mu(y))}{\nu_r(\R)}$, $F^{-1}_{\tilde \mu_r}(w)\ge x$. Since $\phi_\uparrow(F_\mu(x-))\le \lim_{y\to x-}\phi_\pi(F_\mu(y))$, we deduce that when $\mu(\{x\})>0$, $F_{\tilde\nu_r}^{-1}(w)>x$ for $w>\frac{F_\mu(x-)-\phi_\uparrow(F_\mu(x-))}{\nu_r(\R)}$. With this argument replacing the above reference to {\bf Step 3} when $\mu(\{x\})>0$, we check that \eqref{eq:compd} holds.
                          
                        {\bf Uniqueness :}  Let  $\hat\pi^\uparrow \in\Pi_M(\mu,\nu,\nu_l,\nu_r)$ be non-decreasing. Let $\Gamma$ be a Borel subset of $\R^2$ such that $\hat\pi^\uparrow(\Gamma)=1$ and properties $(a)(b)$ in Definition \ref{def:NdCoupling} are satisfied. Let $x\in \R$ be such that $F_\mu(x)>0$. The set $\{z\in (-\infty,x):\exists y\in (z,x]\mbox{ s.t. }(y,z)\in \Gamma\}$ is not empty since $$\int_{\R^2}\mathds{1}_{\Gamma}(y,z)\mathds{1}_{\{z<y\le x\}}\hat\pi^\uparrow(dy,dz)=\int_{\R^2}\mathds{1}_{\{z<y\le x\}}\hat\pi^\uparrow(dy,dz)=\int_{(-\infty,x]} \hat\pi^\uparrow_y((-\infty,y))\mu(dy)>0,$$
                           as $\nu^{\hat\pi^\uparrow}_0=0$ 
                           implies that $\mu(dy)$ a.e. $\hat\pi^\uparrow_y((-\infty,y))>0$. Therefore $l(x):=\sup\{z\in (-\infty,x):\exists y\in (z,x]\mbox{ s.t. }(y,z)\in \Gamma\}$ belongs to $(-\infty,x]$. Since $\Gamma\cap\{(y,z)\in\R^2:z<y\le x\mbox{ and }z>l(x)\}=\emptyset$,
                           $$0=\int_{\R^2} \mathds{1}_{\Gamma}(y,z)\mathds{1}_{\{z<y\le x\}}\mathds{1}_{\{z>l(x)\}}\hat\pi^\uparrow(dy,dz)=\int_{\R^2} \mathds{1}_{\{z<y\le x\}}\mathds{1}_{\{z>l(x)\}}\hat\pi^\uparrow(dy,dz).$$On the other hand, by property $(a)$ in Definition \ref{def:NdCoupling}, $\Gamma\cap (x,+\infty)\times (-\infty,l(x))=\emptyset$ so that
$$0=\int_{\R^2} \mathds{1}_{\Gamma}(y,z)\mathds{1}_{\{y>x\}}\mathds{1}_{\{z<l(x)\}}\hat\pi^\uparrow(dy,dz)=\int_{\R^2} \mathds{1}_{\{y>x\}}\mathds{1}_{\{z<l(x)\}}\hat\pi^\uparrow(dy,dz).$$                        
                           With the equality, $\frac{1}{\nu_l(\R)}\int_{y\in \R}\mathds{1}_{\{z<y\}}\hat\pi^\uparrow(dy,dz)=\tilde \nu_l(dz)$, we deduce that $$\mathds{1}_{\{z<l(x)\}}\tilde \nu_l(dz)\le \frac{1}{\nu_l(\R)}\int_{y\in (-\infty,x]}\mathds{1}_{\{z<y\}}\hat\pi^\uparrow(dy,dz)\le \mathds{1}_{\{z\le l(x)\}}\tilde \nu_l(dz).$$ Therefore $\frac{\phi_{\hat\pi^\uparrow}(F_\mu(x))}{\nu_l(\R)}\in[F_{\tilde \nu_l}(l(x)-),F_{\tilde \nu_l}(l(x))]$ and
 \begin{align}
   &\frac{1}{\nu_l(\R)}\int_{y\in (-\infty,x]}\mathds{1}_{\{z<y\}}\hat\pi^\uparrow(dy,dz)=\mathds{1}_{\{z<l(x)\}}\tilde \nu_l(dz)+\left(\frac{\phi_{\hat\pi^\uparrow}(F_\mu(x))}{\nu_l(\R)}-F_{\tilde \nu_l}(l(x)-)\right)\delta_{l(x)}(dz)\notag\\&=\int^{\frac{\phi_{\hat\pi^\uparrow}(F_\mu(x))}{\nu_l(\R)}}_0\delta_{F_{\tilde \nu_l}^{-1}\left(u\right)}(dz)du=\frac{1}{\nu_l(\R)}\int_0^{\phi_{\hat\pi^\uparrow}(F_\mu(x))}\delta_{F_{\tilde \nu_l}^{-1}\left(\frac{w}{\nu_l(\R)}\right)}(dz)dw.\label{eq:mesgauche}
 \end{align}
                                                                                                                                                                                                   In a symmetric way, with property $(b)$ in Definition \ref{def:NdCoupling},  we check that \begin{equation}
                                                                                                       \int_{y\in (-\infty,x]}\mathds{1}_{\{z>y\}}\hat\pi^\uparrow(dy,dz)=\int_0^{F_\mu(x)-\phi_{\hat\pi^\uparrow}(F_\mu(x))}\delta_{F_{\tilde \nu_r}^{-1}\left(\frac{w}{\nu_r(\R)}\right)}(dz)dw.\label{eq:mesdroite}\end{equation}
                                                                                                                                                                                                   Using the martingale property of $\hat\pi^\uparrow$ for the second equality and $\nu^{\hat\pi^\uparrow}_0=0$ 
                                                                                                                                                                                                   for the fourth, we deduce that
 \begin{align*}
   \int_0^{F_\mu(x)}F_\mu^{-1}(u)du&=\int_{y\in(-\infty,x]}y\mu(dy)=\int_{y\in(-\infty,x]}\int_{z\in\R}z\hat\pi^\uparrow_y(dz)\mu(dy)=\int_{(-\infty,x]\times \R}z\hat\pi^\uparrow(dy,dz)\\&=\int_{(-\infty,x]\times \R}z\mathds{1}_{\{z<y\}}\hat\pi^\uparrow(dy,dz)+\int_{(-\infty,x]\times \R}z\mathds{1}_{\{z>y\}}\hat\pi^\uparrow(dy,dz)\\&=
  \int_0^{\phi_{\hat\pi^\uparrow}(F_\mu(x))}F_{\tilde \nu_l}^{-1}\left(\frac{w}{\nu_l(\R)}\right)dw+\int_0^{F_\mu(x)-\phi_{\hat\pi^\uparrow}(F_\mu(x))}F_{\tilde \nu_r}^{-1}\left(\frac{w}{\nu_r(\R)}\right)dw\\&=G(F_\mu(x),\phi_{\hat\pi^\uparrow}(F_\mu(x))).\end{align*}
Equations \eqref{eq:mesgauche} and \eqref{eq:mesdroite} also hold with $\hat\pi^\uparrow$ replaced by $\pi^\uparrow$ so that $G(F_\mu(x),\phi_{\pi^\uparrow}(F_\mu(x)))=\int_0^{F_\mu(x)}F_\mu^{-1}(u)du$. Note that the equalities  \begin{align*}&\int_{y\in (-\infty,x]}\mathds{1}_{\{z<y\}}\pi^\uparrow(dy,dz)=\int_0^{\phi_{\pi^\uparrow}(F_\mu(x))}\delta_{F_{\tilde \nu_l}^{-1}\left(\frac{w}{\nu_l(\R)}\right)}(dz)dw\\&\int_{y\in (-\infty,x]}\mathds{1}_{\{z>y\}}\pi^\uparrow(dy,dz)=\int_0^{F_\mu(x)-\phi_{\pi^\uparrow}(F_\mu(x))}\delta_{F_{\tilde \nu_r}^{-1}\left(\frac{w}{\nu_r(\R)}\right)}(dz)dw,
   \end{align*}together with $\phi_{\pi^\uparrow}(F_\mu(x))=\int_0^{F_\mu(x)}\phi'_{\uparrow}(u)du=\phi_{\uparrow}(F_\mu(x))$ can also be derived from the definition of $\pi^\uparrow$ combined with \eqref{eq:compg}, \eqref{eq:compd}, \eqref{eq:qt flip}, Proposition 4.9 \cite{Revuz-Yor} and the remark just above applied like in the above derivation of \eqref{eq:chgtvar}.

                                                                                                                                                                                                   By {\bf Step 1} below applied with $\pi=\hat\pi^\uparrow$ and $\pi=\pi^\uparrow$, there exists a unique $v\in[(F_\mu(x)-\nu_r(\R))^+,\phi_{\hat\pi^\uparrow}(F_\mu(x))\vee\phi_{\pi^\uparrow}(F_\mu(x))]$ such that $G(F_\mu(x),v)=\int_0^{F_\mu(x)}F_\mu^{-1}(w)dw$ so that $\phi_{\hat\pi^\uparrow}(F_\mu(x))=\phi_{\pi^\uparrow}(F_\mu(x))$ for each $x\in\R$ and
\begin{align*}
   \int_{y\in (-\infty,x]}\hat\pi^\uparrow(dy,dz)&=\int_0^{\phi_{\pi^\uparrow}(F_\mu(x))}\delta_{F_{\tilde \nu_l}^{-1}\left(\frac{w}{\nu_l(\R)}\right)}(dz)dw+\int_0^{F_\mu(x)-\phi_{\pi^\uparrow}(F_\mu(x))}\delta_{F_{\tilde \nu_r}^{-1}\left(\frac{w}{\nu_r(\R)}\right)}(dz)dw\\&=\int_{y\in (-\infty,x]}\pi^\uparrow(dy,dz).
\end{align*}
Therefore the set\[
\mathcal{M} := \left\{A\in\mathcal{B(\R)}\mid \int_{A} \pi^\uparrow_y\mu(dy)=\int_{A} \hat\pi^\uparrow_y\mu(dy)\right\}\] contains $\left\{(-\infty,x]: x\in\R\right\}$.
                           One easily checks that if $A,B\in\mathcal M$ are such that $B\subset A$, then $A\setminus B\in\mathcal M$. Now let $\left(A_n\right)_{n\in\N}$ be a sequence of elements in $\mathcal M$ such that $A_n\subseteq A_{n+1}$. For $B\in\mathcal{B}(\R)$, since for all $n\in\N$, 
    $\int_{A_n} \left(\hat\pi^\uparrow_y(B)-\pi^\uparrow_y(B)\right)\,\mu(dy) = 0$
    and $\vert\hat\pi^\uparrow_y(B)-\pi^\uparrow_y(B)\vert \leq 1$, 
    by Lebesgue's dominated convergence theorem, 
                           \[\int_{\bigcup_{n\in\N}A_n} \left(\hat\pi^\uparrow_y(B)-\pi^\uparrow_y(B)\right)\,\mu(dy) = 0.\]
                           We deduce that $\bigcup_{n\in\N}A_n\in\mathcal{M}$ and $\mathcal{M}$ is a monotone class. Since $\mathcal M$ contains the class $\mathcal C = \left\{(-\infty,x]: x\in\R\right\}$, which is stable by finite intersections, by the monotone class theorem, we have $\sigma(\mathcal C) \subset \mathcal M$. Since $\sigma(\mathcal C) = \mathcal B(\R)$ and $\mathcal M \subseteq \mathcal B(\R)$, we conclude that $\mathcal M =\mathcal B(\R)$.
In particular, for any $B\in\mathcal B(\R)$, \begin{align*}
   \mu&\left(\left\{y\in\R:\hat\pi^\uparrow_y(B)>\pi^\uparrow_y(B)\right\}\right)=0=\mu\left(\left\{y\in\R:\hat\pi^\uparrow_y(B)<\pi^\uparrow_y(B)\right\}\right).
\end{align*} We deduce that \begin{align*}
   \mu&\left(\left\{y\in\R:\forall x\in{\mathbb Q},\; \hat\pi^\uparrow_y((-\infty,x])=\pi^\uparrow_y((-\infty,x])\right\}\right)=1=\mu\left(\left\{y\in\R:\hat\pi^\uparrow_y=\pi^\uparrow_y\right\}\right).
\end{align*}

{\bf Step 1 :}  
For $x\in\R$, we are going to check that $G(F_\mu(x),\phi_\pi(F_\mu(x)))\le \int_0^{F_\mu(x)}F_\mu^{-1}(w)dw$ ({\bf Step 1.1} below), $G(F_\mu(x),(F_\mu(x)-\nu_r(\R))^+)\ge \int_0^{F_\mu(x)}F_\mu^{-1}(w)dw$ ({\bf Step 1.2} below) and that $v\mapsto G(F_\mu(x),v)$ is decreasing on $[(F_\mu(x)-\nu_r(\R))^+,\phi_\pi(F_\mu(x))]$ ({\bf Step 1.3} below). Since this function is also continuous, this implies the existence of a unique $\phi_\uparrow(F_\mu(x))\in [(F_\mu(x)-\nu_r(\R))^+,\phi_\pi(F_\mu(x))]$ such that $G(F_\mu(x),\phi_\uparrow(F_\mu(x)))=\int_0^{F_\mu(x)}F_\mu^{-1}(v)dv$, where the notation $\phi_\uparrow(F_\mu(x))$ is justified since this point clearly only depends on $x$ through $F_\mu(x)$.

{\bf Step 1.1 :} By Lemma \ref{lem:compmom},  applied with $\eta=\tilde\nu_l$ and $\hat\eta(dz)= \frac 1{\nu_l(\R)}\int_{y\in (-\infty,x]}\mathds{1}_{\{z<y\}}\pi(dy,dz)$ such that $\hat\eta(\R)=\frac{\phi_\pi(F_\mu(x))}{\nu_l(\R)}$, $$\frac 1{\nu_l(\R)}\int_{\R^2}z\mathds{1}_{\{z<y\le x\}}\pi(dy,dz)\ge \int_0^{\frac{\phi_\pi(F_\mu(x))}{\nu_l(\R)}}F_{\tilde \nu_l}^{-1}(w)dw=\frac 1{\nu_l(\R)}\int_0^{\phi_\pi(F_\mu(x))}F_{\tilde \nu_l}^{-1}\left(\frac{w}{\nu_l(\R)}\right)dw$$ and, with $\eta=\tilde\nu_r$, $\hat\eta(dz)=\frac 1{\nu_r(\R)}\int_{y\in (-\infty,x]}\mathds{1}_{\{z>y\}}\pi(dy,dz)\le \tilde\nu_r(dz)$ such that $\hat\eta(\R)=\frac{F_\mu(x)-\phi_\pi(F_\mu(x))}{\nu_r(\R)}$, $$\frac{1}{\nu_r(\R)}\int_{\R^2}z\mathds{1}_{\{y\le x,y<z\}}\pi(dy,dz)\ge 
\int_0^{\frac{F_\mu(x)-\phi_\pi(F_\mu(x))}{\nu_r(\R)}}F_{\tilde \nu_r}^{-1}(w)dw=\frac{1}{\nu_r(\R)}\int_0^{F_\mu(x)-\phi_\pi(F_\mu(x))}F_{\tilde \nu_r}^{-1}\left(\frac{w}{\nu_r(\R)}\right)dw.$$ Using \eqref{eq:qt flip} for the first equality, the inverse transform sampling for the second and the martingale property of $\pi$, and $\nu^\pi_l+\nu^\pi_r=\nu$ for the third, we deduce that
\begin{align*}
  \int_0^{F_\mu(x)}F_\mu^{-1}(w)dw&=\int_0^{1}F_\mu^{-1}(w)\mathds{1}_{\{F_\mu^{-1}(w)\le x\}}dw=\int_{\R}y\mathds{1}_{\{y\le x\}}\mu(dy)
  \\&=\int_{\R^2}z\mathds{1}_{\{z<y\le x\}}\pi(dy,dz)+\int_{\R^2}z\mathds{1}_{\{y\le x,y<z\}}\pi(dy,dz)\ge G(F_\mu(x),\phi_\pi(F_\mu(x))).
\end{align*}

{\bf Step 1.2 :} If $F_\mu(x)< \nu_r(\R)$, 
setting $$\psi(z)=\mathds{1}_{\{z<F_{\tilde \nu_r}^{-1}(\frac{F_\mu(x)}{\nu_r(\R)})\}}+\mathds{1}_{\{z=F_{\tilde \nu_r}^{-1}(\frac{F_\mu(x)}{\nu_r(\R)})\}}\mathds{1}_{\{\tilde \nu_r(\{F_{\tilde \nu_r}^{-1}(\frac{F_\mu(x)}{\nu_r(\R)})\})>0\}}\frac{\frac{F_\mu(x)}{\nu_r(\R)}-\tilde \nu_r((-\infty,F_{\tilde \nu_r}^{-1}(\frac{F_\mu(x)}{\nu_r(\R)})))}{\tilde \nu_r(\{F_{\tilde \nu_r}^{-1}(\frac{F_\mu(x)}{\nu_r(\R)})\})},$$ we have that
\begin{align*}
   \int_0^{F_\mu(x)}F_{\tilde \nu_r}^{-1}\left(\frac{w}{\nu_r(\R)}\right)dw&=\nu_r(\R)\int_\R z\psi(z)\tilde\nu_r(dz)
  =\int_{\R\times \R}z\psi(z)\mathds{1}_{\{y<z\}}\pi(dy,dz)\\&\ge \int_{\R\times \R}y\psi(z)\mathds{1}_{\{y<z\}}\pi(dy,dz)=\int_\R y\hat\mu(dy)\mbox{ where }
\end{align*}
$\hat\mu(dy):=\left(\int_{z\in \R}\psi(z)\mathds{1}_{\{y<z\}}\pi_y(dz)\right)\mu(dy)\le \mu(dy)$ since $\int_{z\in \R}\psi(z)\mathds{1}_{\{y<z\}}\pi_y(dz)\in[0,1]$ and satisfies $$\hat\mu(\R)=\int_{\R\times \R}\psi(z)\mathds{1}_{\{y<z\}}\pi(dy,dz)=\nu_r(\R)\int_\R \psi(z)\tilde\nu_r(dz)=F_\mu(x).$$
By Lemma \ref{lem:compmom} applied with $(\eta,\hat\eta)=(\mu,\hat\mu)$, this ensures that $\int_\R y\hat\mu(dy)\ge \int_0^{F_\mu(x)}F_\mu^{-1}(v)dv$ so that $$\int_0^{F_\mu(x)}F_\mu^{-1}(v)dv\le 
\int_0^{F_\mu(x)}F_{\tilde \nu_r}^{-1}\left(\frac{w}{\nu_r(\R)}\right)dw=G(F_\mu(x),0).$$

If $F_\mu(x)\ge \nu_r(\R)$, one has using $\int_0^{\nu_r(\R)}F_{\tilde \nu_r}^{-1}\left(\frac{w}{\nu_r(\R)}\right)dw=\int_{\R}y\nu_r(dy)=\int_{\R}y\mu(dy)-\int_{\R}y\nu_l(dy)$ and $\nu_r(\R)=1-\nu_l(\R)$ for the second equality,

\begin{align*}
   G(F_\mu(x)&,F_\mu(x)-\nu_r(\R))=\int_0^{F_\mu(x)-\nu_r(\R)}F_{\tilde \nu_l}^{-1}\left(\frac{w}{\nu_l(\R)}\right)dw+\int_0^{\nu_r(\R)}F_{\tilde \nu_r}^{-1}\left(\frac{w}{\nu_r(\R)}\right)dw\\&
=\int_0^{\nu_l(\R)-(1-F_\mu(x))}F_{\tilde \nu_l}^{-1}\left(\frac{w}{\nu_l(\R)}\right)dw+\int_\R y\mu(dy)-\int_\R y\nu_l(dy)\\&=\int_0^{F_\mu(x)}F_\mu^{-1}(w)dw+\int_{F_\mu(x)}^1F_\mu^{-1}(w)dw
-\int^{\nu_l(\R)}_{\nu_l(\R)-(1-F_\mu(x))}F_{\tilde \nu_l}^{-1}\left(\frac{w}{\nu_l(\R)}\right)dw.
\end{align*}
By a reasoning similar to the above derivation of $\int_0^{F_\mu(x)}F_\mu^{-1}(v)dv\le \int_0^{F_\mu(x)}F_{\tilde \nu_r}^{-1}\left(\frac{w}{\nu_r(\R)}\right)dw$, we check that the sum of the last two terms in the right-hand side is non-negative.
Combining the two cases, we obtain that  $G(F_\mu(x),(F_\mu(x)-\nu_r(\R))^+)\ge \int_0^{F_\mu(x)}F_\mu^{-1}(w)dw$. 

{\bf Step 1.3 :}
To prove that $v\mapsto G(F_\mu(x),v)$ is decreasing on $[(F_\mu(x)-\nu_r(\R))^+,\phi_\pi(F_\mu(x))]$, it is enough to check that \begin{equation}
   \forall w\in\left(0,\frac{\phi_\pi(F_\mu(x))}{\nu_l(\R)}\right),\;F_{\tilde \nu_l}^{-1}\left(w\right)<x\mbox{ and }\forall w\in \left(\frac{F_\mu(x)-\phi_\pi(F_\mu(x))}{\nu_r(\R)},1\right),\;x<F_{\tilde \nu_r}^{-1}\left(w\right).\label{eq:compquantx}
 \end{equation}
 Let $x$ be such that $\phi_\pi(F_\mu(x))>0$. Then $F_\mu(x)>0$ and $F_\mu(x)-\phi_\pi(F_\mu(x))>0$. Since, by definition of $\phi_\pi(F_\mu(x))$ and $\nu^\pi_l=\nu_l$,
\begin{align*}
   \phi_\pi(F_\mu(x))&=\int_{\R^2}\mathds{1}_{\{z<y\le x\}}\pi(dy,dz)=\int_{\R^2}\mathds{1}_{\{z<y\wedge x\}}\pi(dy,dz)-\int_{\R^2}\mathds{1}_{\{z<x<y\}}\pi(dy,dz)\\&=\nu_l(\R)\tilde\nu_l((-\infty,x))-\int_{\R^2}\mathds{1}_{\{z<x<y\}}\pi(dy,dz),
\end{align*}
we have $\tilde\nu_l((-\infty,x))\ge \frac{\phi_\pi(F_\mu(x))}{\nu_l(\R)}$, from which we deduce the first part in \eqref{eq:compquantx}.
In a symmetric way and since $\nu^\pi_0=\nu_0=0$,
\begin{align*}
    F_\mu(x)-\phi_\pi(F_\mu(x))&=\int_{\R^2}\mathds{1}_{\{y\le x\}}\pi(dy,dz)-\int_{\R^2}\mathds{1}_{\{z<y\le x\}}\pi(dy,dz)\\&=\int_{\R^2}\mathds{1}_{\{y<z\le x\}}\pi(dy,dz)+\int_{\R^2}\mathds{1}_{\{y\le x<z\}}\pi(dy,dz)\\&=\nu_r(\R)F_{\tilde\nu_r}(x)+\int_{\R^2}\mathds{1}_{\{y\le x<z\}}\pi(dy,dz),
\end{align*}
so that $F_{\tilde\nu_r}(x)\le \frac{F_\mu(x)-\phi_\pi(F_\mu(x))}{\nu_r(\R)}$. With the right-continuity of $F_{\tilde\nu_r}$, we deduce the other part in \eqref{eq:compquantx}.

{\bf Step 2 :} monotonicity of $x\mapsto \phi_\uparrow(F_\mu(x))$ and $x\mapsto F_\mu(x)-\phi_\uparrow(F_\mu(x))$.

{\bf Step 2.1 :} Let us first check that when $F_\mu(x)<F_\mu(x')$, $F_\mu(x)-\phi_\uparrow(F_\mu(x))\le F_\mu(x')-\phi_\uparrow(F_\mu(x'))$. If $\phi_\pi(F_\mu(x'))\le \phi_\uparrow(F_\mu(x))+F_\mu(x')-F_\mu(x)$, then we conclude with the inequality $\phi_\uparrow(F_\mu(x'))\le \phi_\pi(F_\mu(x'))$ established in {\bf step 1}. Let us now suppose that $\phi_\uparrow(F_\mu(x))+F_\mu(x')-F_\mu(x)\le \phi_\pi(F_\mu(x'))$. Then, by the monotonicity of $F_{\tilde \nu_l}^{-1}$,
\begin{align*}
   \int_{\phi_\uparrow(F_\mu(x))}^{\phi_\uparrow(F_\mu(x))+F_\mu(x')-F_\mu(x)}F_{\tilde \nu_l}^{-1}\left(\frac{w}{\nu_l(\R)}\right)dw&\le\int_{\phi_\pi(F_\mu(x'))+F_\mu(x)-F_\mu(x')}^{\phi_\pi(F_\mu(x'))}F_{\tilde \nu_l}^{-1}\left(\frac{w}{\nu_l(\R)}\right)dw.
\end{align*}
We are going to prove that \begin{equation}
   \int_{\phi_\pi(F_\mu(x'))+F_\mu(x)-F_\mu(x')}^{\phi_\pi(F_\mu(x'))}F_{\tilde \nu_l}^{-1}\left(\frac{w}{\nu_l(\R)}\right)dw\le \int_{F_\mu(x)}^{F_\mu(x')}F_{\mu}^{-1}\left(w\right)dw.\label{eq:compdur}
 \end{equation}Adding the two inequalities to $G(F_\mu(x),\phi_\uparrow(F_\mu(x)))=\int_0^{F_\mu(x)}F_{\mu}^{-1}\left(w\right)dw$, we obtain that $$G(F_\mu(x'),\phi_\uparrow(F_\mu(x))+F_\mu(x')-F_\mu(x))\le \int_{0}^{F_\mu(x')}F_{\mu}^{-1}\left(w\right)dw.$$ Since by {\bf Step 1.3}, $v\mapsto G(F_\mu(x'),v)$ is decreasing on $[(F_\mu(x')-\nu_r(\R))^+,\phi_\pi(F_\mu(x'))]$, we deduce that $\phi_\uparrow(F_\mu(x'))\le \phi_\uparrow(F_\mu(x))+F_\mu(x')-F_\mu(x)$.

 To prove \eqref{eq:compdur}, we first remark that since $$F_\mu(x')-F_\mu(x)\le \phi_\pi(F_\mu(x'))=\int_0^{F_\mu(x')}\pi_{F_{\mu}^{-1}(w)}((-\infty,F_{\mu}^{-1}(w)))dw,$$ where we used the inverse transform sampling and \eqref{eq:qt flip} for the equality, \begin{equation}
   \exists u\in[0,F_\mu(x')]\mbox{ s.t. }\int_u^{F_\mu(x')}\pi_{F_{\mu}^{-1}(w)}((-\infty,F_{\mu}^{-1}(w)))dw=F_\mu(x')-F_\mu(x).\label{eq:defu}
 \end{equation}By Lemma \ref{lem:compmom}, applied with $\eta=\frac{1}{F_\mu(x')}\int_0^{F_\mu(x')}\delta_{F_\mu^{-1}(w)}dw$ such that $F_\eta^{-1}(w)=F_\mu^{-1}(F_\mu(x')w)$, $\hat\eta=\frac{1}{F_\mu(x')}\int_u^{F_\mu(x')}\pi_{F_{\mu}^{-1}(w)}((-\infty,F_{\mu}^{-1}(w)))\delta_{F_\mu^{-1}(w)}dw$ with mass
 $\hat\eta(\R)=\frac{F_\mu(x')-F_\mu(x)}{F_\mu(x')}$, \begin{equation}
   \frac{1}{F_\mu(x')}\int_{F_\mu(x)}^{F_\mu(x')}F_\mu^{-1}(w)dw
 \ge\frac{1}{F_\mu(x')
 }\int_u^{F_\mu(x')}\pi_{F_{\mu}^{-1}(w)}((-\infty,F_{\mu}^{-1}(w))){F_\mu^{-1}(w)}dw.\label{eq:minoet1}
 \end{equation} 
 

For the desintegrations $\pi(dy,dz)=\mu(dy)\pi_y(dz)=\nu(dz)\overleftarrow{\pi}_z(dy)$, let \begin{align*}
   \sigma(dy)&=\frac{\mathds{1}_{\{y\le x'\}}}{\phi_\pi(F_\mu(x'))}\pi_y((-\infty,y))\mu(dy)\mbox{ and }\\\theta(dz)&=\mathds{1}_{\{z<x'\}}\overleftarrow{\pi}_z((z,x'])\frac{\nu(dz)}{\phi_\pi(F_\mu(x'))}=\mathds{1}_{\{z<x'\}}\frac{\overleftarrow{\pi}_z((z,x'])}{\overleftarrow{\pi}_z((z,+\infty))}\times\frac{\nu_l(dz)}{\phi_\pi(F_\mu(x'))}.
\end{align*} The coupling $\frac{\mathds{1}_{\{z<y\le x'\}}}{\phi_\pi(F_\mu(x'))}\pi(dy,dz)\in\Pi\left(\sigma,\theta\right)$ giving full weight to $\{(y,z)\in\R^2:y>z\}$, one has $\sigma\ge_{st}\theta$. On the other hand, Lemma \ref{lem:compmom} applied with $\eta=\tilde\nu_l$ and $\hat\eta=\frac{\phi_\pi(F_\mu(x'))}{\nu_l(\R)}\theta$ such that $\hat\eta(\R)=\frac{\phi_\pi(F_\mu(x'))}{\nu_l(\R)}$ implies that $$\theta\ge_{st}\vartheta:=\frac{\nu_l(\R)}{\phi_\pi(F_\mu(x'))}\int_0^{\frac{\phi_\pi(F_\mu(x'))}{\nu_l(\R)}}\delta_{F_{\tilde\nu_l}^{-1}(w)}dw
.$$ Hence $\sigma\ge_{st}\vartheta$ so that $F_\sigma^{-1}\ge F_\vartheta^{-1}$ where $F_\vartheta^{-1}(w)=F_{\tilde\nu_l}^{-1}\left(\frac{\phi_\pi(F_\mu(x'))}{\nu_l(\R)}w\right)$ and
\begin{align*}
   \int_{1-\frac{F_\mu(x')-F_\mu(x)}{\phi_\pi(F_\mu(x'))}}^1F_\sigma^{-1}(w)dw&\ge \int_{1-\frac{F_\mu(x')-F_\mu(x)}{\phi_\pi(F_\mu(x'))}}^1F_\vartheta^{-1}(w)dw\\&=\frac{1}{\phi_\pi(F_\mu(x'))}\int_{\phi_\pi(F_\mu(x'))+F_\mu(x)-F_\mu(x')}^{\phi_\pi(F_\mu(x'))}F_{\tilde \nu_l}^{-1}\left(\frac{w}{\nu_l(\R)}\right)dw.
\end{align*}
By the inverse transform sampling and \eqref{eq:qt flip}, $\sigma=\frac{1}{\phi_\pi(F_\mu(x'))}\int_0^{F_\mu(x')}\pi_{F_\mu^{-1}(w)}((-\infty,F_\mu^{-1}(w)))\delta_{F_\mu^{-1}(w)}dw$ so that, with \eqref{eq:defu},
$\int_{1-\frac{F_\mu(x')-F_\mu(x)}{\phi_\pi(F_\mu(x'))}}^1F_\sigma^{-1}(w)dw=\frac{1}{\phi_\pi(F_\mu(x'))}\int_u^{F_\mu(x')}\pi_{F_{\mu}^{-1}(w)}((-\infty,F_{\mu}^{-1}(w))){F_\mu^{-1}(w)}dw$. Hence
$$
\int_u^{F_\mu(x')}\pi_{F_{\mu}^{-1}(w)}((-\infty,F_{\mu}^{-1}(w))){F_\mu^{-1}(w)}dw\ge 
\int_{\phi_\pi(F_\mu(x'))+F_\mu(x)-F_\mu(x')}^{\phi_\pi(F_\mu(x'))}F_{\tilde \nu_l}^{-1}\left(\frac{w}{\nu_l(\R)}\right)dw.$$Combining this inequality with \eqref{eq:minoet1}, we obtain \eqref{eq:compdur}.

{\bf Step 2.2 :} Let us now check that when $F_\mu(x)<F_\mu(x')$, $\phi_\uparrow(F_\mu(x))\le \phi_\uparrow(F_\mu(x'))$. If $\phi_\uparrow(F_\mu(x))\le F_\mu(x')-\nu_r(\R)$, then we have $\phi_\uparrow(F_\mu(x))\le \phi_\uparrow(F_\mu(x'))$, since $(F_\mu(x')-\nu_r(\R))^+\le \phi_\uparrow(F_\mu(x'))$ by {\bf Step 1}. Let us now suppose that $F_\mu(x')-\phi_\uparrow(F_\mu(x))\le \nu_r(\R)$. Since, by {\bf Step 1}, $\phi_\uparrow(F_\mu(x))\le \phi_\pi(F_\mu(x))$ and $F_{\tilde \nu_r}^{-1}$ is non-decreasing, we have, using a reasoning analogous to the above derivation of \eqref{eq:compdur} for the second inequality,
\begin{align*}
   \int_{F_\mu(x)-\phi_\uparrow(F_\mu(x))}^{F_\mu(x')-\phi_\uparrow(F_\mu(x))}F_{\tilde \nu_r}^{-1}\left(\frac{w}{\nu_r(\R)}\right)dw\ge \int_{F_\mu(x)-\phi_\pi(F_\mu(x))}^{F_\mu(x')-\phi_\pi(F_\mu(x))}F_{\tilde \nu_r}^{-1}\left(\frac{w}{\nu_r(\R)}\right)dw\ge \int_{F_\mu(x)}^{F_\mu(x')}F_{\mu}^{-1}\left(w\right)dw.
\end{align*}
Adding this inequality to $G(F_\mu(x),\phi_\uparrow(F_\mu(x)))=\int_0^{F_\mu(x)}F_\mu^{-1}(w)dw$, we obtain that $$G(F_\mu(x'),\phi_\uparrow(F_\mu(x)))\ge \int_{0}^{F_\mu(x')}F_{\mu}^{-1}\left(w\right)dw.$$
Since $\phi_\uparrow(F_\mu(x))\le\phi_\pi(F_\mu(x))\le \phi_\pi(F_\mu(x'))$ and, by {\bf Step 1.3}, $v\mapsto G(F_\mu(x'),v)$ is decreasing on $[(F_\mu(x')-\nu_r(\R))^+,\phi_\pi(F_\mu(x'))]$, we deduce that $\phi_\uparrow(F_\mu(x'))\ge \phi_\uparrow(F_\mu(x))$.

{\bf Step 3 :} 
 Let $x\in\R$ be such that $\mu(\{x\})>0$. The monotonicity proved in {\bf Step 2} ensures the existence of the left-hand limit $\lim_{y\to x-}\phi_\uparrow(F_\mu(y))$ such that $F_\mu(x-)-\lim_{y\to x-}\phi_\uparrow(F_\mu(y))\le F_\mu(x)-\phi_\uparrow(F_\mu(x))$. Taking the limit $y\to x-$ in the inequality $(F_\mu(y)-\nu_r(\R))^+\le \phi_\uparrow (F_\mu(y))\le F_\mu(y)\wedge \nu_l(\R)$ implies that $(F_\mu(x-)-\nu_r(\R))^+\le \lim_{y\to x-}\phi_\uparrow(F_\mu(y))\le F_\mu(x-)\wedge \nu_l(\R)$. By continuity of $(u,v)\mapsto G(u,v)$ and $u\mapsto\int_0^uF_\mu^{-1}(w)dw$, taking the limit $y\to x-$ in the equality $G(F_\mu(y),\phi_\uparrow(F_\mu(y)))=
\int_0^{F_\mu(y)}F_\mu^{-1}(w)dw$, we obtain that \begin{equation}
   G\left(F_\mu(x-),\lim_{y\to x-}\phi_\uparrow(F_\mu(y))\right)=
\int_0^{F_\mu(x-)}F_\mu^{-1}(w)dw,\label{eq:egx-}
\end{equation} so that we can set $\phi_\uparrow(F_\mu(x-))=\lim_{y\to x-}\phi_\uparrow(F_\mu(y))$. Substracting \eqref{eq:egx-} from $G(F_\mu(x),\phi_\uparrow(F_\mu(x)))=\int_0^{F_\mu(x)}F_\mu^{-1}(w)dw$ and using that, according to \eqref{eq:qt ineq flip}, $F_\mu^{-1}(w)=x$ for $w\in (F_\mu(x-),F_\mu(x)]$, we obtain
\begin{equation}
   \int_{\phi_\uparrow(F_\mu(x-))}^{\phi_\uparrow(F_\mu(x))}\left(F_{\tilde \nu_l}^{-1}\left(\frac w{\nu_l(\R)}\right)-x\right)dw+\int_{F_\mu(x-)-\phi_\uparrow(F_\mu(x-))}^{F_\mu(x)-\phi_\uparrow(F_\mu(x))}\left(F_{\tilde \nu_r}^{-1}\left(\frac w{\nu_r(\R)}\right)-x\right)dw=0.\label{eq:somnul}
\end{equation}
By \eqref{eq:compquantx} and $\phi_\uparrow(F_\mu(x))\le\phi_\pi(F_\mu(x))$, $F_{\tilde \nu_l}^{-1}\left(\frac w{\nu_l(\R)}\right)-x<0$ for $w\in(\phi_\uparrow(F_\mu(x-)),\phi_\uparrow(F_\mu(x)))$. Since, for $y<x$, by  \eqref{eq:compquantx} and $\phi_\uparrow(F_\mu(y))\le\phi_\pi(F_\mu(y))$, $F_{\tilde \nu_r}^{-1}\left(\frac w{\nu_r(\R)}\right)-y>0$ for $w\in \left(F_\mu(y)-\phi_\uparrow(F_\mu(y)),\nu_r(\R)\right)$, one has $F_{\tilde \nu_r}^{-1}\left(\frac w{\nu_r(\R)}\right)-x\ge 0$ for  $w\in(F_\mu(x-)-\phi_\uparrow(F_\mu(x-)),F_\mu(x)-\phi_\uparrow(F_\mu(x)))$.
Now let $u\in (F_\mu(x-),F_\mu(x))$ and, for $v\in[\phi_\uparrow(F_\mu(x-))\vee(u+\phi_\uparrow(F_\mu(x))-F_\mu(x)) ,\phi_\uparrow(F_\mu(x))\wedge(u+\phi_\uparrow(F_\mu(x-))-F_\mu(x-))]$,
  $$H(u,v)=\int_{{\phi_\uparrow(F_\mu(x-))}}^{v}\left(F_{\tilde \nu_l}^{-1}\left(\frac w{\nu_l(\R)}\right)-x\right)dw+\int_{{F_\mu(x-)-\phi_\uparrow(F_\mu(x-))}}^{u-v}\left(F_{\tilde \nu_r}^{-1}\left(\frac w{\nu_r(\R)}\right)-x\right)dw.$$
 If $\phi_\uparrow(F_\mu(x-))\ge u+\phi_\uparrow(F_\mu(x))-F_\mu(x)$, $$H(u,\phi_\uparrow(F_\mu(x-)))=\int_{{F_\mu(x-)-\phi_\uparrow(F_\mu(x-))}}^{u-\phi_\uparrow(F_\mu(x-))}\left(F_{\tilde \nu_r}^{-1}\left(\frac w{\nu_r(\R)}\right)-x\right)dw\ge 0$$ while if $\phi_\uparrow(F_\mu(x-))\le u+\phi_\uparrow(F_\mu(x))-F_\mu(x)$, using \eqref{eq:somnul} for the equality,
  $$H(u,u+\phi_\uparrow(F_\mu(x))-F_\mu(x))=\int_{u+\phi_\uparrow(F_\mu(x))-F_\mu(x)}^{\phi_\uparrow(F_\mu(x))}\left(x-F_{\tilde \nu_l}^{-1}\left(\frac w{\nu_l(\R)}\right)\right)dw\ge 0,$$ so that $H(u,\phi_\uparrow(F_\mu(x-))\vee(u+\phi_\uparrow(F_\mu(x))-F_\mu(x)))\ge 0$. In a similar way, $H(u,\phi_\uparrow(F_\mu(x))\wedge(u+\phi_\uparrow(F_\mu(x-))-F_\mu(x-)))\le 0$ and since $v\mapsto H(u,v)$ is decreasing on the interval, 
  there exists a unique $\phi_\uparrow(u)$ such that $\phi_\uparrow(F_\mu(x-))\le \phi_\uparrow(u)\le \phi_\uparrow(F_\mu(x))$, $F_\mu(x-)-\phi_\uparrow(F_\mu(x-))\le u-\phi_\uparrow(u)\le F_\mu(x)-\phi_\uparrow(F_\mu(x))$ and $H(u,\phi_\uparrow(u))=0$. With \eqref{eq:egx-}, this implies that $G(u,\phi_\uparrow(u))=0$. Repeating the above reasoning with the equality $H(u,\phi_\uparrow(u))=0$ replacing \eqref{eq:somnul}, we obtain for $u'\in(F_\mu(x-),u)$ the existence of a unique $\phi_\uparrow(u')\in[\phi_\uparrow(F_\mu(x-))\vee(u'+\phi_\uparrow(u)-u) ,\phi_\uparrow(u)\wedge(u'+\phi_\uparrow(F_\mu(x-))-F_\mu(x-))]$ solving $H(u',\phi_\uparrow(u'))=0$, which implies that $u'+\phi_\uparrow(u)-u\le \phi_\uparrow(u')$ and $\phi_\uparrow(u')\le \phi_\uparrow (u)$. Hence $\phi_\uparrow$ and $u\mapsto u-\phi_\uparrow(u)$ are non-decreasing on $(0,1)$.
  Moreover, $\phi_\uparrow'(u)du$ a.e. on $(F_\mu(x-),F_\mu(x))$, $\phi_\uparrow(u)<\phi_\uparrow(F_\mu(x))$ and $F_{\tilde \nu_l}^{-1}\left(\frac{\phi_\uparrow(u)}{\nu_l(\R)}\right)<x<F_{\tilde \nu_r}^{-1}\left(\frac{u-\phi_\uparrow(u)}{\nu_r(\R)}\right)$. 
   
                                                                                                                                                                                                   {\bf Step 4 :} Let us now prove \eqref{eq:maxintabs}. Since $\pi^\uparrow$ is the only non-decreasing coupling in $\Pi_M(\mu,\nu,\nu_l,\nu_r)$, it is enough to check that any coupling maximizing $\int_{\R^2}\varphi(|x-y|)\pi(dx,dy)$ over $\pi\in\Pi_M(\mu,\nu,\nu_l,\nu_r)$ is non-decreasing. To do so, we are going to adapt the proof of Lemma 1.11 \cite{BeJu16}. We modify the definition of the set $M$ in this proof into $M=\bigcup_{\tau\in{\cal S}_3}\{M_{l,\tau}\cup M_{r,\tau}\}$, where for $\tau$ in the set ${\cal S}_3$ of permutations of $\{1,2,3\}$, $M_{l,\tau}$ and $M_{r,\tau}$ are the respective images of 
                                                                                                                                                                                                   \begin{align*}M_l&=\{((x_-,y_-),(x_+,y_+),(x_+,z))\in\{\R^2\}^3: y_+<y_-< x_-<x_+<z\}\\
                                                                                                                                                                                                     \mbox{ and }M_r&=  \{((x_-,y),(x_-,z_-),(x_+,z_+))\in\{\R^2\}^3: y<x_-<x_+< z_+<z_-\}.                                                                \end{align*}
                                                                                                                                                                                                                      by $\{\R^2\}^3\ni(\chi_1,\chi_2,\chi_3)\mapsto (\chi_{\tau(1)},\chi_{\tau(2)},\chi_{\tau(3)})\in\{\R^2\}^3$. 
                                                                                                                                                                                                                     Applying Theorem 3.1 \cite{BeJu16} like in the proof of Lemma 1.11 \cite{BeJu16}, we obtain in case $(1)$ of this theorem a Borel subset $\hat\Gamma$ of $\R^2$ such that $\pi(\hat\Gamma)=1$ and $M_l\cap \hat\Gamma^3=\emptyset=M_r\cap \hat\Gamma^3$. By Lemma \ref{lemma: x,z in gamma}, we may suppose that $$\forall x\in\R,\;\exists y<x\mbox{ s.t. }(x,y)\in\hat\Gamma\Leftrightarrow\exists z>x\mbox{ s.t. }(x,z)\in\hat\Gamma.$$We deduce that \begin{align*}
                                                                                    &\{((x_-,y_-),(x_+,y_+))\in\{\R^2\}^2: y_+<y_-< x_-<x_+\}\cap\hat\Gamma^2=\emptyset\\\mbox{and }&  \{((x_-,z_-),(x_+,z_+))\in\{\R^2\}^2: x_-<x_+< z_+<z_-\}\cap\hat\Gamma^2=\emptyset.\end{align*}Setting $\Gamma=\hat\Gamma\cap \{(x,y)\in\R^2:y\ne x\}$, we have \begin{align*}&\{((x_-,y_-),(x_+,y_+))\in\{\R^2\}^2: y_+<y_-\le x_-<x_+\}\cap\Gamma^2=\emptyset\\
                                                                \mbox{and }&\{((x_-,z_-),(x_+,z_+))\in\{\R^2\}^2: x_-<x_+\le z_+<z_-\}\cap\Gamma^2=\emptyset.\end{align*} Conditions $(a)(b)$ in Definition \ref{def:NdCoupling} hold and since \begin{align*}
                                                                  \pi(\{(x,y)\in\R^2:y\ne x\})&=\pi(\{(x,y)\in\R^2:y<x\})+\pi(\{(x,y)\in\R^2:y>x\})\\&=\nu_l(\R)+\nu_r(\R)=1, \end{align*}$\pi(\Gamma)=1$ so that $\pi$ is non-decreasing.  Like in the proof of Lemma 1.11 \cite{BeJu16}, we still need to check that case $(2)$ in Theorem 3.1 \cite{BeJu16} cannot occur. In the notation of the proof of the lemma, we need to construct a competitor $\omega'$ of $\omega$ such that $\pi-\omega+\omega'\in\Pi_M(\mu,\nu,\nu_l,\nu_r)$ and \begin{equation}
                                 \int_{\R^2}\varphi(|x-y|)\omega'(dx,dy)>\int_{\R^2}\varphi(|x-y|)\omega(dx,dy).\label{compcouomom}                                    
                               \end{equation} This is done by choosing $$\alpha'_p=\frac{1}{3}\left(\frac{z-y_-}{z-y_+}\left(\delta_{(x_-,y_+)}+\delta_{(x_+,z)}\right)+\frac{y_--y_+}{z-y_+}\left(\delta_{(x_-,z)}+\delta_{(x_+,y_+)}\right)+\delta_{(x_+,y_-)}\right)$$   when $\alpha_p=\frac{1}{3}\left(\delta_{(x_-,y_-)}+   \delta_{(x_+,y_+)} +\delta_{(x_+,z)}\right)$ for some $y_+<y_-< x_-<x_+<z$ and by choosing
$$\alpha'_p=\frac{1}{3}\left(\frac{z_--z_+}{z_--y}\left(\delta_{(x_-,z_-)}+\delta_{(x_+,y)}\right)+\frac{z_+-y}{z_--y}\left(\delta_{(x_-,y)}+\delta_{(x_+,z_-)}\right)+\delta_{(x_-,z_+)}\right)$$  when $\alpha_p=\frac{1}{3}\left(  \delta_{(x_-,y)}+\delta_{(x_-,z_-)} + \delta_{(x_+,z_+)}\right)$ for some $y<x_-<x_+<z_+<z_-$. Note that, in the first case, \begin{align*}&\int_{v\in\R}\mathds{1}_{\{w<v\}}\alpha_p(dv,dw)=\frac{1}{3}\left(\delta_{y_-}(dw)+ \delta_{y_+}(dw)\right)= \int_{v\in\R}\mathds{1}_{\{w<v\}}\alpha'_p(dv,dw)\\&\mbox{and } \int_{v\in\R}\mathds{1}_{\{w>v\}}\alpha_p(dv,dw)=\frac{1}{3}\delta_{z}(dw)= \int_{v\in\R}\mathds{1}_{\{w>v\}}\alpha'_p(dv,dw),\end{align*}which contributes to ensuring that $\pi-\omega+\omega'\in\Pi_M(\mu,\nu,\nu_l,\nu_r)$.  Moreover, by Lemma \ref{lemma: 5 point cost} applied with $(y,m,z)=(y_+,y_-,z)$, the function $f$ defined by \[f(x) = \frac{z-y_-}{z-y_+}\varphi(\vert x-y_+\vert ) + \frac{y_--y_+}{z-y_+}\varphi(\vert z-x\vert) -\varphi(\vert x-y_-\vert)\] is decreasing on $[y_-,z]$. We deduce that, still in the first case, 
\begin{align*}
  \int_{\R^2}\varphi(|v-w|)\alpha'_p(dv,dw)-\int_{\R^2}\varphi(|v-w|)\alpha_p(dv,dw)=\frac{1}{3}\left(f(x_-)-f(x_+)\right)>0,\end{align*}
which contributes to ensuring \eqref{compcouomom}. One checks in the same way that the two properties still hold in the second case.

                                                                                                                                                                                                   Let us now give an alternative simple argument in the particular case when $\varphi$ is the identity function. For $\pi\in\Pi_M(\mu,\nu,\nu_l,\nu_r)$, we have $\int_{\R^2}(y-z)\mathds{1}_{\{z<y\}}\pi(dy,dz)-\int_{\R^2}(z-y)\mathds{1}_{\{z>y\}}\pi(dy,dz)=0$. Since $\nu^\pi_l=\nu_l$, we deduce with Fubini's theorem and \eqref{eq:defphipi} that
\begin{align*}
  \frac 12&\int_{\R^2}|y-z|\pi(dy,dz)=\int_{\R^2}(y-z)\mathds{1}_{\{z<y\}}\pi(dy,dz)\\&=\int_{\R^2}\int_{\R_+}\mathds{1}_{\{y>x\}}dx\mathds{1}_{\{z<y\}}\pi(dy,dz)-\int_{\R^2}\int_{\R_-}\mathds{1}_{\{y\le x\}}dx\mathds{1}_{\{z<y\}}\pi(dy,dz)-\int_{\R}z\nu_l(dz)
 \\&=\int_{\R_+}(\nu_l(\R)-\phi_\pi(F_\mu(x)))dx-\int_{\R_-}\phi_\pi(F_\mu(x)))dx-\int_{\R}z\nu_l(dz).
\end{align*}
Since, according to {\bf Step 1}, $\forall x\in\R$, $\phi_\uparrow(F_\mu(x)))\le \phi_\pi(F_\mu(x))$, we conclude that $$\int_{\R^2}|y-z|\pi^\uparrow(dy,dz)=\sup_{\pi\in\Pi_M(\mu,\nu,\nu_l,\nu_r)}\int_{\R^2}|y-z|\pi(dy,dz).$$
                                                                                                                                                                                         \end{proof}

 \begin{proof}[Proof of Proposition \ref{prop:decompdecrois}]
   Let us suppose that there exists $\pi^\downarrow\in\Pi_M(\mu,\nu,\nu_l,\nu_r)$ such that the coupling $\pi(dx,dy):=\frac{\pi^\downarrow(dx,dy)+(\nu_l+\nu_r-\nu)(dx)\delta_x(dy)}{\nu_l(\R)+\nu_r(\R)}\in\Pi_M\left(\frac{\mu+\nu_l+\nu_r-\nu}{\nu_l(\R)+\nu_r(\R)},\frac{\nu_l+\nu_r}{\nu_l(\R)+\nu_r(\R)},\frac{\nu_l}{\nu_l(\R)+\nu_r(\R)},\frac{\nu_r}{\nu_l(\R)+\nu_r(\R)}\right)$ is non-increasing. Then, by Proposition \ref{prop:deccoupl}, there exist $-\infty<a\le b<+\infty$ such that $(\mu+\nu_l+\nu_r-\nu)([a,b])=(\nu_l+\nu_r)(\R)$, $(\nu_l+\nu_r)((a,b))=0$ and $\pi(\{(x,x)\})=\frac{(\mu+\nu_l+\nu_r-\nu)(\{x\})}{\nu_l(\R)+\nu_r(\R)}\wedge\frac{(\nu_l+\nu_r)(\{x\})}{\nu_l(\R)+\nu_r(\R)}$ for $x\in\{a,b\}$. Since $\pi(\{(x,x):x\in\R\})=0$, we deduce that we may choose a closed or semi-open or open interval $I$ with ends $a$ and $b$ such that $\frac{\mu+\nu_l+\nu_r-\nu}{\nu_l(\R)+\nu_r(\R)}(I)=1$ and $\frac{\nu_l+\nu_r}{\nu_l(\R)+\nu_r(\R)}(I)=0$. Since $\nu-(\nu_l+\nu_r)=\nu_0^{\pi^\downarrow}\le\mu\wedge\nu$, we have that $\mu+\nu_l+\nu_r-\nu-(\mu-\nu)^+=\nu_l+\nu_r-(\nu-\mu)^+$ is a non-negative measure and we deduce that $\nu_l+\nu_r=(\nu-\mu)^+$ and $\mu+\nu_l+\nu_r-\nu=(\mu-\nu)^+$. Hence the Dispersion Assumption is satisfied. Moreover, since $\pi(\{(-\infty,a)\cup(b,+\infty)\}\times\R)\le \frac{(\mu+\nu_l+\nu_r-\nu)(\R\setminus I)}{\nu_l(\R)+\nu_r(\R)}=0$, $\pi(\R\times (a,b))\le\frac{\nu_l+\nu_r}{\nu_l(\R)+\nu_r(\R)}(I)=0$ and $\pi(\{a,a\})=0$, we have
   \begin{align*}
      \nu_l(dy)&=\nu^\pi_l(dy)=\int_{x\in\R}\mathds{1}_{\{y<x\}}\pi(dx,dy)=\int_{x\in\R}\mathds{1}_{\{y<x\le b\}}\pi(dx,dy)\\&=\int_{x\in\R}\mathds{1}_{\{y<x,y\le a\}}\pi(dx,dy)=\int_{x\in\R}\mathds{1}_{\{y\le a\}}\pi(dx,dy)=\mathds{1}_{\{y\le a\}}(\nu-\mu)^+(dy).
   \end{align*}and, in a symmetric way, $\nu_r(dy)=\mathds{1}_{\{y\ge b\}}(\nu-\mu)^+(dy)$ (notice that, when $a=b$, then $\mu=\delta_a$ and $(\nu-\mu)^+(\{a\})=0$).
     
     In the other direction, under the Dispersion Assumption, denoting by $a\le b$ the ends of the interval $I$, we have $(\mu-\nu)^+([a,b])=(\mu-\nu)^+(\R)$ and $(\nu-\mu)^+((a,b))=0$ so that Proposition \ref{prop:deccoupl} ensures the existence of a unique non-increasing coupling $\pi\in\Pi_M\left(\frac{(\mu-\nu)^+}{(\mu-\nu)^+(\R)},\frac{(\nu-\mu)^+}{(\mu-\nu)^+(\R)}\right)$ such that $$\pi(\{(a,a)\})=\frac{(\mu-\nu)^+(\{a\})}{(\mu-\nu)^+(\R)}\wedge \frac{(\nu-\mu)^+(\{a\})}{(\mu-\nu)^+(\R)}=0=\pi(\{(b,b)\}).$$ Hence $\int_{x\in\R}\mathds{1}_{\{y<x\}}\pi(dx,dy)=\mathds{1}_{\{y\le a\}}\frac{(\nu-\mu)^+(dy)}{(\mu-\nu)^+(\R)}$ and $\int_{x\in\R}\mathds{1}_{\{y>x\}}\pi(dx,dy)=\mathds{1}_{\{y\ge b\}}\frac{(\nu-\mu)^+(dy)}{(\mu-\nu)^+(\R)}$. The coupling $(\mu-\nu)^+(\R)\pi(dx,dy)+\mu\wedge \nu(dx)\delta_x(dy)$ belongs to $\Pi_M(\mu,\nu,\mathds{1}_{\{y\le a\}}(\nu-\mu)^+(dy),\mathds{1}_{\{y\ge b\}}(\nu-\mu)^+(dy))$.\end{proof}
The proof of Proposition \ref{prop:nu0} relies on the following lemma.
\begin{lemma}\label{lem:decompmoy}
   Let $\eta\in{\cal P}_1(\R)$ be such that $\int_{\R}z\eta(dz)=y$. For $x>y$, if $\eta((x,+\infty))>0$ then there exists a non-negative measure $\hat\eta$ such that $\hat\eta(dz)\le \mathds{1}_{\{z<x\}}\eta(dz)$ and $\int_\R(z-x)\hat\eta(dz)+\int_{(x,+\infty)}(z-x)\eta(dz)=0$.
 \end{lemma}
 \begin{proof}[Proof of Lemma \ref{lem:decompmoy}]
  The function $[0,F_\eta(x-)]\ni u\mapsto G(u)=\int_0^u(F_\eta^{-1}(v)-x)dv+\int_{(x,+\infty)}(z-x)\eta(dz)$ is continuous, decreasing and such that $G(0)=\int_{(x,+\infty)}(z-x)\eta(dz)>0$ and $G(F_\eta(x-))=\int_{\R}(z-x)\eta(dz)=y-x<0$. Hence there exists a unique $\hat u\in(0,F_\eta(x-))$ such that $G(\hat u)=0$. We set $\hat\eta(dz)=\int_0^{\hat u}\delta_{F_\eta^{-1}(v)}(dz)dv\le \int_0^{F_\eta(x-)}\delta_{F_\eta^{-1}(v)}(dz)dv=\mathds{1}_{\{z<x\}}\eta(dz)$.
 \end{proof}
 \begin{proof}[Proof of Proposition \ref{prop:nu0}]
 Since $\int_\R(z-x)^+\nu(dz)-\int_\R(x-z)^+\nu(dz)=\int_\R z\nu(dz)-x=\int_{\R}y\mu(dy)-x=\int_\R(y-x)^+\mu(dy)-\int_\R(x-y)^+\mu(dy)$, we have that
   \begin{equation}
      \int_\R(z-x)^+\nu(dz)-\int_\R(y-x)^+\mu(dy)=\int_\R(x-z)^+\nu(dz)-\int_\R(x-y)^+\mu(dy)=\frac 12\left(u_\nu(x)-u_\mu(x)\right).\label{eq:eqpot}
    \end{equation}
From now on, we suppose that $u_\nu(x)>u_\mu(x)$. Then $\nu((-\infty,x))=F_\nu(x-)>0$ and $\nu((x,+\infty))=1-F_\nu(x)>0$. Since $F_\nu^{-1}(v)-x>0$ for $v\in (F_\nu(x),1)$, the function $[0,1-F_\nu(x)]\ni u\mapsto G(u)=\int_{F_\nu(x)}^{F_\nu(x)+u}(F_\nu^{-1}(v)-x)dv$ is increasing. It also is continuous and such that $G(0)=0$ and $G(1-F_\nu(x))=\int_\R(z-x)^+\nu(dz)$. We deduce the existence of a unique $p_+(x)\in (0,1-F_\nu(x)]$ such that \eqref{eq:dd} holds. In a similar way, there exists a unique $p_-(x)\in (0,F_\nu(x-)]$ such that \eqref{eq:gg} holds.

The proof of the other statements relies on three steps. In the first step, we are going to check that \begin{equation}
   \inf_{\pi\in\Pi_M(\mu,\nu)}\pi(\{(x,x)\})\ge (\mu(\{x\})-p_-(x)-p_+(x))^+,\label{eq:miod}
 \end{equation}
 and that when $\mu(\{x\})>p_-(x)+p_+(x)$, if $\pi\in\Pi_M(\mu,\nu)$ is such that $\pi(\{(x,x)\})=\mu(\{x\})-p_-(x)-p_+(x)$ then $\pi_x=\eta_x$ with $\eta_x$ given by \eqref{eq:pix} and $\pi(\{(-\infty,x)\times(x,+\infty)\}\cup\{(x,+\infty)\times(-\infty,x)\})=0$.
In the second step, we will check that if $\inf_{\pi\in\Pi_M(\mu,\nu)}\pi(\{(x,x)\})>0$, then there exists $\pi\in\Pi_M(\mu,\nu)$ such that $\pi_x=\eta_x$ where $\eta_x$ is given by \eqref{eq:pix} so that $\pi(\{(x,x)\})=\mu(\{x\})-p_-(x)-p_+(x)$ and, by \eqref{eq:miod}, $\inf_{\pi\in\Pi_M(\mu,\nu)}\pi(\{(x,x)\})=\mu(\{x\})-p_-(x)-p_+(x)$. As a consequence, $\inf_{\pi\in\Pi_M(\mu,\nu)}\pi(\{(x,x)\})=0$ when $\mu(\{x\})-p_-(x)-p_+(x)\le 0$ and \eqref{eq:miod} is an equality. 
In the last step, we prove that when $\mu(\{x\})\in(0,p_-(x)+p_+(x)]$, 
 there exists $\pi\in\Pi_M(\mu,\nu)$ such that $\pi_x=\eta_x$ where $\eta_x$ is given by \eqref{eq:pixx}
.

{\bf Step 1 :} Let $\pi\in \Pi_M(\mu,\nu)$. Since, by the martingale property and Jensen's inequality, $\mu(dy)$ a.e., $\int_{\R}|z-x|\pi_y(dz)\ge |y-x|$, one has
\begin{align}
   \mu(\{x\})\int_\R|z-x|\pi_x(dz)&\le \int_{\R\times\R}\left(|z-x|-|y-x|\right)\pi_y(dz)\mu(dy)=u_\nu(x)-u_\mu(x)\notag\\&=\int_{F_\nu(x-)-p_-(x)}^{F_\nu(x-)}(x-F_\nu^{-1}(v))dv+\int_{F_\nu(x)}^{F_\nu(x)+p_+(x)}(F_\nu^{-1}(v)-x)dv.\label{eq:diffpotx}
\end{align}
Moreover, by the martingale property for the first equality and by \eqref{eq:dd}, \eqref{eq:gg} and \eqref{eq:eqpot} for the second equality, \begin{align*}
   \mu(\{x\})\int_\R(z-x)^+\pi_x(dz)&-\mu(\{x\})\int_\R(x-z)^+\pi_x(dz)=0\\&=\int_{F_\nu(x)}^{F_\nu(x)+p_+(x)}(F_\nu^{-1}(v)-x)dv-\int_{F_\nu(x-)-p_-(x)}^{F_\nu(x-)}(x-F_\nu^{-1}(v))dv.
\end{align*} Hence,
\begin{align}
  &\mu(\{x\})\int_\R(z-x)^+\pi_x(dz)\le \int_{F_\nu(x)}^{F_\nu(x)+p_+(x)}(F_\nu^{-1}(v)-x)dv,\label{eq:majoddd}\\
  &\mu(\{x\})\int_\R(x-z)^+\pi_x(dz)\le\int_{F_\nu(x-)-p_-(x)}^{F_\nu(x-)}(x-F_\nu^{-1}(v))dv.\notag
\end{align}
We have $\mu(\{x\})\mathds{1}_{\{x<z\}}\pi_x(dz)\le \mathds{1}_{\{x<z\}}\nu(dz)$ so that, by Lemma \ref{lem:compmom} applied with $f(z)=z-x$, $\hat\eta(dz)=\frac{\mu(\{x\})\mathds{1}_{\{x<z\}}\pi_x(dz)}{1-F_\nu(x)}$ and $\eta(dz)=\frac{\mathds{1}_{\{x<z\}}\nu(dz)}{1-F_\nu(x)}$ such that $\hat\eta(\R)=\frac{\mu(\{x\})\pi_x((x,+\infty))}{1-F_\nu(x)}$, $F_\eta(z)=\frac{(F_\nu(z)-F_\nu(x))^+}{1-F_\nu(x)}$ and $F_\eta^{-1}(u)=F_\nu^{-1}(F_\nu(x)+(1-F_\nu(x))u)$,
$$\int_{F_\nu(x)}^{F_\nu(x)+\mu(\{x\})\pi_x((x,+\infty))}\left(F_\nu^{-1}(v)-x\right)dv=(1-F_\nu(x))\int_0^{\hat\eta(\R)}\left(F_\eta^{-1}(v)-x\right)dv\le \mu(\{x\})\int_\R(z-x)^+\pi_x(dz),$$
and when $\mu(\{x\})\pi_x((x,+\infty))>0$,
\begin{align}
   \frac{\mathds{1}_{\{x<z\}}\pi_x(dz)}{\pi_x((x,+\infty))}&\ge_{st}
  \frac{1}{\mu(\{x\})\pi_x((x,+\infty))}\int_{F_\nu(x)}^{F_\nu(x)+\mu(\{x\})\pi_x((x,+\infty))}\delta_{F_\nu^{-1}(u)}du.\label{eq:minostogd}
\end{align}
With \eqref{eq:majoddd}, the first inequality implies that
$$\int_{F_\nu(x)}^{F_\nu(x)+\mu(\{x\})\pi_x((x,+\infty))}\left(F_\nu^{-1}(v)-x\right)dv\le \int_{F_\nu(x)}^{F_\nu(x)+p_+(x)}(F_\nu^{-1}(v)-x)dv.$$Since $F_\nu^{-1}(v)-x>0$ for $v\in(F_\nu(x),1)$, we deduce that $\mu(\{x\})\pi_x((x,+\infty))\le p_+(x)$. In a symmetric way, $\mu(\{x\})\pi_x((-\infty,x))\le p_-(x)$.  Since $\pi(\{(x,x)\})=\mu(\{x\})\left(1-\pi_x((-\infty,x))-\pi_x((x,+\infty))\right)$, we conclude that \eqref{eq:miod} holds. 

Let us now suppose that $\mu(\{x\})>p_-(x)+p_+(x)$ and that $\pi\in\Pi_M(\mu,\nu)$ is such that $\pi(\{(x,x)\})=\mu(\{x\})-p_-(x)-p_+(x)$. Then $\mu(\{x\})\pi_x((x,+\infty))= p_+(x)$ and $\mu(\{x\})\pi_x((-\infty,x))= p_-(x)$. Moreover, by \eqref{eq:minostogd} and \eqref{eq:majoddd}, $\mathds{1}_{\{x<z\}}\pi_x(dz)=\frac{1}{\mu(\{x\})}\int_{F_\nu(x)}^{F_\nu(x)+p_+(x)}\delta_{F_\nu^{-1}(u)}du$. In a symmetric way, $\mathds{1}_{\{z<x\}}\pi_x(dz)=\frac{1}{\mu(\{x\})}\int_{F_\nu(x-)-p_-(x)}^{F_\nu(x-)}\delta_{F_\nu^{-1}(u)}du$ so that $\pi_x=\eta_x$ with $\eta_x$ given by \eqref{eq:pix}. Moreover, the inequality in \eqref{eq:diffpotx} becomes an equality, so that, with the martingale constraint,
\begin{align*}
   0&=\int_{\R^2}\left(\mathds{1}_{\{y<x\}}\left(x-z+2(z-x)^+-x+y\right)+\mathds{1}_{\{y>x\}}\left(z-x+2(x-z)^+-y+x\right)\right)\pi_y(dz)\mu(dy)\\&=\int_{\R^2}2\left(\mathds{1}_{\{y<x\}}(z-x)^++\mathds{1}_{\{y>x\}}(x-z)^+\right)\pi_y(dz)\mu(dy).
\end{align*}
Hence $\pi(\{(-\infty,x)\times(x,+\infty)\}\cup\{(x,+\infty)\times(-\infty,x)\})=0$.

{\bf Step 2 : }Let us suppose that $\inf_{\pi\in\Pi_M(\mu,\nu)}\pi(\{(x,x)\})>0$, which implies that $\mu(\{x\})>0$. Let $\tilde \pi\in\Pi_M(\mu,\nu)$. We are now going to modify $\tilde\pi$ into $\pi\in\Pi_M(\mu,\nu)$ such that $\pi_x=\eta_x$ where $\eta_x$ is given by \eqref{eq:pix}. 
By lemma \ref{lem:decompmoy}, we have that $\mu(dy)$ a.e. on $(-\infty,x)$, $\tilde \pi_y((x,+\infty))>0$ implies the existence of $\hat\pi_y(dz)\le \mathds{1}_{\{z<x\}}\tilde \pi_y(dz)$ such that $\int_\R z\hat\pi_y(dz)+\int_{(x,+\infty)}z\tilde \pi_y(dz)=x\left(\hat\pi_y(\R)+\tilde\pi_y((x,+\infty))\right)$. In the same way $\mu(dy)$ a.e. on $(x,+\infty)$, $\tilde \pi_y((-\infty,x))>0$ implies the existence of $\hat\pi_y(dz)\le \mathds{1}_{\{z>x\}}\tilde \pi_y(dz)$ such that $\int_\R z\hat\pi_y(dz)+\int_{(-\infty,x)}z\tilde \pi_y(dz)=x\left(\hat\pi_y(\R)+\tilde\pi_y((-\infty,x))\right)$. We now set for $u\in[0,1]$
\begin{align*}
  \bar\pi(u,dy,dz)&=\tilde\pi(dy,dz)+u\mathds{1}_{\{y<x\}}\mu(dy)\left((\tilde \pi_y((x,+\infty))+\hat\pi_y(\R))\delta_x(dz)-\mathds{1}_{\{z>x\}}\tilde \pi_y(dz)-\hat\pi_y(dz)\right)\\
                &+u\delta_x(dy)\int_{w\in(-\infty,x)}\left(\mathds{1}_{\{z>x\}}\tilde \pi_w(dz)+\hat\pi_w(dz)-(\tilde \pi_w((x,+\infty))+\hat\pi_w(\R))\delta_x(dz)\right)\mu(dw)\\
  &+u\mathds{1}_{\{y>x\}}\mu(dy)\left((\tilde\pi_y((-\infty,x))+\hat\pi_y(\R))\delta_x(dz)-\mathds{1}_{\{z<x\}}\tilde \pi_y(dz)-\hat\pi_y(dz)\right)\\
                &+u\delta_x(dy)\int_{w\in(x,+\infty)}\left(\mathds{1}_{\{z<x\}}\tilde \pi_w(dz)+\hat\pi_w(dz)-(\tilde\pi_w((-\infty,x))+\hat\pi_w(\R))\delta_x(dz)\right)\mu(dw).
\end{align*}
Since $\mu(dy)$ a.e. on $(-\infty,x)$ (resp. $(x,+\infty)$), $\tilde \pi_y(dz)-u\mathds{1}_{\{z>x\}}\tilde \pi_y(dz)-u\hat\pi_y(dz)+u(\tilde \pi_y((x,+\infty))+\hat\pi_y(\R))\delta_x(dz)$ (resp. $\tilde \pi_y(dz)-u\mathds{1}_{\{z<x\}}\tilde \pi_y(dz)-u\hat\pi_y(dz)+u(\tilde \pi_y((-\infty,x))+\hat\pi_y(\R))\delta_x(dz)$) is a probability measure with expectation $y$, as long as
\begin{equation}
   u\left(\int_{(-\infty,x)}(\tilde \pi_w((x,+\infty))+\hat\pi_w(\R))\mu(dw)+\int_{(x,+\infty)}(\tilde\pi_w((-\infty,x))+\hat\pi_w(\R))\mu(dw)\right)\le \tilde\pi(\{(x,x)\}),\label{eq:majou}
\end{equation}
$\bar\pi(u,dy,dz)\in\Pi_M(\mu,\nu)$ and $\bar\pi(u,\{(x,x)\})>0$, so that the inequality in \eqref{eq:majou} is strict. Therefore \eqref{eq:majou} holds with strict inequality for $u\in[0,1]$, and $\bar\pi(1,dy,dz)\in\Pi_M(\mu,\nu)$. Denoting for simplicity $\bar\pi(dy,dz)=\bar\pi(1,dy,dz)$, we have that  $\bar\pi(\{(-\infty,x)\times(x,+\infty)\}\cup\{(x,+\infty)\times(-\infty,x)\})=0$. As a consequence,
$$\int_{\R^2}\mathds{1}_{\{y\neq x\}}\left(|z-x|-|y-x|\right)\bar\pi(dy,dz)=0$$
so that $\mu(\{x\})\int_\R|z-x|\bar\pi_x(dz)=u_\nu(x)-u_\mu(x)$. With the martingale property, this implies that 
\begin{equation}
   \label{eq:momdifx}\mu(\{x\})\int_\R(x-z)\mathds{1}_{\{z<x\}}\bar\pi_x(dz)=\frac 12(u_\nu(x)-u_\mu(x))=\int_{F_\nu(x-)-p_-(x)}^{F_\nu(x-)}(x-F_\nu^{-1}(v))dv.
\end{equation}
Let $w=F_\nu^{-1}(F_\nu(x-)-p_-(x))$ and $r=F_\nu(w)-F_\nu(x-)+p_-(x)\in[0,\nu(\{w\})]$. We set 
\begin{align*}
&\gamma(dz)=\int_{F_\nu(x-)-p_-(x)}^{F_\nu(x-)}\delta_{F_\nu^{-1}(v)}(dz)dv-\mu(\{x\})\mathds{1}_{\{w<z<x\}}\bar\pi_x(dz)-(\bar\pi(\{(x,w)\})\wedge r)\delta_w(dz)\\
 \mbox{and }&\sigma(dz)=\mu(\{x\})\mathds{1}_{\{z<w\}}\bar\pi_x(dz)+(\bar\pi(\{(x,w)\})-r)^+\delta_w(dz).
\end{align*}
Since $ \int_{F_\nu(x-)-p_-(x)}^{F_\nu(x-)}\delta_{F_\nu^{-1}(v)}(dz)dv=\mathds{1}_{\{w<z<x\}}\nu(dz)+r\delta_w(dz)$ and $\bar\pi((x,+\infty)\times(-\infty,x))=0$, $\int_{F_\nu(x-)-p_-(x)}^{F_\nu(x-)}\delta_{F_\nu^{-1}(v)}(dz)dv=\int_{(-\infty,x]}\mathds{1}_{\{w<z<x\}}\bar\pi_y(dz)\mu(dy)+r\delta_w(dz)$ and $\nu(\{w\})-\bar\pi(\{(x,w)\})=\int_{(-\infty,x)}\bar\pi_y(\{w\})\mu(dy)$ so that setting $\alpha=\mathds{1}_{\{\nu(\{w\})>\bar\pi(\{(x,w)\})\}}\frac{\left(r-\bar\pi(\{(x,w)\})\right)^+}{\nu(\{w\})-\bar\pi(\{(x,w)\})}\in[0,1]$ and $f(z)=\alpha\mathds{1}_{\{z=w\}}+\mathds{1}_{\{w<z<x\}}$,
\begin{align}
  \gamma(dz)&=\int_{(-\infty,x)}\mathds{1}_{\{w<z<x\}}\bar\pi_y(dz)\mu(dy)+\left(r-\bar\pi(\{(x,w)\})\right)^+\delta_w(dz)\notag\\
  &=\int_{(-\infty,x)}\mathds{1}_{\{w<z<x\}}\bar\pi_y(dz)\mu(dy)+\alpha\int_{(-\infty,x)}\bar\pi_y(\{w\})\delta_w(dz)\mu(dy)=\int_{(-\infty,x)}f(z)\bar\pi_y(dz)\mu(dy).\label{eq:etrans}
\end{align}
Since \begin{equation}\label{eq:et-sig}
   \gamma(dz)-\sigma(dz)=\int_{F_\nu(x-)-p_-(x)}^{F_\nu(x-)}\delta_{F_\nu^{-1}(v)}(dz)dv-\mu(\{x\})\mathds{1}_{\{z<x\}}\bar\pi_x(dz),
\end{equation} \eqref{eq:momdifx} implies that $\int_\R (x-z)\gamma(dz)=\int_\R(x-z)\sigma(dz)$. Since $\sigma$ (resp. $\gamma$) only weights $(-\infty,w]$ (resp. $[w,x)$), this implies that $\sigma(\R)\le\gamma(\R)$.  If $\gamma(\R)=0$, then $\sigma(\R)=0$ so that, by \eqref{eq:et-sig}, $\mu(\{x\})\mathds{1}_{\{z<x\}}\bar\pi_x(dz)=\int_{F_\nu(x-)-p_-(x)}^{F_\nu(x-)}\delta_{F_\nu^{-1}(v)}(dz)dv$ and we set $\check\pi=\bar\pi$. Otherwise, since $\int_\R z\gamma(dz)=(\gamma(\R)-\sigma(\R))x+\int_\R z\sigma(dz)$, the probability measures $\frac{\gamma}{\gamma(\R)}$ and $\frac{1}{\gamma(\R)}\left((\gamma(\R)-\sigma(\R))\delta_x+\sigma\right)$ share the same mean and, by comparison of their supports, satisfy $\frac{\gamma}{\gamma(\R)}\le_{cx}\frac{1}{\gamma(\R)}\left((\gamma(\R)-\sigma(\R))\delta_x+\sigma\right)$. By Strassen's theorem, there exists $\theta\in\Pi_M(\frac{\gamma}{\gamma(\R)},\frac{1}{\gamma(\R)}\left((\gamma(\R)-\sigma(\R))\delta_x+\sigma\right))$. We set 
\begin{align*}&\check\pi(dy,dz)=\mathds{1}_{\{y>x\}}\bar\pi(dy,dz)+\mathds{1}_{\{y<x\}}\mu(dy)\left((1-f(z))\bar\pi_y(dz)+\int f(v)\theta_v(dz)\bar\pi_y(dv)\right)\\&+\mu(\{x\})\delta_x(dy)\left(\bar\pi_x(dz)+\frac{1}{\mu(\{x\})}\int_{t\in(-\infty,x)}\left(f(z)\bar\pi_t(dz)-\int f(v)\theta_v(dz)\bar\pi_t(dv)\right)\mu(dt)\right).\end{align*}
For $\mu(dy)$ a.e. $y\in(-\infty,x)$, since $f(v)\bar\pi_y(dv)$ a.e. $\int_\R z\theta_v(dz)=v$ by \eqref{eq:etrans} and the martingale property of $\theta$, we have\begin{align*}
   \int_\R z\check\pi_y(dz)&=\int_\R z(1-f(z))\bar\pi_y(dz)+\int_\R f(v)\int_\R z\theta_v(dz)\bar\pi_y(dv)\\&=\int_\R z(1-f(z))\bar\pi_y(dz)+\int_\R v f(v)\bar\pi_y(dv)=\int_\R z\bar\pi_y(dz)=y.
\end{align*}
Using \eqref{eq:etrans}, the definition of $\theta$ and \eqref{eq:et-sig}, we obtain
\begin{align*}
  \mu(\{x\})\check\pi_x(dz)&=\mu(\{x\})\bar\pi_x(dz)+\gamma(dz)-\int_{v\in\R}\theta_v(dz)\gamma(dv)\\&=\mu(\{x\})\bar\pi_x(dz)+\gamma(dz)-\sigma(dz)-(\gamma(\R)-\sigma(\R))\delta_x(dz)\\&=\mu(\{x\})\mathds{1}_{\{z\ge x\}}\bar\pi_x(dz)+\int_{F_\nu(x-)-p_-(x)}^{F_\nu(x-)}\delta_{F_\nu^{-1}(v)}(dz)dv-(\gamma(\R)-\sigma(\R))\delta_x(dz).\end{align*}The positivity of $\check\pi(\{(x,x)\})=\bar\pi(\{(x,x)\})-(\gamma(\R)-\sigma(\R))$ can be checked by considering $(1-u)\bar\pi+u\check\pi$ for $u\in[0,1]$ as we did with $\bar\pi(u,dy,dz)=(1-u)\tilde\pi(dy,dz)+u\bar\pi(dy,dz)$in {\bf Step 1}. Hence $\check\pi\in\Pi_M(\mu,\nu)$. With a symmetric reasoning, we construct from $\check\pi$ a coupling $\pi\in\Pi_M(\mu,\nu)$ such that $\pi_x=\eta_x$.

{\bf Step 3 : }Let us now suppose that $\mu(\{x\})\in(0,p_-(x)+p_+(x)]$, so that, by {\bf Step 2}, $$\inf_{\pi\in \Pi_M(\mu,\nu)}\pi(\{(x,x)\})=0.$$ Let $(\pi^n)_{n\in\N}$ denote a $\Pi_M(\mu,\nu)$-valued sequence such that $\lim_{n\to\infty}\pi^n(\{(x,x)\})=0$. Since $\Pi_M(\mu,\nu)$ is compact for the weak convergence topology, we may extract a subsequence that we still index by $n$ for notational simplicity and which converges weakly to $\pi^\infty\in\Pi_M(\mu,\nu)$. For $\varepsilon\in (0,F_\nu(x-)\wedge(1-F_\nu(x)))$, we have
\begin{align*}
   \pi^n(\{x\}&\times\{(-\infty,F_\nu^{-1}(F_\nu(x-)-\varepsilon)]\cup[F_\nu^{-1}(F_\nu(x)+\varepsilon),+\infty)\})\\&\ge \pi^n(\{x\}\times\R)-\pi^n(\{(x,x)\}-\nu((F_\nu^{-1}(F_\nu(x-)-\varepsilon),x)\cup(x,F_\nu^{-1}(F_\nu(x)+\varepsilon)))\\&\ge \mu(\{x\})-\pi^n(\{(x,x)\}-2\varepsilon.
\end{align*}
Taking the limit $n\to\infty$ in this inequality and using the closedness of $\{x\}\times\{(-\infty,F_\nu^{-1}(F_\nu(x-)-\varepsilon)]\cup[F_\nu^{-1}(F_\nu(x)+\varepsilon),+\infty)\}$ together with the Portmanteau theorem for the left-hand side, we obtain
$$\pi^\infty(\{x\}\times\{(-\infty,F_\nu^{-1}(F_\nu(x-)-\varepsilon)]\cup[F_\nu^{-1}(F_\nu(x)+\varepsilon),+\infty)\})\ge \mu(\{x\})-2\varepsilon.$$
so that $\pi^\infty(\{(x,x)\})\le 2\varepsilon$. Letting $\varepsilon\to 0+$, we conclude that $\pi^\infty(\{(x,x)\})=0$.

The function $\left[\left(\mu(\{x\})-p_+(x)\right)^+,p_-(x)\wedge{\mu(\{x\})}\right]\ni q\mapsto G(q)=\int_{F_\nu(x-)-q}^{F_\nu(x-)}(x-F_\nu^{-1}(v))dv+\int_{F_\nu(x)}^{F_\nu(x)+\mu(\{x\})-q}(x-F_\nu^{-1}(v))dv$ is continuous and increasing. If $\mu(\{x\})\ge p_+(x)$ then, using \eqref{eq:dd}, \eqref{eq:gg} and \eqref{eq:eqpot} for the equality and $-p_-(x)\le p_+(x)-\mu(\{x\})$ for the inequality, we get
\begin{align*}
   G(\mu(\{x\})-p_+(x))=\int_{F_\nu(x-)-p_-(x)}^{F_\nu(x-)+p_+(x)-\mu(\{x\})}(F_\nu^{-1}(v)-x)dv\le 0
\end{align*}
while otherwise $G(0)=\int_{F_\nu(x)}^{F_\nu(x)+\mu(\{x\})}(x-F_\nu^{-1}(v))dv\le 0$ so that $G(\left(\mu(\{x\})-p_+(x)\right)^+)\le 0$. Since, in a symmetric way, $G(p_-(x)\wedge{\mu(\{x\})})\ge 0$, there exists a unique $q(x)$ in the interval such that $G(q(x))=0$. If $\mu(\{x\})=1$, then $p_-(x)+p_+(x)=1$ by \eqref{eq:miod}, $F_\nu(x-)=F_\nu(x)=q(x)$ and $\eta_x$ given by \eqref{eq:pixx} is such that $\eta_x=\nu$ so that the unique element of $\Pi_M(\mu,\nu)$ is $\delta_x(dy)\eta_x(dz)$ . Let us now suppose that $\mu(\{x\})\in(0,1)$ and check that the probability measure $\eta_x$ given by \eqref{eq:pixx} is such that $\eta_x\le_{cx}\pi^\infty_x$. Since $\pi^\infty(\{(x,x)\})=0$, one has $\pi^\infty_x\le \frac{1}{\mu(\{x\})}\left(\int_{0}^{F_\nu(x-)}\delta_{F_\nu^{-1}(v)}dv+\int_{F_\nu(x)}^{1}\delta_{F_\nu^{-1}(v)}dv\right)$ so that $\pi^\infty_x=\frac{1}{\mu(\{x\})}\left(\int_{0}^{F_\nu(x-)}p(v)\delta_{F_\nu^{-1}(v)}dv+\int_{F_\nu(x)}^{1}p(v)\delta_{F_\nu^{-1}(v)}dv\right)$ for some measurable function $p:[0,1]\to[0,1]$ such that $\int_{0}^{F_\nu(x-)}p(v)dv+\int_{F_\nu(x)}^{1}p(v)dv=\mu(\{x\})$. Let $C=\int_{F_\nu(x-)-q(x)}^{F_\nu(x-)}(1-p(v))dv+\int_{F_\nu(x)}^{F_\nu(x)+\mu(\{x\})-q(x)}(1-p(v))dv$. One has $C=\int_0^{F_\nu(x-)-q(x)}p(v)dv+\int_{F_\nu(x)+\mu(\{x\})-q(x)}^1p(v)dv$ and either $C=0$ and then $\pi^\infty_x=\eta_x$ or $C>0$, which we now suppose. Then \begin{align*}
   &\theta:=\frac{1}{C}\int_{F_\nu(x-)-q(x)}^{F_\nu(x-)}(1-p(v))\delta_{F_\nu^{-1}(v)}dv+\int_{F_\nu(x)}^{F_\nu(x)+\mu(\{x\})-q(x)}(1-p(v))\delta_{F_\nu^{-1}(v)}dv\\\mbox{and }&\vartheta:=\frac{1}{C}\int_0^{F_\nu(x-)-q(x)}p(v)\delta_{F_\nu^{-1}(v)}dv+\int_{F_\nu(x)+\mu(\{x\})-q(x)}^1p(v)\delta_{F_\nu^{-1}(v)}dv\end{align*} are two probability measures such that $\vartheta-\theta=\frac{\mu(\{x\})}{C}(\pi^\infty_x-\eta_x)$. Since $\pi^\infty_x$ and $\eta_x$ have common mean, so do $\vartheta$ and $\theta$. Moreover, \begin{align*}\theta([F_\nu^{-1}(F_\nu(x-)-q(x))&,F_\nu^{-1}(F_\nu(x)+\mu(\{x\})-q(x))])=1\\&=\vartheta((-\infty,F_\nu^{-1}(F_\nu(x-)-q(x))]\cup[
F_\nu^{-1}(F_\nu(x)+\mu(\{x\})-q(x)),+\infty)).\end{align*}
Either $\vartheta((-\infty,F_\nu^{-1}(F_\nu(x-)-q(x))])=0$ so that $\theta\le_{st}\vartheta$ and, with the equality of means, $\theta=\vartheta$ or $\vartheta((-\infty,F_\nu^{-1}(F_\nu(x-)-q(x))])=1$  so that $\vartheta\le_{st}\theta$ and, with the equality of means, $\theta=\vartheta$ or $\vartheta((-\infty,F_\nu^{-1}(F_\nu(x-)-q(x))])\vartheta([
F_\nu^{-1}(F_\nu(x)+\mu(\{x\})-q(x)),+\infty))>0$. In all cases, $\theta\le_{cx}\vartheta$ and $\eta_x\le_{cx}\pi^\infty_x$. 
As a consequence $\frac{\mu-\mu(\{x\})\delta_x}{1-\mu(\{x\})}\le_{cx}\frac{\nu-\mu(\{x\})\pi^\infty_x}{1-\mu(\{x\})}\le_{cx}\frac{\nu-\mu(\{x\})\eta_x}{1-\mu(\{x\})}$. Let $\tilde\pi\in\Pi_M\left(\frac{\mu-\mu(\{x\})\delta_x}{1-\mu(\{x\})},\frac{\nu-\mu(\{x\})\eta_x}{1-\mu(\{x\})}\right)$. Then, $\pi:=(1-\mu(\{x\}))\tilde\pi(dy,dz)+\mu(\{x\})\delta_x(dy)\eta_x(dz)$ belongs to $\Pi_M(\mu,\nu)$ and satisfies $\pi_x=\eta_x$.

\end{proof}
\begin{proof}[Proof of Corollary \ref{cor:ttdroit}]
      Let us first deal with the bound from above. Since for each $\pi\in\Pi(\mu,\nu)$, $\mu(dx)\pi_x(\{x\})\le \mu\wedge\nu(dx)$, we have $\nu^\pi_0(dx)\le \mu\wedge\nu(dx)$ for each $\pi\in\Pi_M(\mu,\nu)$. Moreover, if $\mu\wedge\nu(\R)=1$, then $\mu=\nu$ and $\mu(dx)\delta_x(dy)$ is the only element of $\Pi_M(\mu,\nu)$. If $\mu\wedge\nu(\R)<1$, then $u_\mu\le u_\nu$ implies that $u_{\frac{\mu-\mu\wedge\nu}{1-\mu\wedge\nu(\R)}}
     \le u_{\frac{\nu-\mu\wedge\nu}{1-\mu\wedge\nu(\R)}}$ so that there exists $\hat\pi\in\Pi_M(\frac{\mu-\mu\wedge\nu}{1-\mu\wedge\nu(\R)},\frac{\nu-\mu\wedge\nu}{1-\mu\wedge\nu(\R)})$ and $\pi(dx,dy)=\mu\wedge\nu(dx)\delta_x(dy)+(1-\mu\wedge\nu(\R))\hat\pi(dx,dy)$ is such that $\pi\in\Pi_M(\mu,\nu)$ and $\nu^\pi_0=\mu\wedge\nu$.

      From the decomposition in irreducible components stated in \cite[Theorem A.4]{BeJu16}, we have $\nu^\pi_0(dy)\ge \mathds{1}_{\{u_\mu(y)=u_\nu(y)\}}\mu(dy)$ for each $\pi\in\Pi_M(\mu,\nu)$. With Proposition \ref{prop:nu0}, we deduce the bound from below for $\nu^\pi_0$. In Example \ref{ex:minopidroit}, this bound from below is not attained.

      Let us now suppose that $\mu\neq\nu$ and that $(\nu_l,\nu_r)$ is a couple of non-negative measures such that $\nu-\nu_l-\nu_r=\nu_0:=\mathds{1}_{\{u_\mu=u_\nu\}}\mu+\sum_{x\in{\cal X}_0}\left(\mu(\{x\})-p_-(x)-p_+(x)\right)\delta_x$. By Corollary \ref{cor:inccoup} $(iii)$, there is at most one non-decreasing coupling in $\Pi_M(\mu,\nu,\nu_l,\nu_r)$. We now assume that $\Pi_M(\mu,\nu,\nu_l,\nu_r)\ne \emptyset$ and check that the coupling $\pi^\uparrow$ given by Corollary \ref{cor:inccoup} $(i)$ is non-decreasing. Let $\Gamma$ be a set associated, in the sense of Definition \ref{def:NdCoupling}, to $\tilde\pi(dx,dy)=\frac{\pi^\uparrow(dx,dy)-\nu_0(dx)\delta_x(dy)}{1-\nu_0(\R)}$ which is non-decreasing. We denote by $((a_n,b_n))_{n\in N}$ the irreducible components for $(\mu,\nu)$. Since $\pi^\uparrow \left(\bigcup_{n\in N}(a_n,b_n)\times [a_n,b_n]\right)=1-\mu\left(\{x\in\R:u_\mu(x)=u_\nu(x)\}\right)$ and $$\int_{\R^2}\mathds{1}_{\bigcup_{n\in N}(a_n,b_n)\times [a_n,b_n]}(x,y)\nu_0(dx)\delta_x(dy)=\sum_{x\in{\cal X}_0}\left(\mu(\{x\})-p_-(x)-p_+(x)\right),$$ we have $\tilde \pi\left(\bigcup_{n\in N}(a_n,b_n)\times [a_n,b_n]\right)=1$ and $\tilde \Gamma=\Gamma\cap \bigcup_{n\in N}(a_n,b_n)\times [a_n,b_n]$ is such that $\tilde \pi(\tilde \Gamma)=1$. Setting $\hat\Gamma=\tilde \Gamma\cup\left\{(x,x):x\in\R\mbox{ s.t. }u_\mu(x)=u_\nu(x)\mbox{ or }x\in{\cal X}_0\right\}$, we have $\pi^\uparrow (\hat\Gamma)=1$. For $x$ in the at most countable set ${\cal X}_0$, we have $\pi^\uparrow(\{(-\infty,x)\times(x,+\infty)\}\cup\{(x,+\infty)\times(-\infty,x)\})=0$ by Proposition \ref{prop:nu0}, so that $\pi^\uparrow\left(\bigcup_{x\in{\cal X}_0}\{(-\infty,x)\times(x,+\infty)\}\cup\{(x,+\infty)\times(-\infty,x)\}\right)=0$. Hence $\Gamma^\uparrow=\hat\Gamma\cap\bigcap_{x\in{\cal X}_0}\{\{(-\infty,x]\times(-\infty,x]\}\cup\{[x,+\infty)\times[x,+\infty)\}\}$ is such that $\pi^\uparrow(\Gamma^\uparrow)=1$.

      Let $(x_-, y_-), (x_+, y_+)\in \Gamma^\uparrow$ with $y_-\leq x_-$, $y_+\leq x_+$, and $x_-< x_+$. Either $y_+=x_+$ and then $y_-<y_+$  or there exists $n_{x+}\in N$ such that $x_+\in(a_{n_{x+}},b_{n_{x+}})$ and then $y_+\in[a_{n_{x+}},b_{n_{x+}})$. In the latter case, if $u_\mu(x_-)=u_\nu(x_-)$, then $y_-=x_-\le a_{n_{x+}}\le y_+$ and if there exists $n_{x-}\in N$ such that $x_-\in(a_{n_{x-}},b_{n_{x-}})$, then either $(x_-,y_-),(x_+,y_+)\in\Gamma$ so that $y_-\le y_+$ or $x_-\in{\cal X}_0$ and $y_-=x_-$ so that, since $\Gamma^\uparrow\cap\{(x_-,+\infty)\times (-\infty,x_-)\}=\emptyset$, $y_+\ge x_-=y_-$. A similar reasoning ensures that if $(x_-, z_-), (x_+, z_+) \in \Gamma^\uparrow$ with $x_-\leq z_-$, $x_+\leq z_+$, and $x_-<x_+$, then $z_-\leq z_+$. So Definition \ref{def:NdCoupling} is satisfied by $\pi^\uparrow$ with set $\Gamma^\uparrow$.
    \end{proof}
\begin{proof}[Proof of Proposition \ref{prop:deccoupl}]
  Since, clearly, $(iii)\Rightarrow (i)$, to prove that $(i)\Leftrightarrow (ii)\Leftrightarrow (iii)$, it is enough to check that $(i)\Rightarrow(ii)$ and $(ii)\Rightarrow (iii)$, which we do now.
  
  {\bf Proof of $(i)\Rightarrow (ii)$ :} Let $\pi$ be a non-increasing coupling in $\Pi_M(\mu,\nu)$ and $\Gamma$ be a Borel subset of $\R^2$ such that $\pi(\Gamma)=1$ and conditions $(c)(d)$ in Definition \ref{def:NdCoupling} hold. Let for $x\in\R$, $\Gamma_x = \left\{y: (x,y)\in\Gamma\right\}$. We have $\mu(dx)$ a.e. $\pi_x(\Gamma_x)=1$. With the martingale property, this implies \begin{equation}
   \mu(dx) \mbox{ a.e.},\;\int_{\Gamma_x} y \pi_x (dy) = x\mbox{ so that }\pi_x(\Gamma_x\cap(-\infty,x])\wedge \pi_x(\Gamma_x\cap[x,+\infty))>0.\label{eq:2cotes}
  \end{equation}Therefore $\left\{(x,y)\in\Gamma: y\leq x\right\} \neq \emptyset
$ and $\left\{(x,z)\in\Gamma: z\geq x\right\}\neq \emptyset$. Let 
\begin{align*}
    a = \inf \left\{x\in\R: \exists\, y\leq x,\; (x,y)\in\Gamma \right\}\;\;\mbox{ and }\;\;b  = \sup \left\{x\in\R: \exists\, z\geq x,\; (x,z)\in\Gamma \right\}.
\end{align*}
Since when $x<a$, $\Gamma_x \cap(-\infty,x] = \emptyset$ and when $x>b$, $\Gamma_x \cap[x,+\infty) = \emptyset$, \eqref{eq:2cotes} implies that $$\mu((-\infty,a)\cup(b,+\infty))=0.$$
Let $(x_0,y_0)\in \Gamma$ be such that $y_0\leq x_0$. By condition $(c)$ in Definition \ref{def:NdCoupling}, for any $(x,y)\in \Gamma$ with $y\leq x<x_0$, we have $y_0\le y$ so that $y_0\le x$. Hence $y_0\leq a$, $-\infty<a$ and $\Gamma\cap\left\{(x,y): y\leq x\right\} \subset [a,+\infty) \times (-\infty,a]$. In a symmetric way, condition $(d)$ in Definition \ref{def:NdCoupling} implies that $b<+\infty$ and $\Gamma\cap\left\{(x,z): z\geq x\right\} \subset (-\infty,b] \times [b,+\infty)$. Therefore $\Gamma\subset \R\times\{(-\infty,a] \cup[b,+\infty)\}$ and
$$\nu((-\infty,a] \cup[b,+\infty))=\pi(\R\times\{(-\infty,a] \cup[b,+\infty)\})\ge \pi(\Gamma)=1.$$

{\bf Proof of $(ii)\Rightarrow (iii)$ :} By Proposition \ref{prop:variational lemma to nd coupling}, there exists a coupling $\pi^\star\in\Pi_M(\mu,\nu)$ minimizing $\int_{\R^2}|x-y|\pi(dx,dy)$ over $\pi\in\Pi_M(\mu,\nu)$ and this coupling is non-increasing. 
It remains to check that $\pi^\star$ is the only non-increasing coupling in $\Pi_M(\mu,\nu)$ and that $\pi^\star(\{(a,a)\})=\mu(\{a\})\wedge\nu(\{a\})$ and $\pi^\star(\{(b,b)\})=\mu(\{b\})\wedge\nu(\{b\})$. This is clear when $a=b$ since then $\delta_a(dx)\nu(dy)$ is the only element of $\Pi(\mu,\nu)$. We thus suppose that $a<b$. Let ${\pi^\downarrow}\in\Pi_M(\mu,\nu)$ be non-increasing, $x\in\R$ and $\psi_{\pi^\downarrow}(x)={\pi^\downarrow}((-\infty,x]\times (-\infty,a])$. Since $\psi_{\pi^\downarrow}(x)\le {\pi^\downarrow}(\R\times (-\infty,a])=F_\nu(a)$ and \begin{align*}
 0\le F_\mu(x)-\psi_{\pi^\downarrow}(x)&={\pi^\downarrow}((-\infty,x]\times \R)-{\pi^\downarrow}((-\infty,x]\times (-\infty,a])\\&={\pi^\downarrow}((-\infty,x]\times [b,+\infty))\le \nu([b,+\infty))=1-F_\nu(a), 
\end{align*} we have $\psi_{\pi^\downarrow}(x)\in [(F_\nu(a)+F_\mu(x)-1)^+,F_\nu(a)\wedge F_\mu(x)]$. Using properties $(c)(d)$ in Definition \ref{def:NdCoupling}, we  check like in the derivation of \eqref{eq:mesgauche} and \eqref{eq:mesdroite} in the uniqueness part of the proof of Theorem \ref{thm:inccoup} that \begin{align*}
  &\int_{y\in(-\infty,x]}\mathds{1}_{\{z\le a\}}{\pi^\downarrow}(dy,dz)=\int_{F_\nu(a)-\psi_{\pi^\downarrow}(x)}^{F_\nu(a)}\delta_{F_\nu^{-1}(u)}(dz)du\\
 &\int_{y\in(-\infty,x]}\mathds{1}_{\{z\ge b\}}{\pi^\downarrow}(dy,dz)=\int^1_{1+\psi_{\pi^\downarrow}(x)-F_\mu(x)}\delta_{F_\nu^{-1}(u)}(dz)du.
  \end{align*}
  With the martingale property, this ensures that
  \begin{align*}
   \int_{(-\infty,x]}y\mu(dy)&=\int_{(-\infty,x]}\int_\R z{\pi^\downarrow}_y(dz)\mu(dy)\\&=\int_{(-\infty,x]\times \R}\mathds{1}_{\{z\le a\}}z{\pi^\downarrow}(dz,dy)+\int_{(-\infty,x]\times\R}\mathds{1}_{\{z\ge b\}}z{\pi^\downarrow}(dz,dy)\\&=\int_{F_\nu(a)-\psi_{\pi^\downarrow}(x)}^{F_\nu(a)}F_\nu^{-1}(v)dv+\int^1_{1+\psi_{\pi^\downarrow}(x)-F_\mu(x)}F_\nu^{-1}(v)dv=G(\psi_{\pi^\downarrow}(x)),
  \end{align*}
  where $[(F_\nu(a)+F_\mu(x)-1)^+,F_\nu(a)\wedge F_\mu(x)]\ni u\mapsto G_x(u)=\int_{F_\nu(a)-u}^{F_\nu(a)}F_\nu^{-1}(v)dv+\int^1_{1+u-F_\mu(x)}F_\nu^{-1}(v)dv$.
The three last equalities also hold with ${\pi^\downarrow}$ and $\psi_{\pi^\downarrow}(x)$ replaced by $\pi^\star$ (the coupling minimizing $\int_{\R^2}|x-y|\pi(dx,dy)$ over $\pi\in\Pi_M(\mu,\nu)$) and $\psi_{\pi^\star}(x)=\pi^\star((-\infty,x]\times (-\infty,a])$.
For $(F_\nu(a)+F_\mu(x)-1)^+\leq u'\leq u \leq F_\nu(a)\wedge F_\mu(x)$, we have
\[
G_x(u)- G_x(u')=\int_{F_\nu(a)-u}^{F_\nu(a)-u'} F_\nu^{-1}(v)\,dv - \int_{1+u'-F_\mu(x)}^{1+u-F_\mu(x)} F_\nu^{-1}(v)\,dv  \leq (u-u')a - (u-u')b.
\]
Therefore $G_x$ is a decreasing function and the equality $G_x(\psi_{\pi^\downarrow}(x))=\int_{(-\infty,x]}y\mu(dy)=G_x(\psi_{\pi^\star}(x))$ implies that $\psi_{\pi^\downarrow}(x)=\psi_{\pi^\star}(x)$. Hence \begin{align*}
   \forall x\in\R,\;&\int_{y\in(-\infty,x]}\mathds{1}_{\{z\le a\}}{\pi^\downarrow}(dy,dz)=\int_{y\in(-\infty,x]}\mathds{1}_{\{z\le a\}}\pi^\star(dy,dz)\\&\int_{y\in(-\infty,x]}\mathds{1}_{\{z\ge b\}}{\pi^\downarrow}(dy,dz)=\int_{y\in(-\infty,x]}\mathds{1}_{\{z\ge b\}}\pi^\star(dy,dz)\\\mbox{ and }&\int_{y\in(-\infty,x]}{\pi^\downarrow}_y\mu(dy)=\int_{y\in(-\infty,x]}\pi^\star_y\mu(dy).
                                                                                                                                                                         \end{align*}We conclude that ${\pi^\downarrow}=\pi^\star$ like in the end of the uniqueness part of the proof of Theorem  \ref{thm:inccoup}.
                                                                                                                                                                          Finally, let us check that $\pi^\star(\{(a,a)\})=\mu(\{a\})\wedge\nu(\{a\})$, the equality $\pi^\star(\{(b,b)\})=\mu(\{b\})\wedge\nu(\{b\})$ being derived in a symmetric way. We have $F_\mu(a)=\mu(\{a\})$. When $\mu(\{a\})\le\nu(\{a\})$, then $G_a(\mu(\{a\}))=a\mu(\{a\})$ so that $\psi_{\pi^\star}(a)=\mu(\{a\})$. When $\mu(\{a\})>\nu(\{a\})$, then $G_a(\nu(\{a\}))\ge a\nu(\{a\})+b(\mu(\{a\})-\nu(\{a\}))>a\mu(\{a\})$ so that $\psi_{\pi^\star}(a)>\nu(\{a\})$. Since $\mu(\{a\})\mathds{1}_{\{z\le a\}}\pi^\star_a(dz)=\int_{F_\nu(a)-\psi_{\pi^\star}(a)}^{F_\nu(a)}\delta_{F_\nu^{-1}(u)}(dz)du$, we deduce that $\pi^\star(\{(a,a)\})=\mu(\{a\})\wedge\nu(\{a\})$.

\end{proof}

\subsection{Proofs of the results in Section \ref{sec:MHN_MHK} concerning optimality of $\pi^\uparrow$ and $\pi^\downarrow$ for the cost function $\vert x-y\vert^\rho$ with $\rho>1$ under the nested supports condition}

\begin{proof}[Proof of Proposition \ref{prop:M^HN with alpha_rho}]Since, by Proposition \ref{prop:deccoupl} applied with $a=\underline x$ and $b=\overline x$, there exists a unique non-decreasing (resp. non-increasing) coupling in $\Pi_M(\mu,\nu)$, it is enough to check that any coupling $\pi\in\Pi_M(\mu,\nu)$ that attains $\overline{\mathcal M}_\rho (\mu,\nu)$ (resp. $\underline{\mathcal M}_\rho (\mu,\nu)$) when $\rho\in(1,2)$ or $\underline{\mathcal M}_\rho (\mu,\nu)$ (resp. $\overline{\mathcal M}_\rho (\mu,\nu)$) when $\rho>2$ is non-decreasing (resp. non-increasing). Let $\Gamma$ be a set associated with $\pi$ by Lemma \ref{lem:varbj}. Since $\pi([\underline x, \overline x]\times\{[\underline y, \overline y] \cup [\underline z, \overline z]\})=1$, we suppose without restriction that $\Gamma\subset [\underline x, \overline x]\times\{[\underline y, \overline y] \cup [\underline z, \overline z]\}$.
Let us suppose the existence of $(x_-,y_-),(x_+,y_+)\in\Gamma$ (resp. $(x_-,y_+),(x_+,y_-)\in\Gamma$) such that $x_-<x_+$, $y_-\le x_-$, $y_+\le x_+$, $y_+<y_-$ and obtain a contradiction. By \eqref{eq:l2c}, there exists $z>x_+$ such that $(x_+,z)\in\Gamma$ (resp. $z>x_-$ such that $(x_-,z)\in\Gamma$). We have $$\underline y\le y_+<y_-\le \overline y<\underline x\le x_-<x_+\le \overline x<\underline z\le z\le \overline z.$$
  Let us define the function $f_\rho: \R\to\R$ by
\[
f_\rho(x) = \frac{z-y_-}{z-y_+}\vert x-y_+\vert^\rho+\frac{y_--y_+}{z-y_+}\vert z-x\vert^\rho-\vert x-y_-\vert^\rho.
\]
Since $ (1-\alpha_\rho)y_-+\alpha_\rho z\le(1-\alpha_\rho)\overline y+\alpha_\rho\overline z\le \underline x$, by Lemma \ref{lemma: f_rho alpha_rho} just below applied with $(y,m,z)=(y_+,y_-,z)$, $f_\rho$ is decreasing on $[\underline x,\overline x]$ when $\rho\in(1,2)$ and increasing on $[\underline x,\overline x]$ when $\rho>2$. We deduce the desired contradiction like in the proof of Proposition \ref{prop:variational lemma to nd coupling}.

In a symmetric way, the existence of $(x_-,z_-),(x_+,z_+)\in\Gamma$ (resp. $(x_-,z_+),(x_+,z_-)\in\Gamma$) such that $x_-<x_+$, $x_-\le z_-$, $x_+\le z_+$, $z_+<z_-$ implies the existence of $y< x_-$ such that $(x_-,y)\in\Gamma$ (resp. $y< x_+$ such that $(x_+,y)\in\Gamma$) and
$$\underline y\le y\le \overline y<\underline x\le x_-<x_+\le \overline x<\underline z\le z_+<z_-\le \overline z.$$
Since  $ (1-\alpha_\rho)z_++\alpha_\rho y\ge(1-\alpha_\rho)\underline z+\alpha_\rho\underline y\ge \overline x$, by Lemma \ref{lemma: f_rho alpha_rho} just below applied with $(y,m,z)=(y,z_+,z_-)$, the function $\R\ni x\mapsto \frac{z_--z_+}{z_--y}\vert x-y\vert^\rho+\frac{z_+-y}{z_--y}\vert z_--x\vert^\rho-\vert z_+-x\vert^\rho$ is increasing on $[\underline x,\overline x]$ when $\rho\in(1,2)$ and decreasing on $[\underline x,\overline x]$ when $\rho>2$, yielding again a contradiction. 
\end{proof}

\begin{lemma}\label{lemma: f_rho alpha_rho}
Let $y<m<z$ and $f_\rho:\R\to\R$ be defined by
\[
f_\rho(x) = \frac{z-m}{z-y}\vert x-y\vert^\rho+\frac{m-y}{z-y}\vert z-x\vert^\rho-\vert x-m\vert^\rho.
\]
When $\rho\in(1,2)$ (resp. $\rho>2$), $f_\rho$ is increasing (resp. decreasing) on $[y,(1-\alpha_\rho) m+\alpha_\rho y)]$ and decreasing (resp. increasing) on $[(1-\alpha_\rho) m+\alpha_\rho z, z]$, where $\alpha_\rho\in(0,\frac12)$ is defined in Lemma \ref{lemma psi_rho} just below.

On the other hand, for any $\rho\in (1,2)\cup(2,+\infty)$ and any $\alpha\in(0,\alpha_\rho)$, we may find $y<m<z$ such that $(2-\rho)f'_\rho((1-\alpha) m+\alpha y) < 0$ and $y<m<z$ such that $(2-\rho)f'_\rho((1-\alpha) m+\alpha z) > 0$. \end{lemma}
 \begin{lemma}\label{lemma psi_rho}
For any $\rho>1$, let $\psi_\rho: (0,1] \to\R$ be the function defined by
\begin{equation}\label{eq psi_rho}
\psi_\rho(\alpha)= \alpha +  \alpha^{2-\rho}(1-\alpha)^{\rho-1} +1-\rho,
\end{equation}
the following statements hold:
\begin{itemize}
    \item for all $\rho\in(1,2)\cup(2,+\infty)$, there exists a unique $\alpha_\rho\in(0,1)$ such that $\psi_\rho(\alpha_\rho)=0$,
  \item if $\alpha\in (\alpha_\rho,1]$, then $\psi_\rho(\alpha) >0$ when $\rho\in(1,2)$ and $\psi_\rho(\alpha) < 0$ when $\rho>2$,  \item if $\alpha\in (0,\alpha_\rho)$, then $\psi_\rho(\alpha) <0$ when $\rho\in(1,2)$ and $\psi_\rho(\alpha) > 0$ when $\rho>2$, 
  \item $\rho\mapsto\alpha_\rho$ (represented by Figure \ref{fig:alpha2} below) is continuously differentiable on $(1,2)\cup(2,\infty)$ and can be extended by continuity at $\rho=2$ with the value $\alpha_2$ which is the unique solution to $h(\alpha)=0$ where $(0,1)\ni\alpha\mapsto h(\alpha)= 1+(1-\alpha)\ln{\frac{\alpha}{1-\alpha}}$,
 \item the extended function $\rho\in(1,\infty)\mapsto\alpha_\rho\in(0,1]$ is increasing, 
    \item $\alpha_\rho<\frac{\rho-1}{2}$ when $\rho\in(1,2)$,
    \item $\lim_{\rho\to\infty}\alpha_\rho=\frac12$.
\end{itemize}
\end{lemma}
\begin{figure}[H]
    \scalebox{0.55}{\input{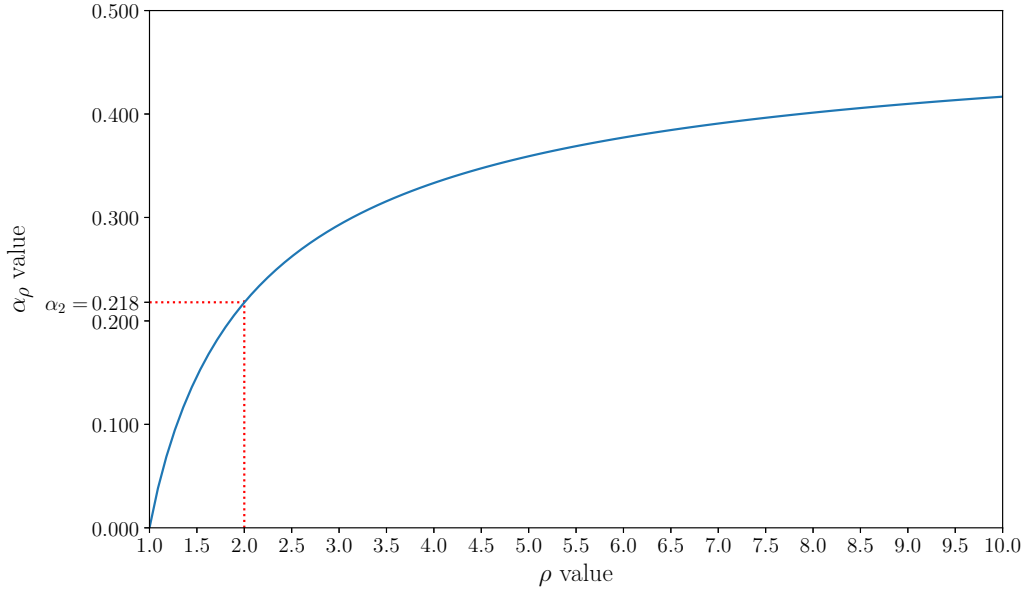}}
    \caption{The function $\rho\mapsto \alpha_\rho$ (with $\alpha_\rho$ computed by a root finding algorithm).}
    \label{fig:alpha2}
\end{figure}
The proof of Lemma \ref{lemma psi_rho} is postponed after the one of Lemma \ref{lemma: f_rho alpha_rho} which relies on the next lemma.
 \begin{lemma}\label{lemma varphi_rho}
For $\rho \in (1,+\infty)\setminus\{2\}$, let us define the function $\varphi_\rho:[0,1] \times\R_+\to\R$ by
\[
\varphi_\rho(\alpha,w) = (\alpha w)^\rho - \frac{w}{1+w}\left(1+\alpha w\right)^\rho - \frac{w^{\rho}}{1+w}\left(1-\alpha\right)^\rho.
\]
Then, we have
\begin{align*}
 &\forall \alpha\in[\alpha_\rho,1],\; \forall w>0,\quad (2-\rho)\partial_\alpha\varphi_\rho(\alpha,w)>0, \\ 
&\forall \alpha\in[0,\alpha_\rho),\mbox{ for $w$ large enough},\quad(2-\rho)\partial_\alpha\varphi_\rho(\alpha,w)<0.
\end{align*}
where $\alpha_\rho$ is defined in Lemma \ref{lemma psi_rho}.
\end{lemma}

The proof of Lemma \ref{lemma varphi_rho} is postponed after the one of Lemma \ref{lemma: f_rho alpha_rho}.

\begin{proof}[Proof of Lemma \ref{lemma: f_rho alpha_rho}] $ $\newline
When $x\in[y,m]$, we have 
\begin{align*}
    f_\rho(x) & = \frac{z-m}{z-y}(x-y)^\rho + \frac{m-y}{z-y}(z-x)^\rho - (m-x)^\rho \\
    &= (z-m)^\rho\left[\frac{z-m}{z-y}\times\left(\frac{m-y}{z-m}\right)^\rho \times\left(1-\frac{m-x}{m-y}\right)^\rho +\frac{m-y}{z-y}\times\left(1+\frac{m-x}{z-m}\right)^\rho-\left(\frac{m-x}{z-m}\right)^\rho\,\right] \\
    &=-(z-m)^\rho\,\varphi_\rho\left(\frac{m-x}{m-y},\frac{m-y}{z-m}\right),
\end{align*}
By Lemma \ref{lemma varphi_rho}, since $x\in [y,(1-\alpha_\rho) m+\alpha_\rho y)]\Leftrightarrow \frac{m-x}{m-y}\in[\alpha_\rho,1]$, we deduce that for $\rho\in(1,2)$ (resp. $\rho>2$), $f_\rho$ is increasing (resp. decreasing) on $[y,(1-\alpha_\rho) m+\alpha_\rho y)]$. Moreover, for any $\alpha\in (0,\alpha_\rho)$, when $\frac{m-y}{z-m}$ is large enough, $(2-\rho)f'_\rho((1-\alpha)m+\alpha y)<0$.

$ $\newline
When $x\in[m,z]$, we have 
\begin{align*}
    f_\rho(x) & =  \frac{z-m}{z-y}(x-y)^\rho + \frac{m-y}{z-y}(z-x)^\rho -(x-m)^\rho \\
    &= (m-y)^\rho\left[ \frac{z-m}{z-y}\times\left(1+\frac{x-m}{m-y}\right)^\rho+\frac{m-y}{z-y}\times\left(\frac{z-m}{m-y}\right)^\rho\times\left(1-\frac{x-m}{z-m} \right)^\rho-\left(\frac{x-m}{m-y}\right)^\rho \,\right]. \\
     &=- (m-y)^\rho \,\varphi_\rho\left(\frac{x-m}{z-m},\frac{z-m}{m-y}\right) .
\end{align*}
By Lemma \ref{lemma varphi_rho} again, since $x\in[(1-\alpha_\rho) m+\alpha_\rho z, z]\Leftrightarrow \frac{x-m}{z-m}\in[\alpha_\rho,1]$, we deduce that for $\rho\in(1,2)$ (resp. $\rho>2$), $f_\rho$ is decreasing (resp. increasing) on $[(1-\alpha_\rho) m+\alpha_\rho z, z]$. Moreover, for any $\alpha\in (0,\alpha_\rho)$, when $\frac{z-m}{m-y}$ is large enough, $(2-\rho)f'_\rho((1-\alpha)m+\alpha z)>0$.

\end{proof}

\begin{proof}[Proof of Lemma \ref{lemma varphi_rho}]
Taking the derivative of $\varphi_\rho(\alpha,w)$ with respect to $\alpha$, we have
$$\frac1\rho\partial_\alpha\varphi_\rho(\alpha,w) = \alpha^{\rho-1} w^\rho - \frac{w^2}{1+w}\left(1+\alpha w\right)^{\rho-1} + \frac{w^\rho}{1+w}\left(1-\alpha\right)^{\rho-1}$$
so that $\frac1\rho\partial_\alpha\varphi_\rho(0,w)=\frac{w^\rho-w^2}{1+w}$ is negative (resp. positive) when $w>1$ and $\rho\in (1,2)$ (resp. $\rho>2$). Let us now suppose that $\alpha\in (0,1)$. Then
\begin{align*}
\frac1\rho\partial_\alpha\varphi_\rho(\alpha,w) 
&= \frac{\alpha^{\rho-2}w^{\rho}}{1+w}\left(\alpha + \alpha^{2-\rho}(1-\alpha)^{\rho-1} - (\rho-1)\right) \\
&\quad + \frac{w^2}{1+w}\left((\rho-1)\left(\alpha w\right)^{\rho-2} + \left(\alpha w\right)^{\rho-1} - \left(1+\alpha w\right)^{\rho-1}\right)   \\
&=\frac{\alpha^{\rho-2}w^{\rho}}{1+w}\psi_\rho(\alpha) +\frac{w^2(\rho-1)}{1+w}\int_{\alpha w}^{1+\alpha w}\left((\alpha w)^{\rho-2} - v^{\rho-2}\right)dv
\\
&= \frac{\alpha^{\rho-2}w^{\rho}}{1+w}\left(\psi_\rho(\alpha) + (\rho-1)\int_{\alpha w}^{1+\alpha w}\left(1-\left(\frac{\alpha w}{v}\right)^{2-\rho}\right)dv\right). 
\end{align*} 
Since
\begin{align*}
0<\int_{\alpha w}^{1+\alpha w}\left(1-\left(\frac{\alpha w}{v}\right)^{2-\rho}\right)dv\le 1-\left(\frac{\alpha w}{1+\alpha w}\right)^{2-\rho}  &\quad\mbox{ when } \rho\in(1,2) ,\\
0>\int_{\alpha w}^{1+\alpha w}\left(1-\left(\frac{\alpha w}{v}\right)^{2-\rho}\right)dv\ge 1-\left(\frac{1+\alpha w}{\alpha w}\right)^{\rho-2}  &\quad\mbox{ when } \rho>2 ,
\end{align*}
by the first and second items in Lemma \ref{lemma psi_rho}, we deduce that for $\alpha\in[\alpha_\rho,1]$, we have 
$\partial_\alpha\varphi_\rho(\alpha,w)>0$ when $\rho\in(1,2)$ and $\partial_\alpha\varphi_\rho(\alpha,w)< 0$ when $\rho>2$.\\
For $\alpha\in(0,\alpha_\rho)$, since
$\lim_{w\to\infty} \int_{\alpha w}^{1+\alpha w}\left(1-\left(\frac{\alpha w}{v}\right)^{2-\rho}\right)dv=0$, by the third item in Lemma \ref{lemma psi_rho}, we may find $w$ large enough so that $\partial_\alpha\varphi_\rho(\alpha,w)<0$ when $\rho\in(1,2)$ and $\partial_\alpha\varphi_\rho(\alpha,w)>0$ when $\rho>2$.
\end{proof}

\begin{proof}[Proof of Lemma \ref{lemma psi_rho}]
$ $\newline
Taking the derivative of $\psi_\rho$, we get
\[
\psi_\rho'(\alpha) = 1+(2-\rho)\left(\frac{1-\alpha}{\alpha}\right)^{\rho-1} - (\rho-1)\left(\frac{\alpha}{1-\alpha}\right)^{2-\rho}.
\]
It appears that
\begin{itemize}
    \item when $1<\rho<2$,
    \begin{itemize}
        \item[(1)] $\psi_\rho'$ is a decreasing function, so $\psi_\rho$ is strictly concave,
        \item[(2)] $\lim_{\alpha\to0^+}\psi_\rho(\alpha) =1-\rho <0$ and $\psi_\rho(1) =2-\rho >0$.
    \end{itemize}
    \item when $\rho>2$,
    \begin{itemize}
        \item[(1)] $\psi_\rho'$ is an increasing function, and $\psi_\rho$ is strictly convex,
        \item[(2)] $\lim_{\alpha\to0^+}\psi_\rho(\alpha) =+\infty$ and $\psi_\rho(1) =2-\rho <0$.
    \end{itemize}
\end{itemize}
With the continuity of $\psi_\rho$, we deduce that there exists a unique $\alpha_\rho\in(0,1)$ such that $\psi_\rho(\alpha_\rho)=0$. Moreover, for $\alpha\in(\alpha_\rho,1]$, $\psi_\rho(\alpha)>0$ when $\rho\in (1,2)$ and $\psi_\rho(\alpha) <0$ when $\rho>2$. And for $\alpha\in (0,\alpha_\rho)$, $\psi_\rho(\alpha)<0$ when $\rho\in (1,2)$ and $\psi_\rho(\alpha) >0$ when $\rho>2$.

Since $\psi_\rho(\alpha_\rho) = 0$, we have $\left(\frac{1-\alpha_\rho}{\alpha_\rho}\right)^{\rho-2} = \frac{\rho-1-\alpha_\rho}{1-\alpha_\rho}$,  
\begin{align*}
\psi_\rho'(\alpha_\rho) &= 1+(2-\rho)\frac{\rho-1-\alpha_\rho}{\alpha_\rho} - (\rho-1)\frac{\rho-1-\alpha_\rho}{1-\alpha_\rho}
= \frac{(\rho-1)(2-\rho)}{\alpha_\rho(1-\alpha_\rho)},
\end{align*}
so that $\psi_\rho'(\alpha_\rho) > 0$ for $\rho\in(1,2)$ and $\psi_\rho'(\alpha_\rho) < 0$ for $\rho>2$. 
Furthermore, by the implicit function theorem, as the mapping $(\rho,\alpha)\in \left\{(1,2)\cup(2,\infty)\right\}\times(0,1]\mapsto\psi_\rho(\alpha)\in\R$ is $C^1$, it follows that the function $\rho\mapsto\alpha_\rho$ is also $C^1$ on $(1,2)\cup(2,\infty)$ and its derivative writes
\[\partial_\rho\alpha_\rho = \frac{1-\alpha_\rho^{2-\rho}(1-\alpha_\rho)^{\rho-1}\ln{\frac{1-\alpha_\rho}{\alpha_\rho}}}{\psi_\rho'(\alpha_\rho)} . \]
Since $\psi_\rho(\alpha_\rho) = 0$, we have 
\[
(\rho-2)\ln{\frac{1-\alpha_\rho}{\alpha_\rho}} = \ln{\frac{\rho-1-\alpha_\rho}{1-\alpha_\rho}}= (\rho-2)\int_0^1\frac{du}{1-\alpha_\rho+(\rho-2)u}.
\]
Then, since $\alpha_\rho^{2-\rho}(1-\alpha_\rho)^{\rho-1} = \rho-1-\alpha_\rho$, the numerator of the formula of $\partial_\rho\alpha_\rho$ can be written as 
\begin{align*}
1-\alpha_\rho^{2-\rho}(1-\alpha_\rho)^{\rho-1}\ln{\frac{1-\alpha_\rho}{\alpha_\rho}} &= 1-\int_0^1\frac{\rho-1-\alpha_\rho}{1-\alpha_\rho+(\rho-2)u}du.
\end{align*}
When $\rho\in(1,2)$, 
\begin{equation*}
\psi_\rho\left(\frac{\rho-1}{2} \right)=\frac{\rho-1}{2} + \frac{\rho-1}{2} \left(\frac{3-\rho}{\rho-1}\right)^{\rho-1}
+1-\rho= \frac{\rho-1}{2}\left(\left(\frac{3-\rho}{\rho-1}\right)^{\rho-1} -1\right) > 0,
\end{equation*}
so that $\alpha_\rho <\frac{\rho-1}{2}$, and
\[\int_0^1\frac{\rho-1-\alpha_\rho}{1-\alpha_\rho+(\rho-2)u}\,du < \int_0^1\frac{\rho-1-\alpha_\rho}{1-\alpha_\rho+(\rho-2)}\,du = 1.\]
When $\rho>2$,
\[\int_0^1\frac{\rho-1-\alpha_\rho}{1-\alpha_\rho+(\rho-2)u}\,du > \int_0^1\frac{\rho-1-\alpha_\rho}{1-\alpha_\rho+(\rho-2)}\,du = 1.\]
Hence, we can deduce that
\begin{itemize}
    \item $1-\alpha_\rho^{2-\rho}(1-\alpha_\rho)^{\rho-1}\ln{\frac{1-\alpha_\rho}{\alpha_\rho}} > 0$ when $\rho\in(1,2)$,
    \item $1-\alpha_\rho^{2-\rho}(1-\alpha_\rho)^{\rho-1}\ln{\frac{1-\alpha_\rho}{\alpha_\rho}} < 0$ when $\rho>2$.
\end{itemize}
We conclude that $\partial_\rho\alpha_\rho>0$ on $(1,2)\cup(2,+\infty)$.

$ $\newline
Since $h'(\alpha)=\frac1\alpha-\ln{\frac{\alpha}{1-\alpha}}$ is a decreasing function, $h$ is strictly concave. We have $\lim_{\alpha\to 0+}h(\alpha) = -\infty$ and $\lim_{\alpha\to1-}h(\alpha) = 1$. With the continuity of $h$ and $h(\frac 12)=1$, we deduce that there exists a unique $\alpha_2\in (0,\frac 12)$ such that $h(\alpha_2) = 0$. Furthermore, by the fact that $e^x \geq 1+x$ with strict inequality when $x\ne 0$, we have
\begin{align*}
\psi_\rho(\alpha) &=\alpha+(1-\alpha)e^{(2-\rho)\ln{\frac{\alpha}{1-\alpha}}}+1-\rho\\
&\geq\alpha+(1-\alpha)\left(1+(2-\rho)\ln{\frac{\alpha}{1-\alpha}}\right)+1-\rho\\
&=(2-\rho)\left(1+(1-\alpha)\ln{\frac{\alpha}{1-\alpha}}\right)=(2-\rho)h(\alpha),
\end{align*}
 with strict inequality when $\ln{\frac{\alpha}{1-\alpha}}\ne 0$. Hence $\psi_\rho(\alpha_2)> (2-\rho)h(\alpha_2)=0$. We deduce that
\begin{itemize}
    \item for $\rho \in(1,2)$, we have $\alpha_\rho < \alpha_2$ and $h(\alpha_\rho)<0$, 
    \item for $\rho \in(2,+\infty)$, we have $\alpha_\rho > \alpha_2$ and $h(\alpha_\rho)>0$.
\end{itemize}
On the other hand, applying Taylor's theorem to the term $e^{(2-\rho)\ln{\frac{\alpha}{1-\alpha}}}$, we have 
\begin{align*}
\psi_\rho(\alpha) 
&=\alpha+(1-\alpha)\left(1+(2-\rho)\ln{\frac{\alpha}{1-\alpha}}\right)+1-\rho + {\mathcal O}\left((2-\rho)^2\right)=(2-\rho)\left(h(\alpha)+{\mathcal O}(2-\rho)\right).
\end{align*}
For $\alpha\in(0,\alpha_2)$, we have $h(\alpha) < 0$ and therefore for $\rho<2$ close enough to $2$, $\psi_\rho(\alpha)<0$, which implies that $\alpha <\alpha_\rho< \alpha_2$. Hence  $\lim_{\rho\to2^-}\alpha_\rho = \alpha_2$. In the same way, we can prove that $\lim_{\rho\to2^+}\alpha_\rho = \alpha_2$.
$ $\newline

We have $\psi_\rho(\frac 12)=2-\rho$, so that, when $\rho>2$, $\psi_\rho(\frac 12)<0$ and $\alpha_\rho<\frac 12$.
On the other hand, for $\alpha<\frac12$, $\frac{1-\alpha}{\alpha}>1$ and $\lim_{\rho \to +\infty}\psi_\rho(\alpha) = +\infty$
so that for large values of $\rho$, we have $\psi_\rho(\alpha)>0$ and $\alpha<\alpha_\rho$. We deduce that
\[\lim_{\rho\to+\infty} \alpha_\rho =\frac12.\]
\end{proof}

\begin{proof}[Proof of Proposition \ref{prop:M^HK leq_cx}] 
For $\pi \in \Pi_M(\mu,\underline\nu)$, we have 
\begin{align*}
\pi_{x_-}(\{y_+\})\geq 0 = \pi_{x_-}^{\uparrow}(\{y_+\}),&\qquad \pi_{x_-}(\{z_+\})\geq 0 = \pi_{x_-}^{\uparrow}(\{z_+\}),\\
\pi_{x_-}^{\uparrow}(\{y_-\}) = \frac{\nu(\{y_-\})}{\mu(\{x_-\})}\geq \pi_{x_-}(\{y_-\}), &\qquad \pi_{x_-}^{\uparrow}(\{z_-\}) = \frac{\nu(\{z_-\})}{\mu(\{x_-\})}\geq \pi_{x_-}(\{z_-\}).
\end{align*}
Now suppose that $\pi \neq \pi^{\uparrow}$ and define
\begin{equation*}
    \hat\tau 
    = \frac{\pi_{x_-}(\{y_+\})\delta_{y_+} + \pi_{x_-}(\{z_+\})\delta_{z_+}}{\pi_{x_-}(\{y_+\}) + \pi_{x_-}(\{z_+\})},
\end{equation*}
\begin{equation*}
    \check\tau = \frac{\left(\pi_{x_-}^{\uparrow}(\{y_-\}) -\pi_{x_-}(\{y_-\})\right)\delta_{y_-} +\left(\pi_{x_-}^{\uparrow}(\{z_-\}) -\pi_{x_-}(\{z_-\})\right)\delta_{z_-}}{\pi_{x_-}^{\uparrow}(\{y_-\}) - \pi_{x_-}(\{y_-\}) + \pi_{x_-}^{\uparrow}(\{z_-\}) - \pi_{x_-}(\{z_-\})}.
\end{equation*}
Since 
\[\pi_{x_-}(\{y_-\}) +\pi_{x_-}(\{y_+\}) + \pi_{x_-}(\{z_-\}) +\pi_{x_-}(\{z_+\}) = 1 = \pi_{x_-}^{\uparrow}(\{y_-\})  + \pi_{x_-}^{\uparrow}(\{z_-\}), \]
the denominators of the formulas defining $\hat\tau$ and $\check \tau$ are equal. Let us denote the common value by $s$. Clearly $s > 0$ and $\pi_{x_-} - \pi_{x_-}^{\uparrow} = s(\hat\tau - \check\tau)$. Moreover, since $p\pi_{x_-}+(1-p)\pi_{x_+}=\nu=p\pi^{\uparrow}_{x_-}+(1-p)\pi^{\uparrow}_{x_+}$, $\pi_{x_+} - \pi_{x_+}^{\uparrow} = \frac{ps}{1-p}(\check\tau - \hat\tau)$.

Since  \[
\int y \,\pi_{x_-}(dy) = x_- = \int y \,\pi_{x_-}^{\uparrow}(dy),
\] $\hat \tau$ and $\check \tau$ have a common expectation $m$ and $m\in [y_+,z_+]\cap [y_-,z_-] = [y_+, z_-]$ and we have
\begin{align}
    \hat\tau = \frac{z_+-m}{z_+-y_+}\delta_{y_+} + \frac{m - y_+}{z_+-y_+}\delta_{z_+} , \qquad
    \check\tau = \frac{z_--m}{z_--y_-}\delta_{y_-} + \frac{m - y_-}{z_--y_-}\delta_{z_-} .\label{eq:defhattaucheck}
\end{align}
The difference between $\pi$ and $\pi^{\uparrow}$ can be expressed as $\pi - \pi^{\uparrow} = ps\, \delta_{x_-}\otimes(\hat\tau-\check\tau)+ps\delta_{x_+}\otimes(\check\tau-\hat\tau)$.  Let us denote
\begin{align*}
\underline \gamma_m &=
\frac12\left(\frac{z_--m}{z_--y_-}\delta_{(x_--y_-)^2} + \frac{m-y_-}{z_--y_-}\delta_{(z_--x_-)^2} + \frac{z_+-m}{z_+-y_+}\delta_{(x_+-y_+)^2} + \frac{m-y_+}{z_+-y_+}\delta_{(z_+-x_+)^2}\right), \\
\overline \gamma_m &=
\frac12\left(\frac{z_+-m}{z_+-y_+}\delta_{(x_--y_+)^2} + \frac{m-y_+}{z_+-y_+}\delta_{(z_+-x_-)^2} 
+ \frac{z_--m}{z_--y_-}\delta_{(x_+-y_-)^2} + \frac{m-y_-}{z_--y_-}\delta_{(z_--x_+)^2}\right).
\end{align*}
We have
\begin{equation}
   {\rm sq}\#\pi -{\rm sq}\#\pi^{\uparrow} = 2ps(\overline\gamma_m - \underline\gamma_m).\label{eq:diffimMMarr}
\end{equation}
Let us check that for any $m\in[y_+,z_-]$, $\underline\gamma_m\leq_{cx}\overline \gamma_m$. Since $\underline \gamma_m$ and $\overline \gamma_m$ linearly depend on $m$, it is equivalent to check that $\underline \gamma_{y_+}\leq_{cx}\overline \gamma_{y_+}$ and $\underline \gamma_{z_-}\leq_{cx}\overline \gamma_{z_-}$. 

Since \begin{align*}(x_--y_+)^2\vee(z_--x_+)^2&\leq(x_+-y_+)^2\wedge(x_--y_-)^2\wedge(z_--x_-)^2\\&\le(x_+-y_+)^2\vee(x_--y_-)^2\vee(z_--x_-)^2\le (x_+-y_-)^2,\end{align*}
we have
\begin{align*}
   \underline \gamma_{y_+} &= \frac12\left(\delta_{(x_+-y_+)^2} + \frac{z_--y_+}{z_--y_-}\delta_{(x_--y_-)^2} + \frac{y_+-y_-}{z_--y_-}\delta_{(z_--x_-)^2}\right)\\&\le_{cx}\frac12\left(\delta_{(x_--y_+)^2} + \frac{z_--y_+}{z_--y_-}\delta_{(x_+-y_-)^2} + \frac{y_+-y_-}{z_--y_-}\delta_{(z_--x_+)^2} \right)=\overline \gamma_{y_+}.
\end{align*}
In a symmetric way, 
since \begin{align*}(x_--y_+)^2\vee(z_--x_+)^2&\leq(x_+-y_+)^2\wedge (z_+-x_+)^2\wedge(z_--x_-)^2\\&\le(x_+-y_+)^2\vee(z_+-x_+)^2\vee(z_--x_-)^2\le (z_+-x_-)^2,\end{align*} we have $\underline \gamma_{z_-} \leq_{cx} \overline \gamma_{z_-}$. Combining the two cases together, we conclude that for any $m\in[y_+,z_-]$, $\underline\gamma_m\leq_{cx}\overline \gamma_m$, so that, by \eqref{eq:diffimMMarr}, $\vert y-x\vert^2\# \pi^{\uparrow}(dx,dy) \leq_{cx} \vert y-x\vert^2\# \pi(dx,dy)$.

For $\pi\in\Pi_M(\mu,\overline\nu)\setminus\{\pi^{\downarrow}\}$,
we check in the same way that $$\pi-\pi^{\downarrow}=ps\, \delta_{x_-}\otimes(\check\tau-\hat\tau)+ps\delta_{x_+}\otimes(\hat\tau-\check\tau)$$ with $\hat\tau$ and $\check\tau$ satisfying \eqref{eq:defhattaucheck} for some $m\in[y_+,z_-]$ and
$$s=\pi_{x_-}^{\downarrow}(\{y_+\}) - \pi_{x_-}(\{y_+\}) + \pi_{x_-}^{\downarrow}(\{z_+\}) - \pi_{x_-}(\{z_+\})=\pi_{x_-}(\{y_-\})+\pi_{x_-}(\{z_-\})>0.$$
As a consequence, ${\rm sq}\#\pi -{\rm sq}\#\pi^{\downarrow} = 2ps(\underline\gamma_m- \overline\gamma_m)$ so that ${\rm sq}\#\pi\le_{cx}{\rm sq}\#\pi^{\downarrow}$.
\end{proof}

\begin{proof}[Proof of Proposition \ref{prop:M*HN}]
  Reasoning like in the beginning of the proof of Proposition \ref{prop:M^HK leq_cx}, we check that for $\pi \in \Pi_M(\mu,\overline\nu)\setminus\{\overline \pi^\star\}$,
$\pi = \overline \pi^\star + p s \delta_{x_-}\otimes\left(\hat\tau - \check\tau\right) + p s \delta_{x_+}\otimes\left(\check\tau - \hat\tau\right)$ where 
  \begin{align}\label{eq:tau}
    \hat\tau &= \frac{z_--m}{z_--y_+}\delta_{y_+} + \frac{m - y_+}{z_--y_+}\delta_{z_-},\;\;\;\;\;\; 
               \check\tau = \frac{z_+-m}{z_+-y_-}\delta_{y_-} + \frac{m - y_-}{z_+-y_-}\delta_{z_+}\\
    s&=\pi_{x_-}(\{y_+\}) + \pi_{x_-}(\{z_-\})=\overline \pi_{x_-}^\star(\{y_-\}) - \pi_{x_-}(\{y_-\}) + \overline \pi_{x_-}^\star(\{z_+\}) - \pi_{x_-}(\{z_+\})>0.\notag\end{align}

Therefore the difference between the costs of $\pi$ and $\overline \pi^\star$ can be calculated as  
\begin{equation}\label{productM}
\begin{split}
    &\int \vert x-y\vert^\rho \pi(dx,dy) - \int \vert x-y\vert^\rho \overline \pi^\star(dx,dy) 
     \quad = ps\left(f(x_-,m) - f(x_+,m)\right) ,
\end{split}
\end{equation}
where for $x\in[y_+,z_-]$,
\begin{align*}
f(x,m) = \frac{z_--m}{z_--y_+}(x-y_+)^\rho + \frac{m-y_+}{z_--y_+}(z_--x)^\rho - \frac{z_+-m}{z_+-y_-}(x-y_-)^\rho - \frac{m-y_-}{z_+-y_-}(z_+-x)^\rho .
\end{align*}
In the same way, we can check that any $\pi \in \Pi_M(\mu,\underline \nu)$ distinct from $\underline \pi^\star$ can be expressed as 
\begin{align*}
    \pi &= \underline \pi^\star + p s \delta_{x_-}\otimes\left(\check\tau- \hat\tau \right) + p s \delta_{x_+}\otimes\left(\hat\tau-\check\tau  \right),
\end{align*}
for $\hat\tau$ and $\check\tau$ satisfying \eqref{eq:tau} for some $m\in[y_+,z_-]$ and
\[s = \pi_{x_-}(\{y_-\})+\pi_{x_-}(\{z_+\})=\underline \pi_{x_-}^\star(\{y_+\}) - \pi_{x_-}(\{y_+\}) + \underline \pi_{x_-}^\star(\{z_-\}) - \pi_{x_-}(\{z_-\})>0.\] 
Hence the difference between the costs of $\pi$ and $\underline \pi^\star$ can be calculated as  
\begin{equation*}
\int \vert x-y\vert^\rho \pi(dx,dy) - \int \vert x-y\vert^\rho \underline \pi^\star(dx,dy) 
    = ps(f(x_+,m) - f(x_-,m)).\end{equation*}
 Checking the optimality of $\overline \pi^\star$ and $\underline \pi^\star$ is equivalent to check when $1<\rho<2$ (resp. $\rho > 2$ ) that for every $m\in [y_+,z_-]$, $f(x_-,m)<f(x_+,m)$ (resp. $f(x_-,m)>f(x_+,m)$).

The function $m\mapsto f(x,m)$ being affine, it is also equivalent to show $f(x_-, y_+) < f(x_+, y_+)$ and $f(x_-, z_-) < f(x_+, z_-)$ for $1<\rho<2$ (resp.  $f(x_-, y_+) > f(x_+, y_+)$ and $f(x_-, z_-) > f(x_+, z_-)$ for $\rho > 2$), which we are now going to demonstrate. We have 
\begin{align*}
f(x,y_+) &= \left( x-y_+\right)^\rho - \frac{z_+-y_+}{z_+-y_-}\left( x-y_-\right)^\rho - \frac{y_+-y_-}{z_+-y_-}\left(z_+-x\right)^\rho , \\ 
\frac1\rho\partial_x f(x,y_+) &= \left( x-y_+\right)^{\rho-1} - \frac{z_+-y_+}{z_+-y_-}\left( x-y_-\right)^{\rho-1} + \frac{y_+-y_-}{z_+-y_-}\left(z_+-x\right)^{\rho-1},\\f(x,z_-) &= \left(z_--x \right)^\rho - \frac{z_+-z_-}{z_+-y_-}\left( x-y_-\right)^\rho - \frac{z_--y_-}{z_+-y_-}\left( z_+-x\right)^\rho,\\ 
\frac1\rho\partial_x f(x,z_-) &= -\left(z_--x \right)^{\rho-1} - \frac{z_+-z_-}{z_+-y_-}\left( x-y_-\right)^{\rho-1} + \frac{z_--y_-}{z_+-y_-}\left( z_+-x\right)^{\rho-1} .
\end{align*}

We are going to check that $(2-\rho)\partial_x f(z_-,z_-)>0$ and $(2-\rho)\partial_x f(z_-,y_+)>0$, so that by continuity of $x\mapsto (\partial_x f(x,z_-),\partial_x f(x,y_+))$, there exists $x_\rho \in (y_+,z_-)$ such that for all $x \in [x_\rho,z_-]$, $(2-\rho)\partial_x f(x,z_-)>0$, $(2-\rho)\partial_x f(x,y_+)>0$ and for all $x_-,x_+$ such that $x_\rho<x_-<x_+<z_-$, $(2-\rho)\left(f(x_+,z_-)-f(x_-,z_-)\right)>0$ and $(2-\rho)\left(f(x_+,y_+)-f(x_-,y_+)\right)>0$.

On the one hand, we have 
\begin{align*}
\frac1\rho\partial_x f(z_-,z_-) &= - \frac{z_+-z_-}{z_+-y_-}\left( z_--y_-\right)^{\rho-1} + \frac{z_--y_-}{z_+-y_-}\left( z_+-z_-\right)^{\rho-1} \\
&=\frac{z_+-z_-}{z_+-y_-}\left( z_--y_-\right)^{\rho-1}\left(\left(\frac{z_--y_-}{z_+-z_-}\right)^{2-\rho} - 1\right).
\end{align*}
Since, by assumption $z_--y_->z_+-z_-$, we deduce that $\partial_x f(z_-,z_-) > 0$ when $1<\rho<2$ and $\partial_x f(z_-,z_-) < 0$ when $\rho>2$. On the other hand, 
\begin{align*}
\frac{(z_+-y_-)}{\rho}\partial_x f(z_-,y_+) &= (z_+-y_-)(z_--y_+)^{\rho-1} - (z_+-y_+)(z_--y_-)^{\rho-1} + (y_+-y_-)(z_+-z_-)^{\rho-1} \\ 
                                            &= g\left(y_+-y_-, z_--y_+, z_+-z_-\right),\end{align*}
                                          where, for $a,b,c\ge 0$, 
$$g(a,b,c)= (a+b+c)\,b^{\,\rho-1}-(b+c)(a+b)^{\,\rho-1}+ac^{\,\rho-1}.$$
When $\rho\in (1,2)$ and $b>0$, by concavity of $\R_+\ni x\mapsto x^{\rho-1}$, we have
\begin{align*}
   g(a,b,c)&\ge (a+b+c)\,b^{\,\rho-1}-(b+c)(b^{\,\rho-1}+(\rho-1)b^{\rho-2}a)+ac^{\,\rho-1}\\&=(2-\rho) a b^{\rho-1}+ac\left(c^{\rho-2}-((\rho -1)^{\frac{1}{\rho-2}}b)^{\rho-2}\right).
\end{align*}
and the inequality is reversed when $\rho>2$ by convexity of $\R_+\ni x\mapsto x^{\rho-1}$. We deduce that
\begin{equation}
   \label{signg}\forall a,b>0,\;\forall c\in[0,(\rho-1)^{\frac{1}{\rho-2}}b],\;g(a,b,c)>0\mbox{ when $\rho\in(1,2)$ and }g(a,b,c)<0\mbox{ when }\rho>2. 
\end{equation}
Since $y_+-y_->0$, $z_--y_+>0$ and $z_--y_+\ge (\rho-1)^{\frac{1}{2-\rho}}(z_+-z_-)$, we have $\partial_x f(z_-, y_+) > 0$ when $\rho\in (1,2)$ and  $\partial_x f(z_-, y_+)< 0$ when $\rho>2$.
\end{proof}


\bibliographystyle{plain}
\bibliography{ndmcoupling.bib}

\begin{thebibliography}{10}

\bibitem{AlCoJo17a}
A.~Alfonsi, J.~Corbetta, and B.~Jourdain.
\newblock Sampling of one-dimensional probability measures in the convex order
  and computation of robust option price bounds.
\newblock {\em International Journal of Theoretical and Applied Finance},
  22(3), 2019.

\bibitem{BeHePe12}
M.~Beiglb{\"o}ck, P.~{Henry-Labord{\`e}re}, and F.~Penkner.
\newblock Model-independent bounds for option prices: A mass transport
  approach.
\newblock {\em Finance and Stochastics}, 17(3):477--501, 2013.

\bibitem{BeJu16}
M.~Beiglb{\"o}ck and N.~Juillet.
\newblock On a problem of optimal transport under marginal martingale
  constraints.
\newblock {\em Annals of Probability}, 44(1):42--106, 2016.

\bibitem{diamond2016cvxpy}
S.~Diamond and S.~Boyd.
\newblock {CVXPY}: {A} {P}ython-embedded modeling language for convex
  optimization.
\newblock {\em Journal of Machine Learning Research}, 17(83):1--5, 2016.

\bibitem{HeTo13}
P.~{Henry-Labord{\`e}re} and N.~{Touzi}.
\newblock {An explicit martingale version of the one-dimensional Brenier
  theorem}.
\newblock {\em Finance and Stochastics}, 20(3):635--668, 2016.

\bibitem{HoKl}
D.~Hobson and M.~Klimmek.
\newblock Robust price bounds for the forward starting straddle.
\newblock {\em Finance and Stochastics}, 19(1):189--214, 2015.

\bibitem{HobsonNeuberger}
D.~Hobson and A.~Neuberger.
\newblock Robust {Bounds} for {Forward} {Start} {Options}.
\newblock {\em Mathematical Finance}, 22(1):31--56, 2012.

\bibitem{JoMa18}
B.~Jourdain and W.~Margheriti.
\newblock A new family of one dimensional martingale couplings.
\newblock {\em Electronic Journal of Probability}, 25, 2020.

\bibitem{JoMa22}
B.~Jourdain and W.~Margheriti.
\newblock Martingale {W}asserstein inequality for probability measures in the
  convex order.
\newblock {\em Bernoulli}, 28(2):830--858, 2022.

\bibitem{kertzrosler}
R.~P. Kertz and U.~R\"{o}sler.
\newblock Complete lattices of probability measures with applications to
  martingale theory.
\newblock In {\em Game theory, optimal stopping, probability and statistics},
  volume~35 of {\em IMS Lecture Notes Monogr. Ser.}, pages 153--177. Inst.
  Math. Statist., Beachwood, OH, 2000.

\bibitem{marmi}
M.~Marcus and V.~J. Mizel.
\newblock Absolute continuity on tracks and mappings of {S}obolev spaces.
\newblock {\em Arch. Rational Mech. Anal.}, 45:294--320, 1972.

\bibitem{Revuz-Yor}
D.~Revuz and M.~Yor.
\newblock {\em Continuous {Martingales} and {Brownian} {Motion}}.
\newblock Grundlehren der mathematischen {Wissenschaften}. Springer-Verlag,
  Berlin Heidelberg, 3 edition, 1999.

\bibitem{ShakedShanthikumar}
M.~Shaked and G.~Shanthikumar.
\newblock {\em Stochastic {Orders}}.
\newblock Springer {Series} in {Statistics}. Springer-Verlag, New York, 2007.

\bibitem{St65}
V.~Strassen.
\newblock The existence of probability measures with given marginals.
\newblock {\em Annals of Mathematical Statistics}, 36(2):423--439, 1965.

\end{thebibliography}
\end{document}